\documentclass{article}

\usepackage{amsfonts,amssymb, latexsym, amsmath, amsthm,mathrsfs, verbatim,geometry}

\usepackage{slashed}
\usepackage{amsfonts}
\usepackage{url,color}
\usepackage{enumerate}
\usepackage{esint}
\usepackage{pifont}
\usepackage{upgreek}
\usepackage{times}
\usepackage{calligra}
\usepackage{graphicx}
\usepackage{caption}
\usepackage{amscd}
\usepackage{amsthm}
\usepackage{amsmath,bm}
\usepackage{amssymb}
\usepackage{comment}
\usepackage{tikz}
\usepackage[utf8]{inputenc}
\usepackage{gensymb}
\usepackage{pifont}
\usepackage{autobreak}
\usepackage{color}
\usepackage{verbatim}
\usepackage{systeme}
\usepackage[font=scriptsize]{caption}
\usepackage[ngerman, english]{babel}
\usepackage[title,toc,page]{appendix}
\usepackage{subfig}
\usepackage{tikz}
\usetikzlibrary{calc}
\usetikzlibrary{calc}
\usetikzlibrary{arrows.meta}
\usetikzlibrary{hobby}
\usepackage{tkz-euclide}
\newcommand{\Jm}{\mathfrak J}
\newcommand{\gaSix}{\gamma_\star}
\newcommand{\gaOne}{\gamma_1}
\newcommand{\Barrier}{\mathfrak B}

\newtheorem{lem}{Lemma B.\ignorespaces}

\newcommand{\R}{\mathbb{R}}
\newcommand{\ga}{\gamma}
\newcommand{\la}{\lambda}

\newcommand{\al}{\alpha}

\numberwithin{equation}{section}

\usepackage{listings}

\allowdisplaybreaks

\usepackage{hyperref}
\hypersetup{
    colorlinks,
    citecolor=red,
    filecolor=green,
    linkcolor=blue,
    urlcolor=black
}
\definecolor{backcolour}{rgb}{0.95,0.95,0.92}
\lstdefinestyle{mystyle}{
    backgroundcolor=\color{backcolour},   
}

\usepackage{tabularx}
\usepackage{multirow}
\usepackage{blindtext} % Package to generate dummy text
\usepackage{charter} % Use the Charter font
\usepackage{microtype} % Slightly tweak font spacing for aesthetics
\usepackage[english, ngerman]{babel} % Language hyphenation and typographical rules
\usepackage{amsthm, amsmath, amssymb} % Mathematical typesetting
\usepackage{float} % Improved interface for floating objects
\usepackage{graphicx, multicol} % Enhanced support for graphics
\usepackage{xcolor} % Driver-independent color extensions
\usepackage{marvosym, wasysym} % More symbols
\usepackage{rotating} % Rotation tools
\usepackage{censor} % Facilities for controlling restricted text
\usepackage{pseudocode} % Environment for specifying algorithms in a natural way
\usepackage{booktabs} % Enhances quality of tables
\usepackage{csquotes} % Context sensitive quotation facilities
\newtheorem{theorem}{Theorem}[section]

\newtheorem{lemma}[theorem]{Lemma}
\newtheorem{remark}[theorem]{Remark}
\newtheorem{definition}{Definition}[section]
\newtheorem{proposition}[theorem]{Proposition}

\title{On self-similar converging shock waves}

\author{Juhi Jang\thanks{Department of Mathematics, University of Southern California, Los Angeles, CA 90089, USA, and Korea Institute for Advanced Study, Seoul, Korea.  Email: juhijang@usc.edu.}, \  Jiaqi Liu\thanks{Department of Mathematics, University of Southern California, Los Angeles, CA 90089, USA.  Email: jiaqil@usc.edu.} \ and Matthew Schrecker\thanks{Department of Mathematics, University of Bath, Claverton Down, Bath, UK. Email: mris21@bath.ac.uk.}}

\date{}

\begin{document}

%-------------------------------
%	TITLE SECTION
%-------------------------------

\maketitle

\abstract{In this paper, we rigorously prove the existence of self-similar converging shock wave solutions for the non-isentropic Euler equations for $\gamma\in (1,3]$. These solutions are analytic away from the shock interface before collapse, and the shock wave reaches the origin at the time of collapse. The region behind the shock undergoes a sonic degeneracy, which causes numerous difficulties for smoothness of the flow and the analytic construction of the solution. The proof is based on 
continuity arguments, nonlinear invariances, and barrier functions.}

\tableofcontents

\section{Introduction}\label{sec:intro}

The converging 
shock wave problem is a classical hydrodynamical problem in gas dynamics, where a spherical shock originates from  infinity or a large radius (for example, by a spherical piston) in a spherically symmetric medium and propagates towards the center of symmetry, becoming stronger as it approaches the origin. In finite time, the spherical shock collapses at the center. The problem was first discussed by Guderley in his seminal work \cite{Guderley42}  (see also Landau \cite{Landau87} and Stanyukovich \cite{Stanyukovich}).  Due to a wide range of applications such as detonation, laser fusion and chemical reactions, the theory of converging shocks has attracted a lot of attention in the mathematics and physics communities over several decades \cite{Axford81, Courant48, Guderley42, Jenssen18, Landau87, Lazarus81, Ponchaut06, Sed, Welsh67, ZRY67}, and  is still an active area of research \cite{Giron23, Jenssen23, Ramsey12}. In addition, imploding shock waves  are frequently used as a test problem in scientific computing and algorithms for compressible flows \cite{Giron23, Ramsey12}. A rigorous analysis, therefore, is not only of mathematical interest but also of practical importance as it lays out foundational evidence in support of these applications. 

It has been long known since Guderley that for an inviscid perfect gas, only a particular choice of similarity exponent would lead to a converging self-similar radially symmetric shock wave. Despite many works \cite{Courant48, Guderley42, Landau87, Lazarus81} regarding the numerical value of such a similarity exponent and the corresponding self-similar solutions based on phase portraits and numerics, a rigorous construction of self-similar converging shock wave solutions that are smooth away from the shock interface has remained elusive. In this paper, we give a rigorous construction of self-similar converging shock wave solutions described by the non-isentropic compressible Euler equations for an ideal perfect gas.

The Euler system for compressible gas flows in radial symmetry is given by the system of PDEs 
\begin{equation}\label{symmetric equations}
	\begin{aligned}
    \rho_t+\frac{1}{r^m}(r^m\rho u)_r&=0,\\
    (\rho u)_t +\frac{1}{r^m}\big(r^m(\rho u^2)\big)_r + p_r &=0,\\
    \Big[\rho \big(e +\cfrac{u^2}{2}\big)\Big]_t + \frac{1}{r^m}\Big[r^m\rho u\big( e +\cfrac{u^2}{2}+\cfrac{p}{\rho} \big)\Big]_r &=0,    
\end{aligned}
\end{equation}
where $\rho = \rho(t,r)\geq 0$ is the density, $u=u(r,t)$ is the radial fluid velocity,  
$p(t,r)\geq 0$ is the pressure, and $e(t,r)$ is the specific internal energy. Here $(t,r)\in \mathbb R \times \mathbb R_+$ and $m=1,2$ distinguishes flows with cylindrical or spherical symmetry. The equations in \eqref{symmetric equations} 
stand for the conservation of mass, 
momentum, and 
energy respectively. 
We consider an ideal perfect gas whose equation of state is given by 
\begin{align}
    p = (\gamma-1)\rho e = (\gamma-1)c_v\rho \theta,\label{pressure}
\end{align}
where $\gamma>1$ and $c_v$ are positive constants. The specific entropy $S$ is related to
\begin{align}
	p\rho^{-\gamma} = \text{Constant}\cdot \text{exp}(\frac{S}{c_v}).\label{entropy}
\end{align}
 By the conservation laws \eqref{symmetric equations}, the entropy $S$ remains constant along particle trajectories in smooth regions of the flow:
\begin{align}
	S_t+uS_r=0. \label{conservation law of entropy}
\end{align}
The sound speed is given by 
\begin{align}
    c = \sqrt{\cfrac{\gamma p}{\rho}}.\label{sound speed}
\end{align}
By taking $u$, $\rho$ and $c$ to be the main unknowns, the system \eqref{symmetric equations} takes the form away from vacuum 
\begin{align}
    \rho_t+(\rho u)_r+\cfrac{m\rho u}{r} &=0,\label{rho_t}\\
     u_t +uu_r+\frac{1}{\gamma\rho}(\rho c^2)_r &=0,
     \label{u_t}\\
     c_t+uc_r+\frac{\gamma-1}{2}c(u_r+\frac{mu}{r}) &=0.\label{c_t}
\end{align}

The system \eqref{rho_t}--\eqref{c_t} admits a three-parameter family of invariant scalings: the scaling transformation 
\begin{align}\label{scaling symmetry}
	\rho(t,r) \to \nu^\kappa \rho(\frac{t}{\nu^{\lambda}},\frac{r}{\nu}), \quad 
	u(t,r) \to \nu^{1-\lambda} u(\frac{t}{\nu^{\lambda}},\frac{r}{\nu}), \quad 
	c(t,r)\to \nu^{1-\lambda}  c(\frac{t}{\nu^{\lambda}},\frac{r}{\nu}), 
\end{align}
for $\nu>0, \ \lambda>0, \ \kappa\in \mathbb R$, 
leaves the system invariant. 
This scaling symmetry is intimately connected to the existence of self-similar solutions. Self-similarity is an important concept in hydrodynamics due to its universal nature and the possibility that self-similar solutions are  attractors for different physical phenomena in fluid and gas dynamics  \cite{Sedov, ZRY67}. In the physics literature  \cite{ZRY67}, two kinds of self-similar solutions have been discussed:  Type I if all self-similar parameters are completely determined from a dimensional analysis and Type II otherwise. Converging self-similar shock waves emerge as Type II solutions as the speed of collapse, which is a free parameter, is determined only \textit{a posteriori} through the regularity requirement of solutions. To analyze the converging shock wave problem, inspired by the scaling symmetry \eqref{scaling symmetry}, we introduce the similarity variable\footnote{This is consistent with some of the literature, for instance by Morawetz \cite{Morawetz51} and Lazarus \cite{Lazarus81}, while other authors use the equivalent similarity variable $y = {r}/{|t|^{\frac{1}{\lambda}}}$ (see \cite{Biasi21, Bilbao96, BCG22, Chisnell98, Merle22a}).} 
\begin{align}
    x = \frac{t}{r^{\lambda}}\label{similarity variable},
\end{align}
and the ansatz
\begin{align}
    u(t,r) & = -\frac{r}{\lambda t}V(x) = -\frac{r^{1-\lambda}}{\lambda}\frac{V(x)}{x},\label{ansatz for velocity}\\
    c(t,r) & = -\frac{r}{\lambda t}C(x) = -\frac{r^{1-\lambda}}{\lambda}\frac{C(x)}{x},\label{ansatz for sound speed}\\
    \rho(t,r) & = r^{\kappa}R(x)\label{ansatz for density}
\end{align}
where $\lambda>1$ and $\kappa$ are free parameters. This self-similar ansatz applied to \eqref{conservation law of entropy} in any region where the flow is smooth leads to an algebraic relation between $V$, $C$ and $R$: 
\begin{align}
	R(x)^{q+1-\gamma}(\cfrac{C(x)}{x})^2|1+V(x)|^q \equiv \text{constant},\label{entropy conserved gives RVC}
\end{align} where $q = \tfrac{2(\lambda-1)}{m+1}$. 
Therefore by plugging \eqref{ansatz for velocity}--\eqref{ansatz for density} to the Euler system  \eqref{rho_t}--\eqref{c_t} and using \eqref{entropy conserved gives RVC}, we obtain the system of ODEs for two unknowns $V(x)$, $C(x)$: 
\begin{equation}\label{ODE system}
	\begin{aligned}
		    \cfrac{d V}{d x} &= -\frac{1}{\lambda x}\frac{G(V(x),C(x);\gamma,z)}{D(V(x),C(x))} \ \ \text{ and } \ \ 
     \cfrac{d C}{d x} = -\frac{1}{\lambda x}\frac{F(V(x),C(x);\gamma,z)}{D(V(x),C(x))},
	\end{aligned}
\end{equation}
where
\begin{align}
   D(V,C) &=(1+V)^2 - C^2, \label{D(V,C)}\\
   G(V,C;\gamma,z) &= C^2[(m+1)V+2mz]-V(1+V)(\lambda+V)\label{G(V,C)},\\
   F(V,C;\gamma,z) &= C\big\{C^2[1+\frac{mz}{(1+V)}]- a_1(1+V)^2+a_2(1+V)-a_3\big\}\label{F(V,C)},
\end{align}
and
\begin{equation}\label{a_1234&z}
	\begin{aligned}
	z = \cfrac{\lambda -1 }{m\gamma}, \ \ \ a_1 = 1+\cfrac{m(\gamma-1)}{2}, \ \ \ a_2 =\cfrac{m(\gamma-1)+m z\gamma (\gamma-3)}{2}, \ \ \ a_3 = \cfrac{m z\gamma (\gamma-1)}{2} .
	\end{aligned}
\end{equation}
The derivation of the ODE system is standard and we have adopted the notation used by Lazarus \cite{Lazarus81}.

We seek a solution for which the shock converges towards the origin for $t<0$ along a self-similar path 
which is described by a constant value of the similarity variable $x$,
\begin{align}
	x \equiv -1 \quad \text{so that}\quad r_{shock} = (-t)^{\frac{1}{\lambda}}, \ \  t<0, \label{converging shock}
\end{align}
and  the shock reaches the origin at $t=0$.
Moreover, the flows on either side of the shock are assumed to be similarity flows with the same values of $\gamma$, $\lambda$, and $\kappa$ in \eqref{ansatz for velocity}-\eqref{ansatz for density}.
Under this assumption, we still require that the jump in the similarity variables is consistent with the standard Rankine-Hugoniot jump conditions across the shock.  
Let the subscript 0 and 1 denote evaluation immediately ahead of and behind the shock. The Rankine-Hugoniot conditions and Lax entropy condition, reformulated in the self-similar variables, are 
	\begin{align}
	1+V_1 & = \frac{\gamma-1}{\gamma+1}(1+V_0) + \frac{2C_0^2}{(\gamma+1)(1+V_0)},\nonumber\\
	C_1^2 &= C_0^2 + \frac{\gamma-1}{2}[(1+V_0^2)-(1+V_1)^2],\label{jump condition}\\
	R_1(1+V_1) &= R_0(1+V_0),\nonumber\\
	C_0^2&<(1+V_0)^2.\label{subsoundspeed}
	\end{align}
We assume that the fluid ahead of the shock is at rest and at a constant density and pressure. Then, by \eqref{ansatz for density}, we have 
	$\kappa = 0$ and $R(x)$ is a constant. For convenience, we let
\begin{align}
	R(x)\equiv 1\quad \text{for}\quad -\infty<x<-1. \label{R pre shock}
\end{align}
Also, by \eqref{sound speed}, the sound speed $c$ is also constant ahead of the shock. As we assume $\lambda>1$, it implies that $C$ must vanish identically there. By the assumption that the fluid is at rest before the shock,
\eqref{ansatz for velocity} implies that $V$ also must vanish identically there. Therefore, we have
\begin{align}
	V(x) = C(x) \equiv 0\qquad \text{for} \quad -\infty<x<-1 \label{V and C pre shock}
\end{align}
so that $(V_0,C_0,R_0)=(0,0,1)$. Obviously, \eqref{subsoundspeed} is satisfied. Then, applying \eqref{jump condition}, we get
	
\begin{align}
		V_1 &= -\frac{2}{\gamma+1},\label{V after shock}\\
	C_1 &= \frac{\sqrt{2\gamma(\gamma-1)}}{\gamma+1},\label{C after shock}\\
	R_1 & = \frac{\gamma+1}{\gamma-1}.\label{R after shock}
\end{align}
 As we are interested in solutions such that $u$, $c$ and $\rho$ are well-behaved at any location $r>0$ at $t=0$, we seek solutions  
 such that 
\begin{align}\label{behavior at collapse}
	 u(0,r) = -\frac{r^{1-\lambda}}{\lambda}\lim_{x\rightarrow 0}\frac{V(x)}{x}<\infty, \quad 
    c(0,r)  =-\frac{r^{1-\lambda}}{\lambda}\lim_{x\rightarrow 0}\frac{C(x)}{x}<\infty, \quad
\rho(0,r)  = R(0).
\end{align}
In particular, we require 
\begin{align}
	V(0)=C(0)=0\label{V and C at t,r=0}.
\end{align}
The converging shock wave problem is, for given adiabatic index $\gamma$, to find a smooth solution to \eqref{ODE system} for $-1<x<0$   connecting the shock interface represented by $(V_1,C_1)$ at $x=-1$ to the ultimate collapsed state $(0,0)$ at $x=0$. Together with the pre-shock state \eqref{V and C pre shock}, such a piecewise smooth solution to \eqref{ODE system} gives rise to a collapsing shock solution to the Euler system  \eqref{rho_t}--\eqref{c_t}. 

A key difficulty in solving the collapsing shock wave problem is that singularities of the dynamical system \eqref{ODE system} may occur when $D=0$ or $x=0$. The moveable singularity $D=0$ is associated with the so-called sonic singularity (the condition $D=0$ means exactly that the fluid speed and sound speed coincide), while the singularity at $x=0$ is a removable singularity which is due to the symmetry assumption. For smooth solutions, if $D=0$ at some point $x=x_{sonic}$, $G$ and $F$ must vanish at $x=x_{sonic}$. For our problem, $D(V_1,C_1)<0$ (\textit{cf.}  \eqref{D(V1,C1)<0}) and $D(0,0)=1>0$ and hence any smooth solution must pass through a sonic point ($D=0$) at which $G=F=0$. This triple vanishing property is not satisfied by generic values of $\lambda$, but it is expected that there exists a particular value of $\la$ allowing smooth passage through the sonic point. The main result of this paper is the existence (and, for a certain range of $\ga$, uniqueness) of this $\la$ which yields a converging shock wave solution.

\begin{theorem}[Informal statement]\label{thm-informal}
 (i) Let $\gamma\in(1,3]$.  Then there exists a collapsing shock solution to the non-isentropic Euler equations  \eqref{rho_t}-\eqref{c_t}.   
 
(ii) Moreover, suppose $\gamma\in(1,\frac53]$. Then there is a unique blow-up speed $\la$ such that the aforementioned solution exists.
\end{theorem}

The precise statement of Theorem \ref{thm-informal} will be given in Theorem \ref{main theorem} after we discuss the basic structure of the phase portrait plane associated with the ODE system \eqref{ODE system} and introduce the set of important parameters appearing in our analysis in Section \ref{sec:2}. We remark that self-similar collapsing shock waves are solutions of unbounded amplitude (\textit{cf.}  \eqref{behavior at collapse}) and their continuation to expanding shock solutions are genuine weak solutions to the Euler system  \eqref{rho_t}--\eqref{c_t}, as shown by Jenssen-Tsikkou \cite{Jenssen18}. Before moving forward, we mention some works on compressible Euler flows with a focus on weak solutions and singularities.

The study of the compressible Euler equations has a long history and a correspondingly vast literature, much of it focused on the one-dimensional problem. As is well known, a fundamental difficulty 
in the analysis of the compressible Euler equations stems from the expected formation of singularities in the solutions, a phenomenon known since the time of Riemann and Stokes. For a survey of the literature on the 1D Euler equations, including existence of weak solutions and formation of singularities, we refer to \cite{Courant48, Dafermos16, Landau87} and the references therein. 

Although there is no general theory for the existence of weak solutions for the multi-dimensional problem, in recent years, the existence of weak entropy solutions for the isentropic system under the assumption of spherical symmetry has been established in \cite{Chen15, Chen18, Schrecker20} using the vanishing viscosity method from artificial viscosity solutions of certain auxiliary problems. This has been extended to cover more physical, density-dependent viscosities in \cite{Chen22}. The weak solutions constructed in these works are based on a finite energy method that allows for discontinuous and unbounded solutions to arise, especially at the origin. Earlier results, \cite{Chen97, Makino92, Makino94}, gave existence results on gases in an exterior region surrounding a solid ball, and relied on boundedness of solutions. 

The formation of singularities in the multi-dimensional compressible Euler equations was first rigorously established in  
\cite{Sideris85}. 
To better  understand the structure of the singularities, there has been much interest in the study of shock formation in solutions of the multi-dimensional compressible Euler equations. The first rigorous results are those in spherical symmetry of \cite{Yin04}, which studies the formation and development of shocks in spherical symmetry for perturbations of constant data for the non-isentropic system. The  work  
\cite{Christodoulou07} on shock formation for irrotational, isentropic, relativistic gases gives a truly multi-dimensional result and sharp understanding of the geometry of the solution at the blow-up time (see also \cite{Christodoulou16, ChrisMiao14}).  
In recent years, there have been further exciting  
developments on  shock formation 
to allow for non-trivial vorticity and entropy and to remove symmetry assumptions \cite{Abbrescia22, Buckmaster23, Luk18} while still showing the finite time formation of a singularity with sharp asymptotic behavior on approach to the blowup. 
Moreover, in \cite{Buckmaster22b} the authors  
have established the local-in-time continuation of a shock solution from the first blow-up time for the full, non-isentropic Euler equations; see also a recent work \cite{SV23} for the maximal development problem. 

As well as these shock solutions, other kinds of strong singularity  
have also been areas of active interest,  
especially the implosion solutions of Merle--Rapha\"el--Rodnianski--Szeftel, constructed in \cite{Merle22a} and whose finite-codimension stability is established in \cite{Merle22b}. 
These solutions of the isentropic Euler equations with $\ga$-law pressure (excluding a countable set of $\ga\in(1,\infty)$) have been constructed using a self-similar ODE analysis, and the authors must also handle the presence of triple points in the phase plane (the sonic points), through which the solutions must pass smoothly.  
The existence of these solutions has been extended to cover a wider range of $\ga$ in \cite{BCG22} and to allow for non-radial perturbations in  \cite{CGSS23}. Following these works, the construction of continuous (but not necessarily smooth) implosion solutions to the non-isentropic Euler equations has been achieved in \cite{Jenssen23} using a combination of analytic and numerical techniques. This result also discusses the continuation of the blowup solution past the first blowup time with an expanding shock wave solution.

\section{Basic structure of phase portrait and main result}\label{sec:2}

In this section, we discuss the basic structure of the phase portrait of the ODE system  \eqref{ODE system} and the main result of the paper along with the methodology. 
In our analysis, we will primarily make use of the following ODE associated with the system  
\eqref{ODE system}
\begin{align}
	\cfrac{dC}{dV} = \cfrac{F(V,C;\gamma,z)}{G(V,C;\gamma,z)},\label{ODE}
\end{align} 
which makes  the phase portrait analysis more accessible in the $(V,C)$ plane.

We denote the initial data point by
\begin{align}\label{initial}
	P_1=(V_1,C_1)
\end{align}
in the $(V,C)$ plane.

\begin{lemma}\label{C1V1}
For  $\gamma\in(1,3]$, the initial data points $V_1(\ga)$ and $C_1(\ga)$ given in \eqref{V after shock} and \eqref{C after shock} are monotone increasing with respect to $\gamma$.
\end{lemma}
\begin{proof}
The result follows from direct computation:
\begin{align*}
    V_1'(\gamma) = \cfrac{2}{(\gamma+1)^2}>0, \quad    C_1'(\gamma) = \cfrac{\gamma}{(\gamma+1)^2\sqrt{2\gamma(\gamma-1)}}>0.
\end{align*}
\end{proof}

\subsection{Roots of $F$, $G$ and $D$}

In this subsection, we summarize the critical points of the dynamical system \eqref{ODE system} and some fundamental monotonicity properties with respect to the parameters $z$ and $\gamma$. 

\

\noindent\textbf{Triple points} $F=G=D=0$.
The triple points at which $F=G=D=0$ are crucial to understanding the dynamics of solutions to the ODE system \eqref{ODE system}. On the one hand, at these points, generic trajectories will suffer a loss of regularity. On the other hand, at least one such point must be passed through for a trajectory to reach from the initial data $P_1$ to the origin.
\begin{lemma}[\cite{Lazarus81}]\label{triple roots}
	The solutions to $F=G=D=0$ are
	\begin{align}
		P_2 &=(-1,0),\\
		P_6 &= (V_6,C_6) = (\frac{-1+(\gamma-2)z-w}{2}, 1+V_6),\label{P6}\\
		P_7 &=(V_7,C_7) = (\frac{-1+(\gamma-2)z-w}{2},-1-V_7),\\
		P_8 &=(V_8,C_8) = (\frac{-1+(\gamma-2)z+w}{2}, 1+V_8),\label{P8}\\
		P_9 &=(V_9,C_9) = (\frac{-1+(\gamma-2)z+w}{2}, -1-V_9),
	\end{align}
where
\begin{align}
	w(z) = +\sqrt{1-2(\gamma+2)z+(\gamma-2)^2z^2}\label{w(z)}.
\end{align}
\end{lemma}
\begin{remark}
	Since $w\geq 0$, we will always have $V_8 \geq V_6$ and $C_8\geq C_6$.
\end{remark}
\begin{remark}\label{zM}
	\eqref{V after shock} and \eqref{C after shock} imply $C_1>1+V_1$ immediately behind the shock, while the condition  \eqref{V and C at t,r=0} implies $C(0)<1+V(0)$. Since we require that $u$ and $c$ are all well behaved at any location away from the origin, the trajectory must at least continuously pass through the line $D(V,C)=0$ at some $x_0\in(-1,0)$. Comparing this with the ODE system \eqref{ODE system}, we see that we must have $F(x_0)=G(x_0)=0$ to ensure continuity. Thus, the trajectory can only pass through the sonic line $D =0$ at $P_6$ or $P_8$. As a consequence, $w(z)$ given by \eqref{w(z)} must be a real number, which gives us the constraint
	\begin{align}
		    z \leq z_M(\gamma) = (\sqrt{\gamma} +\sqrt{2})^{-2}\label{z_{max}}.
	\end{align}
	Recall that $z = \frac{\lambda-1}{m\gamma}$ from \eqref{a_1234&z}. That is, equivalently, we must have
\begin{align}
	\lambda\leq\lambda_M = m\gamma z_M+1 = \frac{m\gamma}{(\sqrt{\gamma} +\sqrt{2})^2}+1\label{lambda_{max}}.
\end{align}
We will henceforth restrict the range of parameters $\lambda$ and $z$ to $(1,\lambda_M]$ and $(0,z_M]$, respectively, and will use  both $z$ and $\lambda$ as convenient.
\end{remark}

\begin{remark}\label{bounds of w(z)}
	For $z\in(0,z_M]$, the function $w(z)$ defined by \eqref{w(z)} is a decreasing function in $z$ and satisfies
	\begin{align*}
		0\leq w(z)<1,
	\end{align*}
	where $w(z_M)=0$ and $w(z)\rightarrow 1$ when $z\rightarrow 0$.
\end{remark}

The following lemma establishes the monotonicity properties of 
the locations of $P_6$ and $P_8$ with respect to $z$.

\begin{lemma}\label{Lemma V_6/8 monoton}
For any $\gamma\in (1,3]$ , $V_6$ and $C_6$ are strictly increasing, and $V_8$ and $C_8$ are strictly decreasing with respect to  $z\nearrow z_M$. 
\end{lemma}

\begin{proof} For any fixed $\gamma\in (1,3]$, we write $V_6 = V_6(z)$ and so $V_6'(z)= \frac{dV_6}{dz}$.  We compute 
\begin{align*}
2V_6'(z)= \gamma-2+\cfrac{ (\gamma+2)-(\gamma-2)^2z   
   }{\sqrt{1-2(\gamma+2)z+(\gamma-2)^2z^2}}.
\end{align*}
For any  $z\in(0, z_M]$, we have $w(z)=\sqrt{1-2(\gamma+2)z+(\gamma-2)^2z^2}< 1$ and $(\gamma-2)^2z<1<\gamma+2$. Hence $$2V_6'(z) \geq \gamma-2+ \big(\gamma +2 -(\gamma-2)^2z\big)=2\gamma -(\gamma-2)^2z>0.$$
Arguing similarly for $V_8$,
we have
\begin{align*}
    2V_8'(z) 
    &= \gamma-2-\cfrac{ (\gamma+2)-(\gamma-2)^2z   
   }{\sqrt{1-2(\gamma+2)z+(\gamma-2)^2z^2}}\\
   &< \gamma-2-\big(\gamma+2-(\gamma-2)^2z\big) <0.
\end{align*}
Since $C_6 = 1+V_6$ and $C_8 = 1+V_8$, the desired results follow.
\end{proof}

From the definitions of $P_6$ and $P_8$ in \eqref{P6} and \eqref{P8} and $z_M$ in \eqref{z_{max}}, we have 
\begin{equation}\label{V68(z_M) C68(z_M)}
	\begin{aligned}
		&V_6(z_M) = V_8(z_M) = \frac{-\sqrt{2}}{\sqrt{\gamma}+\sqrt{2}}, \quad 
		C_6(z_M) = C_8(z_M) = \frac{\sqrt{\gamma}}{\sqrt{\gamma}+\sqrt{2}}.
	\end{aligned}
\end{equation}
Therefore, by Lemma \ref{Lemma V_6/8 monoton}, we have that 
\begin{equation}\label{upper and lower bound for V/C_6/8}
	\begin{aligned}
	-1\leq V_6&\leq \frac{-\sqrt{2}}{\sqrt{\gamma}+\sqrt{2}} \leq V_8\leq 0, \quad 
	0\leq C_6\leq \frac{\sqrt{\gamma}}{\sqrt{\gamma}+\sqrt{2}} \leq C_8\leq 1.
\end{aligned}
\end{equation}

\begin{figure}
    \centering
    \includegraphics[width=0.6\textwidth]{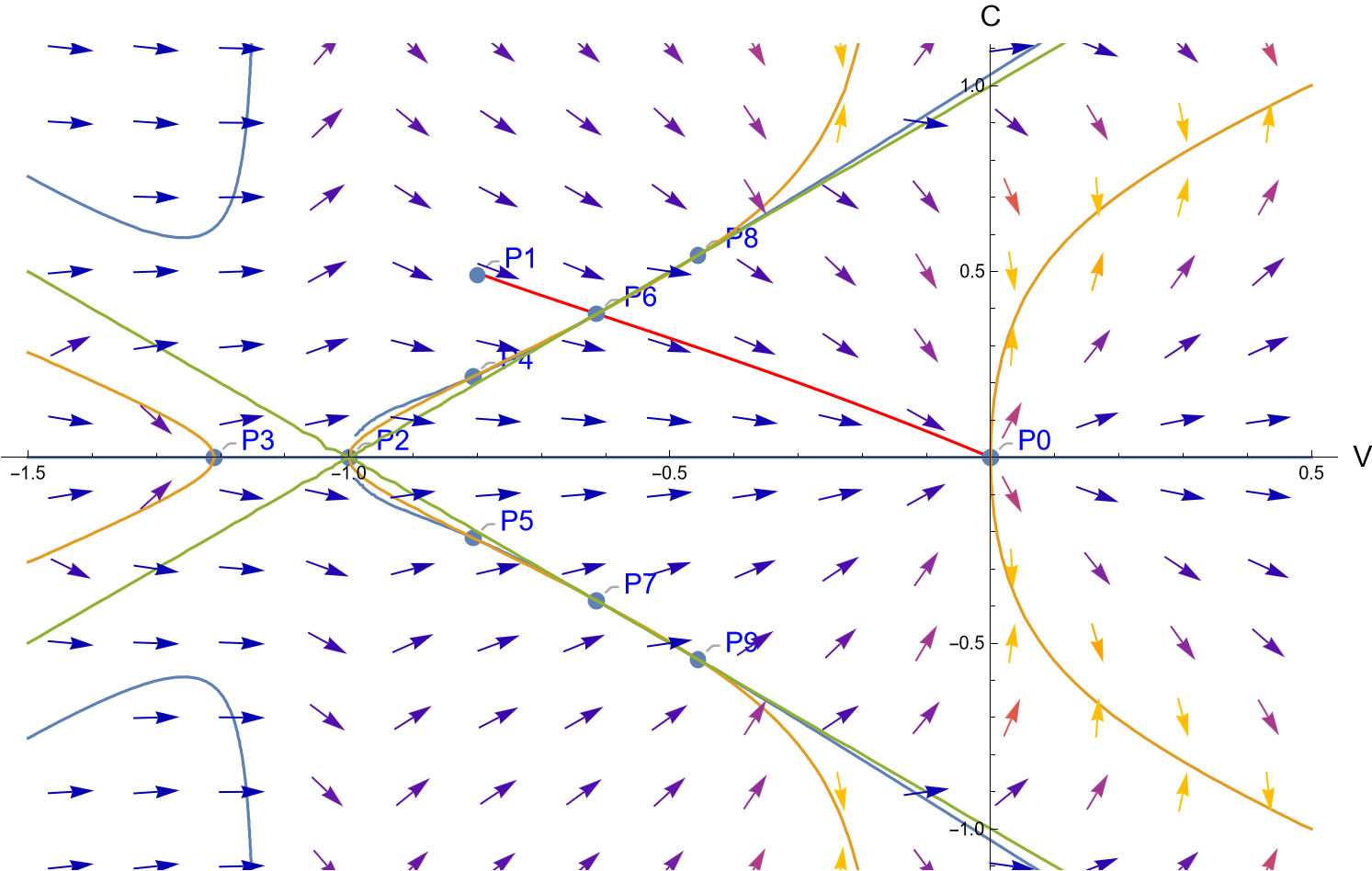}
    \caption{$m=1$, $\gamma = 1.5$, $z = 0.14$, green line: $D=0$, orange line: $G=0$, blue line: $F=0$, red line: solution trajectory}
    \label{fig:9 critical points}
\end{figure}

\noindent\textbf{Double roots} $F=G=0$. In addition to the triple points, there are also a number of stationary points of the ODE system \eqref{ODE system} at which $F=G=0$ but $D\neq0$.  To simplify notation, we define
\begin{align*}
	H(V) = \sqrt{\cfrac{V(1+V)(\lambda+V)}{(m+1)V+2mz}}.
\end{align*}
Then, the double points of the system may be directly  
computed as in the following lemma (\textit{cf.} \cite{Lazarus81}).

\begin{lemma}\label{double root}
	The solutions to $F=G=0$ and $D\neq 0$ are
	\begin{align}
		P_0 & = (0,0),\\
		P_3 &=(V_3,C_3) = (-\lambda, 0),\\
		P_4&=(V_4,C_4) = (\frac{-2\lambda}{\gamma+1+m(\gamma-1)}, H(V_4)),\label{P4}\\
		P_5&=(V_5,C_5) = (\frac{-2\lambda}{\gamma+1+m(\gamma-1)}, -H(V_5)).
	\end{align}
\end{lemma}

\begin{remark}
Since the solution $C$ of  \eqref{ODE system} must remain positive before the collapse $t<0$ in order to be physically meaningful, the points $P_2$, $P_3$, $P_5$,  $P_7$ and $P_9$ do not play a role in the construction of the solution before the collapse. 
\end{remark}

We observe that $D$, $F$, and $G$ at $P_1$ satisfy the following sign conditions:
	\begin{align}
		D(V_1,C_1)&<0,\label{D(V1,C1)<0}\\
		F(V_1,C_1) &>0,\\
		G(V_1,C_1) &<0.
	\end{align}
	Further details on the signs of $D$, $F$, and $G$ can be found in  \cite{Lazarus81}. 
	By \eqref{D(V1,C1)<0}, $P_1$ will be always located above the sonic line $D(V,C)=0$ as in  Figure 1.

\subsection{Main result and methodology}\label{subsec:mainresult}

Many authors have claimed that, for each $\gamma\in(1,3]$, there exists a $\lambda_{std}\text{ or }z_{std}$ such that the corresponding trajectory exists from $P_1$ to the origin $P_0$, analytically passes through the triple point $P_6$ or $P_8$ and is monotone decreasing to the origin, therefore describing a collapsing shock solution of the compressible Euler equations (see, for example, \cite{Courant48, Guderley42, Jenssen18, Landau87, Lazarus81}). The goal of this paper is to prove rigorously the existence of such a $z_{std}$ and the corresponding analytic solution to \eqref{ODE}. 

The self-similar solutions that we  construct are built by concatenating two trajectories in the phase-plane in such a way that we obtain an analytic solution of the ODE \eqref{ODE}.
\begin{itemize}
	\item The first trajectory connects $P_1$ to either $P_6$ or $P_8$ in the 2nd quadrant of the $(V,C)$-plane. To ensure the trajectory passes through $P_6$ or $P_8$ analytically, we need the trajectory to enter the triple point $P_6$ or $P_8$ with a specific slope.
	\item The second trajectory connects either $P_6$ or $P_8$ to the origin $P_0=(0,0)$, which is a stable node for \eqref{ODE system}. Since the first trajectory passes through $P_6$ or $P_8$ analytically, this second one is uniquely determined by the slope at $P_6$ or $P_8$.
\end{itemize}
Directly solving the initial value problem for  \eqref{ODE} poses 
complexity due to the non-linearity of $F(V,C;\gamma,z)$ and $G(V,C;\gamma,z)$ and the two parameters $\gamma$ and $z$. One significant challenge in this problem lies in the non-trivial nature of solutions around the triple points $P_6$ and $P_8$, which can be entered either along a primary or a secondary direction by solutions of \eqref{ODE}. Along the dominant, primary direction, the solutions will be only of finite regularity, and so we require the solutions to connect along the secondary direction to ensure analyticity. This property of analytic connection fails for generic choices of the parameter $z$, and so the isolated value, $z_{std}$, that enables this analytic connection must be carefully constructed. Moreover, it emerges that, for some ranges of $\gamma\in(1,3]$, the solution emanating from the initial condition $P_1$ will converge to $P_6$, while for other $\ga$, it will converge to $P_8$. We must therefore understand which of the triple points the solution from $P_1$ should connect to in order to identify $z_{std}$ and establish an analytic connection.

 To address these challenges effectively, we employ barrier functions for a number of purposes (\textit{cf.}  Definition \ref{def:barrier}).  First, to exclude connection from $P_1$ to $P_8$ (respectively $P_6$) for small (respectively large) values of $\ga$. Second, to establish an appropriate interval of candidate values of $z$ containing $z_{std}$. Employing the barrier function $B_{k_M}(V)=\sqrt{\frac{C^2_6(z_M)}{V_6(z_M)}V}$, we exclude connections to $P_8$ for small $\gamma$, and we will also exclude connection to $P_6$ for $\ga\geq 2$. In addition, the function $B_1(V)=\sqrt{-V}$ is essential for establishing connection from $P_1$ to $P_8$ for intermediate values of $\ga$. This motivates the following definitions.
 \begin{itemize}
 	\item $\gaSix$ is defined to be the value such that $P_1$ lies on the curve $C=B_{k_M}(V)$ (see  \eqref{gaSix}). $\gaSix \approx 1.7$.
	\item $\gaOne$ is defined to be the value such that $P_1$ lies on the curve $C=B_1(V)$ (see  \eqref{gaOne}). $\gaOne = 1+\sqrt{2}$.
 \end{itemize}
 
 As mentioned above, we also need to limit the window of possible $z$ values for which we may have an analytic connection from $P_1$ to either $P_6$ or $P_8$. This leads us to the following definitions of key values of $z$.
 
\begin{itemize}
	\item For any $\gamma\in(1,3]$, $z_M(\gamma)$ is defined to be the value such that $V_6(z_M)=V_8(z_M)$, which means $P_6$ and $P_8$ are coincident ({see} \eqref{z_{max}}).
	\item For any $\ga\in(1,3]$, $z_m(\ga)$ is defined to be the value such that $V_6(z_m)=V_1$, so that $P_6$ lies on the vertical line through $P_1$ ({see} \eqref{zm}).
	\item For any $\gamma\in(1,2]$, $z_g(\gamma)$ is defined to be the value such that $V_4=V_6$, which means $P_4$ and $P_6$ are coincident ({see} \eqref{zg}).
	\item For any $\gamma\in(\gaSix,\gaOne]$, $z_1(\gamma)$ is defined to be the value such that the curve $C=B_1(V)$ intersects the sonic line at $P_8$ (see  \eqref{z1}).
	\item For any $\gamma\in(\gaOne,3]$, $z_2(\gamma)$ is defined to be the value such that the curve $C=\sqrt{-\frac{3}{2}V}$ intersects the sonic line at $P_8$ (see \eqref{z2}).
\end{itemize}

\smallskip

We now state the main result of the paper. 

\begin{theorem} \label{main theorem} (i) For all $\gamma\in (1, 3]$, there is a monotone decreasing analytic solution to \eqref{ODE} connecting $P_1$ to the origin. 

(ii) For $\gamma \in (1, \gaSix]$, the solution is unique (in the sense of unique $z$) and it connects $P_1$ to the origin via $P_6$. The value of $z_{std}$ lies in $(z_g, z_M ]$. 

(iii) For $\gamma \in (\gaSix , 2)$, if such a solution connects through $P_6$, then $z_{std} \in (z_g,z_M]$ and $z_{std}$ gives the only such connection through $P_6$. If a solution connects through $P_8$, then $z_{std} \in (z_1,z_M]$. 

(iv) For $\gamma \in [2, 3]$, any such solution must connect through $P_8$ with  a $z_{std}$ value 
$z_{std} \in (z_1,z_M]$ if $\gamma \in  [2,\gaOne]$ or $z_{std} \in (z_2,z_M]$ if $\gamma \in (\gaOne,3]$. 
\end{theorem}

\smallskip 

\begin{remark}
To see that these solutions from Theorem \ref{main theorem} do indeed give solutions of the original self-similar problem (in the $x$ variable) is straightforward. Note that, given $C(V)$, one can solve for $V(x)$ (and hence $C(x)$) via the ODE $V'(x)=-\frac{1}{\lambda x}\frac{G(V(x),C(V(x)))}{D(V(x),C(V(x)))}$ simply by integrating, away from critical points. As $C(V)$ is an analytic function in $V$ and $G$ and $D$ are both analytic in $(V,C)$ with simple zeros at the triple points $P_6$ and $P_8$, repeated application of the chain rule establishes that the solution $V(x)$ remains smooth (indeed, analytic), as it passes through a unique sonic point $x_{\textup{sonic}}$ where $(V(x_{\textup{sonic}}),C(x_{\textup{sonic}}))$ is either $P_6$ or $P_8$. As $G(0,0)=0$, it is clear that if, for some $x_0\in(x_{\textup{sonic}},0)$, we hit $V(x_0)=C(x_0)=0$, then the ODE has a local, unique solution, which is the identically zero solution. But this extends backwards for all $x$, contradicting the initial data and the sonic time. So the solution cannot hit zero except at $x=0$. 
\end{remark}

\smallskip

Our strategy for proving the existence of the solutions constructed in Theorem \ref{main theorem} proceeds in three key stages, inspired by recent mathematical constructions of self-similar gravitational collapse \cite{GHJ21,GHJ23,GHJS22}, where the authors developed the shooting methods for self-similar non-autonomous ODE systems to connect smoothly  two behaviors at the center and at the far field through the sonic point.  
First, in Section \ref{sec:analytic0}, we construct local, analytic solutions around each of the triple points. That is, for all $z\in[z_m,z_M]$, we construct a local solution around $P_6$ and we construct a local solution around $P_8$ for all $z\in(0,z_M]$. In order to show the local existence of such solutions, we first choose a local branch 
at the triple points along the secondary direction of \eqref{ODE} with a negative slope $c_1<0$ (\textit{cf.}  Section \ref{local behavior}) and   
derive a formal recurrence relation for the Taylor coefficients of a power series $C(V;\ga,z)=\sum_{k=0}^\infty c_k(\ga,z)(V-V_*(\ga,z))^k$. 
Once we have found the recurrence relation for the higher order coefficients, a series of combinatorial estimates and an inductive argument allow us to bound coefficients to all orders and establish the convergence of the series in Theorem \ref{analytic}.

The second main step of the proof is to show the existence, for each $\ga\in(1,3]$, of a $z_{std}$ such that the local analytic solution from either $P_6$ or $P_8$, extended backwards in $V$, connects to $P_1$. This is achieved in Section \ref{sec:existence} via a continuity argument. We show first that the solution from $P_6$ for $z=z_m$ always passes below $P_1$ in the phase space, while there always exists a $z\in(0,z_M)$ such that the solution from $P_8$ passes above $P_1$. Then, depending on whether the solution for $z=z_M$ passes above or below $P_1$, we may apply a continuity argument to either $P_6$ or $P_8$ to establish the connection.

The third main step in the construction is to prove that the solution connecting $P_1$ smoothly to either $P_6$ or $P_8$ then continues to connect to the origin. In fact, the behavior of connecting to the origin is not limited only to the solution that connects to $P_1$, but holds for a non-trivial interval of $z$ around $z_{std}$, as the origin is an attractive point in the phase plane. 
A key difficulty is that the solution must connect from inside the second quadrant, else the velocity changes sign before collapse. We cannot, a priori, exclude the possibility that, for some range of $z$, the solution passes through the $C$-axis for some positive value of $C$ before converging to the origin from the first quadrant. To show that this does not occur, we apply careful barrier arguments to gain an upper bound on the solution which traps it into a region in the second quadrant in which it must converge to the origin. This notion is made precise in the following definition.

\begin{definition}[Lower barrier function and upper barrier function]\label{def:barrier}  We say that a differentiable function $ B(V)$ is a lower barrier for $C(V)$ on $(V_a, V_b)$ if $B(V) < C(V)$ on $(V_a, V_b)$, and a upper barrier if  $B(V) > C(V)$ on $(V_a, V_b)$. 
\end{definition}

In practice, $C(V)$ will be the solution of \eqref{ODE} and $B(V)$ is a specific differentiable function where we design $B$ such that at one end point, $C(V)$ is greater or less than $B(V)$, and show that the solution $C(V)$ stays above or below $B(V)$ as $V$ moves to the other end point. The latter part will be achieved by nonlinear invariances of \eqref{ODE}. Suppose we intend to show that $B$ is a lower barrier for $C$ and that $C(V_a)>B(V_a)$ (respectively $C(V_b)>B(V_b)$). We assume for a contradiction that there exists $\overline{V}\in(V_a,V_b)$ such that $C(\overline{V}) = B(\overline{V})$. By simple continuity and compactness arguments, there exists a minimal (respectively maximal) such $\overline{V}$, from which we deduce that, at $\overline{V}$, we must have 
$$\frac{d}{dV}(C-B)\big|_{\overline{V}}\leq 0, \ \text{ respectively } \ \frac{d}{dV}(C-B)\big|_{\overline{V}}\geq 0.$$ 
To derive a contradiction, we therefore prove that, whenever $C(\overline{V})=B(\overline{V})$, then we must have 
\begin{equation}\label{barrier_argument}
\frac{d}{dV}(C-B)\big|_{\overline{V}}> 0, \ \text{ respectively } \ \frac{d}{dV}(C-B)\big|_{\overline{V}}< 0. 
\end{equation}

\

As the self-similar blowup speed $z$ varies, the associated solutions from the triple points $P_6$ and $P_8$ efficiently explore a large portion of the phase space, with the solutions from $P_8$ in particular moving far up in the phase plane. In order, therefore, to apply the precise barrier arguments that will force the solution to the right of the triple point to converge to the origin, we in fact require better control on the range of $z$ (depending on $\ga$) for which the solution to the left connects to $P_1$, else we lose effective control on the trajectory to the right and cannot exclude the possibility that the trajectory passes through $V=0$ away from the origin. This improved control on $z$ also allows us to make more quantitative and qualitative statements concerning the behavior of the imploding shock solution, especially for $\ga\in(1,\ga_\star]$.

 To this end, we first limit the range of $\ga$ for which the connecting solution may come from $P_6$ or $P_8$. This is achieved in Sections \ref{sec:P6only}--\ref{sec:P8only}, in which we employ our first barrier arguments to the left in order to show that for $\ga\in(1,\gamma_\star]$, the solution must connect to $P_6$, and for $\ga\in[2,3]$, it must connect to $P_8$.

Following this, in Section \ref{sec:P6toleft}, we improve the range of $z$ for which the solution from $P_6$ (given $\ga\in(1,2]$) may connect to $P_1$, tightening the range $z\in[z_m,z_M]$ to the much sharper $z\in(z_g,z_M]$ by showing that the trajectory is bounded from above, for this range of $z$, by the solution to a simpler ODE that allows for explicit integration and estimation. This  improvement ensuring $z_{std}> z_g$ is essential, as the structure of the phase portrait changes fundamentally as $P_4$ crosses $P_6$ at $z=z_g$. 
We are then able also to show in Lemma \ref{uniquenessP6} that, for $\ga\in(1,2]$, there is at most one value of $z\in(z_g,z_M]$ for which the solution from $P_6$ may connect to $P_1$ by studying the derivative $\frac{\partial}{\partial z}(\frac{dC}{dV})$.  

The next section, Section \ref{sec:P8toleft}, contains the analogous sharpening of the possible range of $z$ for solutions from $P_8$. In it, we show that, for $\ga\in(\ga_\star,\ga_1]$, solutions with $z\in(0,z_1]$ cannot connect to $P_1$ by employing the barrier $B_1(V)=\sqrt{-V}$, while for $\ga\in(\ga_1,3]$, solutions  with $z\in(0,z_2]$ cannot connect to $P_1$ by employing the barrier $B_{\frac32}(V)=\sqrt{-\frac32V}$  (\textit{cf.} the definitions of $z_1$ and $z_2$ above).

Having established these tighter ranges of $z$, depending on $\ga$, for the existence of the imploding shock solution, in Section \ref{sec:solntoright} we are then able to prove that the solution must connect to the origin within the second quadrant. A simple proof in Lemma \ref{lemma:Cneq0} shows that the trajectories can never hit the $V$-axis, and so it suffices to find upper barriers connecting to the origin. Indeed, for $\ga\in(1,2]$, we show that the solutions from $P_6$ for all $z\in[z_g,z_M]$ admit $B_1(V)=\sqrt{-V}$ as an upper barrier, and the solutions from $P_8$ for any $\ga\in(\ga_\star,\ga_1]$ and $z\in(z_1,z_M]$ admit the same upper barrier. Finally, for the remaining range, $\ga\in(\ga_1,3]$ and $z\in(z_2,z_M]$, the barrier $B_{\frac32}(V)=\sqrt{-\frac32 V}$ is an upper barrier for the solution to the right.

Finally, in Section \ref{sec:mainthmproof}, we put together the earlier results in order to establish the proof of the main theorem.

\section{Local smooth solutions around sonic points}\label{sec:analytic0}

In this section, we show the existence of local analytic solutions around the triple point $P_*=P_6 \text{ or } P_8$: 
\[
C(V) = C_*+ \sum_{\ell =1}^{\infty}c_{\ell}(V-V_*)^{\ell},
\]    
where the Taylor coefficients $c_\ell=c_\ell(\ga,z)$ and with a choice of branch having a negative slope $c_1<0$. The first step is to show that it is always possible to choose a branch with $c_1<0$ for the admissible range $z\in (0, z_M]$ at $P_8$ and $z\in [z_m,z_M]$ at $P_6$ (see Section \ref{local behavior}). The second step is to derive a recursive formula to define $c_\ell$ for $\ell\ge 2$ and prove the convergence of the Taylor series with positive radius of convergence (see  Section \ref{sec:analytic}). 

\subsection{Choice of branch at $P_6$ and $P_8$}\label{local behavior}

Throughout this section, for ease of notation, we will denote by $P_*$ either $P_6$ or $P_8$. From \eqref{ODE}, we have $\frac{dC}{dV} = \frac{0}{0}$ at $P_*$. Therefore, for smooth solutions, by using L'H\^{o}pital's rule, we see that the slope $c_1$ at $P_*$ must solve the quadratic equation
\begin{align}\label{quadratic equation of c_1}
    -G_C(V_*,C_*)c_1^2+(F_C(V_*,C_*)-G_V(V_*,C_*))c_1+F_V(V_*,C_*) = 0,
\end{align}
where \begin{equation}\label{all partials}
	\begin{aligned}
	G_C(V_*,C_*) &= \frac{\partial G}{\partial C}\Big|_{(V_*,C_*)}=2C_*[(m+1)V_*+2mz],\\
	G_V(V_*,C_*) &= \frac{\partial G}{\partial V}\Big|_{(V_*,C_*)}=(m+1)C_*^2-3V_*^2-2(\lambda+1)V_*-\lambda,\\
	F_C(V_*,C_*) &= \frac{\partial F}{\partial C}\Big|_{(V_*,C_*)}=3C_*^2[1+\frac{mz}{(1+V_*)}]- a_1(1+V_*)^2+a_2(1+V_*)-a_3,\\
	F_V(V_*,C_*) &= \frac{\partial F}{\partial V}\Big|_{(V_*,C_*)}\!\begin{aligned}[t]
		&=C_*\Big\{-mz- 2a_1(1+V_*)+a_2\Big\}\\
		&=-mzC_*- 2a_1(1+V_*)^2+a_2(1+V_*)\\
		&=C_*^2-3a_1(1+V_*)^2+2a_2(1+V_*)-a_3.
	\end{aligned}
	\end{aligned}
\end{equation}
Solving the quadratic equation \eqref{quadratic equation of c_1}, we get
\begin{align}\label{c_1}
	c_1 = \frac{F_C(V_*,C_*)-G_V(V_*,C_*)\pm R(V_*,C_*)}{2G_C(V_*,C_*)},
\end{align}
where
\begin{equation}\label{R}
	R(V_*,C_*) = \sqrt{(F_C(V_*,C_*)-G_V(V_*,C_*))^2+4F_V(V_*,C_*)G_C(V_*,C_*)}.
\end{equation}
Since the first trajectory should be monotone decreasing from $P_1$ to $P_*$, we demand the slope $c_1$ at $P_*$ to be negative. In particular, solutions for  \eqref{quadratic equation of c_1} must be real, which requires the expression under the square root of \eqref{R} to be non-negative. 

In order to establish the necessary conditions for $R$ to be real and to understand the possible solutions of \eqref{quadratic equation of c_1}, 
we analyze the properties of the four partial derivatives in \eqref{all partials}.  
Using $G(V_*, C_*)=0$, $F(V_*, C_*)=0$ and $C_*= 1+V_*$, we see that 
\begin{align}
   G_C(V_*,C_*) &= \frac{G_C(V_*,C_*)C_*}{1+V_*} = 2V_*(\lambda +V_*) < 0 \ \text{ since }-1<V_*<0,\ \ \la>1,\label{G_C}\\
   G_V(V_*,C_*) &= (m+1)C_*^2-(\lambda+V_*)-2V_*(\lambda+V_*)-V_*(1+V_*),\label{G_V}\\
   F_C(V_*,C_*) &= 2C_*^2[1+\frac{(\lambda -1)}{\gamma(1+V_*)}] = 2C_*[C_*+\frac{(\lambda -1)}{\gamma}] >0,\label{F_C}\\
   F_V(V_*,C_*) &= -\frac{(\lambda -1)}{\gamma}C_*-a_1C_*^2+a_3- [1+\frac{(\lambda -1)}{\gamma C_*}]C_*^2.\label{F_V}
\end{align}
Summing \eqref{F_C} with \eqref{F_V} and summing \eqref{G_C} with \eqref{G_V} and applying the definitions of $a_1$ and $a_3$ from \eqref{a_1234&z}, we find the simpler identities
\begin{align}
    F_C(V_*,C_*) +F_V(V_*,C_*) &= -\frac{m(\gamma-1)}{2}C_*^2+\frac{(\gamma-1)(\lambda-1)}{2},\label{F_C+F_V=constant}\\
    G_C(V_*,C_*)+G_V(V_*,C_*) &= mC_*^2-(\lambda-1).\label{G_C+G_V=constant}
\end{align}
In turn, these identities imply, recalling \eqref{P6} and \eqref{P8},
\begin{align}
    F_C(V_*,C_*) +F_V(V_*,C_*) &= -\frac{\gamma-1}{2}(G_C(V_*,C_*)+G_V(V_*,C_*)),\label{F_C+F_V=constant(G_C+G_V)}\\
    G_C(V_6,C_6)+G_V(V_6,C_6)&=-mwC_6<0\label{G_C+G_V=C_6},\\
   G_C(V_8,C_8)+G_V(V_8,C_8)&=mwC_8>0\label{G_C+G_V=C_8}.
\end{align}

As a direct consequence, we first obtain the following. 

\begin{lemma}\label{c_1 at P_8 different signs, R^2>0}
	At $P_8$, for any $\gamma\in (1,3]$ and $\lambda\in(1,\lambda_M]$, $R(V_8,C_8)$ is real and strictly positive, and the two solutions of \eqref{quadratic equation of c_1} must have different signs.
\end{lemma}

\begin{proof}
	We first obtain the sign of $F_V(V_8,C_8)$ from \eqref{F_C+F_V=constant(G_C+G_V)}, \eqref{G_C+G_V=C_8} and \eqref{F_C}:
	\begin{align*}
		 F_V(V_8,C_8) = -\frac{\gamma-1}{2}(G_C(V_8,C_8)+G_V(V_8,C_8))-F_C(V_8,C_8) <0.
	\end{align*}
	Thus, as we also have $G_C(V_8,C_8) < 0$ from \eqref{G_C}, it is clear that $R$ is real and positive.\\
	Applying \eqref{c_1}, the product of the two solutions of \eqref{quadratic equation of c_1} is given by $-2\frac{F_V(V_8,C_8)}{G_C(V_8,C_8)}$. As we have just shown that $F_V(V_8,C_8)$ and $G_C(V_8,C_8)$ are both negative, we conclude the proof.
\end{proof}

\begin{remark}\label{whyweneedzm}
	The situation at $P_6$ is different. For $\lambda$ sufficiently close to 1, $R(V_6,C_6)\not\in\R$, and so we require an appropriate range of $\la$ (equivalently of $z$) which  guarantees the 
	above properties at $P_6$. As the first trajectory connecting $P_1$ and $P_6$ is supposed to be monotone decreasing, it is sufficient to consider 
	 only those $V_6\geq V_1$. We therefore denote by $\la_m$ (equivalently $z_m$) 
	the value such that 
	\begin{align}\label{V6zm}
		V_6(\la_m)=V_1.
	\end{align}
	By a straightforward calculation, we have
\begin{align}
	\lambda_m &= \cfrac{m\gamma(\gamma-1)}{(2\gamma-1)(\gamma+1)}+1,\label{lambda_m}\\
	z_m &= \cfrac{(\gamma-1)}{(2\gamma-1)(\gamma+1)}\label{zm}.
\end{align}
It is straightforward to check that $\la_m<\la_M$ for any $\gamma\in(1,3]$. By Lemma \ref{Lemma V_6/8 monoton}, we 
have $V_6(\la)\geq V_1 $ for any $\la\in[\la_m,\la_M]$.

Moreover, by \eqref{D(V1,C1)<0},we have 
\begin{equation}\label{ineq:C6lam}
C_6(\la_m)=1+V_6(\la_m)=1+V_1<C_1.
\end{equation}

\end{remark}

We now show that within the new sonic window the quadratic equation \eqref{quadratic equation of c_1} at $P_6$ has two real solutions with different signs. 

\begin{lemma}\label{c_1 at P_6 must one neg and R^2>0}
	At $P_6$, for any $\gamma\in (1,3]$ and $\lambda\in [\lambda_m ,\lambda_M]$ where $\lambda_m$ is given by \eqref{lambda_m} and $\lambda_M$ is given by \eqref{lambda_{max}}, the two solutions of \eqref{quadratic equation of c_1} are both real and have different signs.
\end{lemma}
\begin{proof}
	If $F_V(V_6,C_6)<0$, then by the same argument as in Lemma \ref{c_1 at P_8 different signs, R^2>0}, $R$ must be real and one of the solutions must be negative. 
	Suppose $F_V(V_6,C_6)\geq 0$. By \eqref{all partials}, as $C_6>0$, we then have
	$$-mz-2a_1(1+V_6)+a_2\geq 0.$$
	As $a_1=1+\frac{m(\gamma-1)}{2}>0,$ this is
	is equivalent to 
	$$V_6\leq \frac{a_2-mz}{2a_1}-1.$$
	We will now show that, in fact, for all $\la\in[\la_m,\la_M]$, the reverse inequality holds.
	Given $\gamma\in(1,3]$, we always have
	\begin{align*}
		\frac{a_2-mz}{2a_1}-1 - V_6(\lambda_m)& = \frac{a_2-mz}{2a_1}-1+\frac{2}{\gamma+1}\\
		&=\frac{1}{2a_1}\Big\{\frac{m(\gamma-1)+(\gamma-3)(\lambda-1)}{2}-\frac{\lambda-1}{\gamma}-\frac{m(\gamma-1)^2+2(\gamma-1)}{\gamma+1}\Big\}\\
		 &= \frac{1}{2a_1}\frac{\gamma(\gamma-1)(-m\gamma+3m-4)+(\gamma+1)(\gamma^2-3\gamma-2)(\lambda-1)}{2\gamma(\gamma+1)}\\
		 &=\begin{cases}
		 	& \frac{1}{2a_1}\frac{(\gamma^2-3\gamma-2)(\lambda-1)-\gamma(\gamma-1)}{2\gamma}, \quad\quad\text{when }m=1\\
		 	&\frac{1}{2a_1}\frac{(\gamma+1)(\gamma^2-3\gamma-2)(\lambda-1)-2\gamma(\gamma-1)^2}{2\gamma(\gamma+1)}, \quad\quad\text{when }m=2.
		 \end{cases}		 
	\end{align*}
	In each case, we see that $\frac{a_2-mz}{2a_1}-1 - V_6(\lambda_m)<0$, and so
	which means $V_6(\lambda_m)>\frac{a_2-mz}{2a_1}-1$. By Lemma \ref{Lemma V_6/8 monoton}, $V_6$ is strictly increasing in $\lambda\in (1 ,\lambda_M)$. We conclude that for $\gamma\in(1,3]$ and any $\lambda\in[\lambda_m,\lambda_M]$, we always have $F_V<0$. This means that two slopes at $P_6$ have different signs and $R(V_6,C_6)\in\R$.
\end{proof}
In conclusion, for each $\gamma\in (1,3]$ and for the appropriate range of $z$ at $P_*$,   
there exists exactly one negative slope 
\begin{align}\label{c1}
	c_1=\frac{F_C(V_*,C_*)-G_V(V_*,C_*)+ R(V_*,C_*)}{2G_C(V_*,C_*)}<0
\end{align}
which will be our choice of branch. Here we have used $G_C < 0$ by \eqref{G_C}. 

\subsection{Analyticity at $P_6$ and $P_8$}\label{sec:analytic}

As shown in the previous section, to have the first trajectory  
with negative slope, 
the ranges of $\lambda$ at $P_6$ and $P_8$ are taken differently. For notational convenience, we define
\begin{align}
	\Lambda = \begin{cases}
		[\lambda_m, \lambda_M] & \text{if } P_* = P_6, \\
		(1, \lambda_M] & \text{if } P_* = P_8.
	\end{cases}
\end{align}
	We write the formal Taylor series around the point $P_*$ as
\begin{align}\label{Taylor Series}
    C(V) = \sum_{\ell =0}^{\infty}c_{\ell}(V-V_*)^{\ell},
\end{align}
where $c_0=C_*$.  
In a neighborhood of $(V_*,C_*)$, we formally have 
\begin{align}\label{taylor expansion of dc/dv}
    \frac{dC}{dV} = \sum_{\ell =1}^{\infty}\ell c_{\ell}(V-V_*)^{\ell-1}.
\end{align}
Now, to simplify notation, we set
\begin{equation}\label{v(C^2)_l(c)^3_l taylor expansion}
     \left\{  
\begin{aligned}
    &v = V-V_*  
    \\
    &(c^2)_{\ell} = \sum_{\substack{i+j = \ell\\ i,j\geq 0}}c_ic_j, \\
    &(c^3)_{\ell} = \sum_{\substack{i+j+k = \ell\\ i,j,k\geq 0}}c_ic_jc_k .
\end{aligned}
\right. 
\end{equation}
With this notation, the following quantities have a simple expression: \begin{allowdisplaybreaks}
\begin{equation}\label{C^2C^3C'C^2 taylor expansion}
\begin{aligned}
    C^2 &= (\sum_{\ell =0}^{\infty}c_{\ell}v^{\ell})^2 = \sum_{\ell =0}^{\infty}(c^2)_{\ell}v^{\ell},\\
    C^3 &= (\sum_{1}^{\infty}c_{\ell}v^{\ell})^3 = \sum_{\ell =0}^{\infty}(c^3)_{\ell}v^{\ell},\\
    C'C^2 &= \frac{1}{3}(C^3)' =  \frac{1}{3}\sum_{\ell =1}^{\infty}\ell(c^3)_{\ell}v^{\ell-1}.
\end{aligned}
\end{equation}
\end{allowdisplaybreaks}

\

%%%%%%%%%%%%%%%%%%%%%%%%%
%%%%%%%%%%%%%%%%%%%%%%%%%

\begin{lemma}\label{lemma:Taylorcoeffs}
Suppose $C(V)$ defined by \eqref{Taylor Series} is an analytic solution of \eqref{ODE}. Then the following identity holds:
\begin{align}\label{taylor expansion from ode}
	\sum_{\ell\geq 2}(A_\ell c_\ell-B_\ell)v^\ell + [-G_C(V_*,C_*)c_1^2+[F_C(V_*,C_*)-G_V(V_*,C_*)]c_1+F_V(V_*,C_*)]c_0v = 0,
\end{align}
where for each $\ell\geq 2$,
\begin{align}
    A_{\ell} &= C_*\big[F_C(V_*,C_*)-G_C(V_*,C_*)c_1-\ell[G_V(V_*,C_*)+G_C(V_*,C_*)c_1]\big],\label{def:Aell}\\
    B_{\ell} &=\frac{(1+V_*)\Big[(m+1)V_*+2mz\Big]}{3}(\ell +1)\sum_{\substack{i+j+k = \ell+1\\ i,j,k\leq \ell-1}}c_ic_jc_k\notag\\
&-\Big[(1+V_*+mz)-\frac{(m+1)(1+2V_*)+2mz}{3}\ell\Big]\sum_{\substack{i+j+k = \ell\\ i,j,k\leq \ell-1}}c_ic_jc_k\notag\\&
  -\Big[1-\frac{m+1}{3}(\ell -1)\Big]\sum_{\substack{i+j+k = \ell-1}}c_ic_jc_k\label{def:Bell}\\
  &-\Big[\big[6V_*^2+(3\lambda + 6)V_*+2\lambda +1\big](\ell -1)-3a_1(1+V_*)^2+2a_2(1+V_*)-a_3\Big]c_{\ell-1}\notag\\
  &-\Big[(\lambda +2+4V_*)(\ell -2)-3a_1(1+V_*)+a_2\Big]c_{\ell-2}-\Big[\ell-3-a_1\Big]c_{\ell-3},\notag
\end{align}
where we use the convention $c_n=0$ for $n<0$.  
\end{lemma}

%%%%%%%%%%%%%%%%%%%%%%%%%

\begin{proof}
The identity follows by substituting \eqref{Taylor Series}, \eqref{taylor expansion of dc/dv}, \eqref{v(C^2)_l(c)^3_l taylor expansion}, and \eqref{C^2C^3C'C^2 taylor expansion} into 
$$(1+V)F(V,C)=(1+V)\frac{dC}{dV}G(V,C)$$ and grouping the coefficients of $v^{\ell}$. For the details, we refer to  Appendix \ref{caltaylorexp}.
\end{proof}

%%%%%%%%%%%%%%%%%%%%%%%%%

Since we are seeking an analytic solution around the sonic point $P_*$, we demand that \eqref{taylor expansion from ode} holds for all $|v|<\epsilon$ where $\epsilon>0$ is sufficiently small.  
We therefore require that the coefficient of $v^\ell$ should be zero at every order $\ell\in\mathbb{N}$. In Section \ref{local behavior}, we have already shown the existence of $c_1<0$  satisfying
\begin{align*}
	-G_C(V_*,C_*)c_1^2+[F_C(V_*,C_*)-G_V(V_*,C_*)]c_1+F_V(V_*,C_*) = 0.
\end{align*}
For $\ell\geq 2$, we directly obtain the recursive relation for $c_\ell$,
\begin{align}\label{recursive relation}
    A_{\ell}c_{\ell} = B_{\ell},
\end{align}
where we note from \eqref{def:Bell} that $B_\ell$ involves only coefficients $c_i$ for $0\le i\le \ell -1$.  
To ensure the solvability of $c_{\ell}$  for all $\ell \geq 2$, it is obvious that we require $A_\ell\neq 0$, and so we need the following non-vanishing condition:
\begin{equation}\tag{NVC}\label{nonvanish condition}
	\begin{aligned}
			F_C(V_*,C_*)-G_C(V_*,C_*)c_1-\ell[G_V(V_*,C_*)+G_C(V_*,C_*)c_1]\neq 0 \text{ for any }\ell\geq 2.
	\end{aligned}
\end{equation}
In the following lemma, we show that \eqref{nonvanish condition} holds for any $\lambda\in\Lambda$. 

\begin{lemma}\label{nonvanish condition lemma}
	Let $\gamma\in(1,3]$, $P_*\in\{P_6,P_8\}$. Then, for any $\lambda\in\Lambda$ and $\ell\geq 2$,
	  \eqref{nonvanish condition} is satisfied.
\end{lemma} 

\begin{proof}
	Recalling \eqref{G_V} and $C_*=1+V_*$, we first see
	 \begin{align*}
	 	G_V(V_*,C_*) &= (m+1)C_*^2-(\lambda+V_*)-2V_*(\lambda+V_*)-V_*(1+V_*)\\
	 	&=(m-2)V_*^2+2(m-\lambda)V_*+(m+1-\lambda).
	 \end{align*}
	 When $m=1$, this gives that
	 \begin{align*}
	 	G_V(V_*,C_*) = -V_*^2+2(1-\lambda)V_*+2-\lambda
	 \end{align*}
	 is quadratic with respect to $V_*$. Furthermore, $G_V(-1,0)=\lambda-1>0$ and $G_V(0,1) = 2-\lambda\geq 2-\lambda_M = 1-\frac{\gamma}{\sqrt{\gamma} +\sqrt{2}}>0$ for any $\gamma\in(1,3]$. Thus, as $V_*\in(-1,0)$ for all $\la\in\Lambda$, we obtain $G_V(V_*,C_*)>0$.
	 
	 When $m=2$, we have
	 \begin{align*}
	 	G_V(V_*,C_*) = 2(2-\lambda)V_*+3-\lambda=2(2-\lambda)(1+V_*)+\lambda-1>0.
	 \end{align*} 
	 Therefore, for any $\gamma\in(1,3]$ and $\lambda\in\Lambda$, $G_V(V_*,C_*)>0$.
	 	Denote
	\begin{align*}
		f(\ell)=:F_C(V_*,C_*)-G_C(V_*,C_*)c_1-\ell[G_V(V_*,C_*)+G_C(V_*,C_*)c_1].
	\end{align*}
	Notice that $f(\ell)$ is a linear equation with respect to $\ell$. When $\ell=1$, using \eqref{c1}, we see $f(1) = F_C(V_*,C_*)-G_V(V_*,C_*)-2G_C(V_*,C_*)c_1 = -R(V_*,C_*)<0$. Thus, $f(\ell)<0$ for any $\ell\geq 2$ since we have just shown $G_V(V_*,C_*)>0$ and $G_C(V_*,C_*)c_1>0$ by \eqref{G_C} and \eqref{c1}.
		
	In conclusion,\eqref{nonvanish condition} is satisfied for any $\lambda\in\Lambda$ 
	at $P_*$.
\end{proof}

Since $A_{\ell}\neq 0$ by Lemma \ref{nonvanish condition lemma}, we can rewrite
\begin{align}
	c_\ell = \cfrac{B_\ell}{A_\ell}.\label{c_l equation}
\end{align}
In the following, we estimate the growth of $B_L$ under the inductive growth assumption on $c_\ell$ for $2\le \ell \le L-1$.  

\begin{lemma}\label{lemma for B_L upperbound}
For any fixed $\gamma\in(1,3]$ and $\lambda\in\Lambda$, let $\alpha\in (1,2)$ be given. Then, there exists a constant $K_*=K_*(\gamma)>1$ such that if $K\geq K_*$ and $L\geq 5$, then if also the following inductive assumption holds,
\begin{align}\label{c_l assumption}
    |c_{\ell}| \leq \cfrac{K^{\ell-\alpha}}{\ell^3},\quad 2\leq \ell \leq L-1,
\end{align}
then we have
\begin{align}\label{B_L upperbound}
    |B_L| \leq \beta \cfrac{K^{L-\alpha}}{L^2} \Big(
    \cfrac{1}{K^{\alpha-1}}+\frac{1}{K}\color{black}\Big)
\end{align}
for some constant $\beta=\beta(\gamma,\lambda)$.
\end{lemma}

For the proof, we will require the following result from \cite{GHJS22} to estimate certain combinations of coefficients. 
\begin{lem}[{\cite[Lemma B.1]{GHJS22}}]
There exists a universal constant $a>0$ such that for all $L\in \mathbb{N}$, the following inequalities hold
\begin{align*}
    \sum_{\substack{i+j+k = L\\ i,j,k\geq 1}}\cfrac{1}{i^3j^3k^3}& \leq \cfrac{a}{L^3},\\
    \sum_{\substack{i+j = L\\ i,j\geq 1}} \cfrac{1}{i^3j^3} &\leq \cfrac{a}{L^3}.
\end{align*}
\end{lem}

\begin{proof}[Proof of Lemma \ref{lemma for B_L upperbound}]
First, by using the induction assumption \eqref{c_l assumption} and Lemma B.1, we have
\begin{align*}
    \Big|\sum_{\substack{i+j+k = L\\ i,j,k\leq L-1}}c_ic_jc_k\Big| 
    & = \Big|6c_0c_1c_{L-1}+3c_0\sum_{j=2}^{L-2}c_jc_{L-j}+ 3c_1^2c_{L-2} +3c_1\sum_{j=2}^{L-3}c_jc_{L-1-j}+\sum_{\substack{i+j+k = L\\ i,j,k\geq 2}}c_ic_jc_k\Big|\\
    & \leq 6\big|c_0c_1\big|\cfrac{K^{L-1-\alpha}}{(L-1)^3}+3|c_0|K^{L-2\alpha}\sum_{j=2}^{L-2}\frac{1}{j^3(L-j)^3}\\
    &\qquad + 3c_1^2\cfrac{K^{L-2-\alpha}}{(L-2)^3} +3|c_1|K^{L-1-2\alpha}\sum_{j=2}^{L-3}\frac{1}{j^3(L-1-j)^3}+K^{L-3\alpha}\sum_{\substack{i+j+k = L\\ i,j,k\geq 2}}\frac{1}{i^3j^3k^3}\\
    &\lesssim \cfrac{K^{L-1-\alpha}}{(L-1)^3}+\cfrac{K^{L-2\alpha}}{L^3}+\cfrac{K^{L-2-\alpha}}{(L-2)^3}+\cfrac{K^{L-1-2\alpha}}{(L-1)^3} + \cfrac{K^{L-3\alpha}}{L^3}\\
    &\lesssim \cfrac{K^{L-1-\alpha}}{L^3},
\end{align*}
where we have used that $c_0$ and $c_1$ are bounded by a constant depending on $\gamma$ and $\la$ as well as the assumptions $\al\in(1,2)$ and $K\geq 1$ and, moreover, $L-2> 2$, so that the inductive assumption applies still to $c_{L-2}$. Note that as $L\geq 5$, there exists a universal constant $C>0$ such that $\frac{L}{L-1},\frac{L}{L-2},\frac{L}{L-3}\leq C$. 
Next, a similar argument yields
\begin{align*}
\Big|\sum_{\substack{i+j+k = L-1}}c_ic_jc_k\Big| 
    & = \Big|3c_0^2c_{L-1}+6c_0c_1c_{L-2}+3c_0\sum_{j=2}^{L-3}c_jc_{L-1-j}+ 3c_1^2c_{L-3}+3c_1\sum_{j=2}^{L-4}c_jc_{L-2-j}+\sum_{\substack{i+j+k = L-1\\ i,j,k\geq 2}}c_ic_jc_k\Big|\\
    &\lesssim \cfrac{K^{L-1-\alpha}}{(L-1)^3}+\cfrac{K^{L-2-\alpha}}{(L-2)^3}+\cfrac{K^{L-1-2\alpha}}{(L-1)^3}+\cfrac{K^{L-3-\alpha}}{(L-3)^3}+\cfrac{K^{L-2-2\alpha}}{(L-2)^3} + \cfrac{K^{L-1-3\alpha}}{(L-1)^3}\\
    &\lesssim \cfrac{K^{L-1-\alpha}}{L^3}, 
\end{align*}
where we again note that as $L\geq 5$, $L-3\geq 2$, so that the inductive assumption applies to $c_{L-3}$. Again using similar arguments, we bound
\begin{align*}
	\Big|\sum_{\substack{i+j+k = L+1\\ i,j,k\leq L-1}}c_ic_jc_k\Big| &=\Big|3c_0\sum_{\substack{j+k = L+1\\ j,k\leq L-1}}c_jc_k+3c_1^2c_{L-1}+3c_1\sum_{\substack{j+k = L\\ j,k\leq L-2}}c_jc_k+\sum_{\substack{i+j+k = L+1\\ i,j,k\geq 2}}c_ic_jc_k\Big|\\
	  &\lesssim \cfrac{K^{L+1-2\alpha}}{(L+1)^3}+ \cfrac{K^{L-1-\alpha}}{(L-1)^3}+\cfrac{K^{L-2\alpha}}{L^3}+ \cfrac{K^{L+1-3\alpha}}{(L+1)^3}\\
	  &\lesssim  \cfrac{K^{L+1-2\alpha}}{L^3}. 
\end{align*}
Now we  estimate $B_L$, recalling the definition in \eqref{def:Bell}, by employing these three combinatorial estimates as
\begin{align*}
    |B_L| & \lesssim (L+1)\bigg(\Big|\sum_{\substack{i+j+k = L+1\\ i,j,k\leq L-1}}c_ic_jc_k\Big| + \Big|\sum_{\substack{i+j+k = L\\ i,j,k\leq L-1}}c_ic_jc_k\Big| + \Big|\sum_{\substack{i+j+k = L-1}}c_ic_jc_k\Big| +|c_{\ell-1}|+|c_{\ell-2}|+|c_{\ell-3}|\bigg) \\
    & \lesssim (L+1)\bigg(\cfrac{K^{L+1-2\alpha}}{L^3} + \cfrac{K^{L-1-\alpha}}{L^3} + \cfrac{K^{L-1-\alpha}}{L^3} + \cfrac{K^{L-1-\alpha}}{L^3} + \cfrac{K^{L-2-\alpha}}{L^3}+ \cfrac{K^{L-3-\alpha}}{L^3}\bigg)\\
    &\lesssim \cfrac{K^{L-\alpha}}{L^2} \Big( 
    \cfrac{1}{K^{\alpha-1}}+\frac{1}{K}\Big),
\end{align*}
where have used that there exists a universal constant $C>0$ such that $\frac{L+1}{L}\leq C$ for all $L\geq 5$.
\end{proof}

We next justify the inductive growth assumption on $c_\ell$. 

\begin{lemma}\label{lemma for c_l}
For any fixed $\gamma\in(1,3]$ and $\lambda\in\Lambda$, let $\alpha\in(1,2)$ be given. Let $c_{\ell}$ be the coefficients in the formal Taylor expansion of $C(V)$ around $(V_*,C_*)$ solving the recursive relation of Lemma \ref{lemma:Taylorcoeffs}. Then there exists a constant $K=K(\ga,\la)>1$ such that  
$c_\ell$ satisfies the bound
\begin{align}\label{c_l condition}
    |c_{\ell}| \leq \cfrac{K^{\ell-\alpha}}{\ell^3}. 
\end{align}
\end{lemma}
\begin{proof} We argue by induction on $\ell$.  
When $\ell =2,3,4$, it is clear from \eqref{c_l condition}, the forms of $A_\ell$ and $B_\ell$ defined by \eqref{def:Aell}--\eqref{def:Bell}, and the non-vanishing condition \eqref{nonvanish condition} that there exists a constant $K(\gamma,\lambda)$ such that $c_2$, $c_3$, and $c_4$ satisfy the bounds. 
Suppose for some $L\geq 5$, \eqref{c_l condition} holds for all $2\leq \ell \leq L-1$. Then we may apply Lemma \ref{lemma for B_L upperbound}   
and with the recursive relation \eqref{c_l equation}, we obtain 
\begin{align*}
    |c_L| \leq \cfrac{\beta}{|A_L|} \cfrac{K^{L-\alpha}}{L^2} \Big( 
    \cfrac{1}{K^{\alpha-1}}+\frac{1}{K}\color{black}\Big),
\end{align*}
where $A_L = C_*\big[F_C(V_*,C_*)-G_C(V_*,C_*)c_1-L[G_V(V_*,C_*)+G_C(V_*,C_*)c_1]\big]$. As $A_L$ is linear in $L$ and non-zero for all $L$, there exists constants $\eta_1 = \eta_1(\gamma,\lambda)$ and $\eta_2= \eta_2(\gamma,\lambda)$ such that $\eta_1 L\leq |A_L| \leq \eta_2 L$ for all $L\ge 5$. Therefore,
\begin{align*}
    |c_L| \leq \beta\frac{1}{\eta_1 L} \cfrac{K^{\ell-\alpha}}{L^2} \Big( 
    \cfrac{1}{K^{\alpha-1}}+\frac{1}{K}\color{black}\Big). 
\end{align*}
Choosing $K$  sufficiently large, as $\al>1$, it is clear that the estimate \eqref{c_l condition} holds for $\ell = L$, thus concluding the proof.
\end{proof}

We are now ready to prove the main result of this section. 

\begin{theorem}\label{analytic}
For any fixed $\gamma\in(1,3]$ and $\lambda\in\Lambda$, there exists $\epsilon=\epsilon(\ga,\la)>0$ such that the Taylor series
\begin{align}\label{local analytic solution}
	C(V) = \sum_{\ell =0}^{\infty} c_{\ell}(V-V_*)^{\ell}
\end{align}
 converges absolutely on the interval $(V_*-\epsilon,V_*+\epsilon)$. Moreover, $C(V)$ is the unique analytic solution to \eqref{ODE}.
\end{theorem}
\begin{proof}
	Let $\alpha \in (1,2)$ be fixed and suppose $|V-V_*|<\epsilon$, where $\epsilon>0$ is to be chosen later. By \eqref{c_l condition} in Lemma \ref{lemma for c_l}, there exists a constant $K(\alpha,\gamma,\lambda)$ such that
	\begin{align*}
		|\sum_{\ell =2}^{\infty} c_{\ell}(V-V_*)^{\ell}| \leq \sum_{\ell =2}^{\infty} \cfrac{K^{\ell-\alpha}}{\ell^3}\epsilon^{\ell}\leq \sum_{\ell=2}^\infty  (K\epsilon)^\ell< \infty,
	\end{align*}
	provided $\epsilon <\frac{1}{K}$.
	Thus, $\sum_{\ell=0}^{\infty} c_{\ell}(V-V_*)^{\ell}$ converges absolutely for $V \in (V_*-\epsilon,V_*+\epsilon)$ with $0<\epsilon<\frac{1}{K}$.
\end{proof}

\begin{remark}\label{continuous dependence of the local analytic solution on gamma, z}
Notice that the local analytic solution $C(V)$ obtained in Theorem \ref{analytic} depends on $\gamma$ and $z$. Since all the coefficients $c_\ell$ for any $\ell\geq 1$ are continuous functions of $(\gamma,\la)$ for $\gamma\in (1,3]$ and $\lambda\in \Lambda$ by \eqref{nonvanish condition} and \eqref{G_C}, by standard compactness and uniform convergence, we deduce that $C=C(V;\gamma,\la)$ is a continuous function of $(\gamma,\la)$ (equivalently continuous in $(\ga,z)$) on its domain.  
\end{remark}
\begin{lemma}\label{sign of dC/dV to the left}
The local analytic solution \eqref{local analytic solution}, propagated by \eqref{ODE} to the left of the sonic point $P_*$, except at the sonic point itself,  remains strictly away 
from the zeros of $F$, $G$, and $D$. In addition, the solution to the left of the sonic point $P_*$ satisfies $F>0$, $G<0$, $D<0$ and $\frac{dC}{dV} <0$.
\end{lemma}
\begin{proof}
The result is owing to our  choice of the negative branch $c_1$ (see  \eqref{c1}). The sign conditions then follow from $F_C(V_*,C_*)>0$ and $G_C(V_*,C_*)<0$ from  \eqref{F_C} and \eqref{G_C}. 
\end{proof}

\section{Solving to left: basic setup and constraints on connections}\label{sec:leftbasic}

In this section, we introduce the basic setup for the continuity argument for the first trajectory 
to the left of the sonic point and also show that the solutions of \eqref{ODE system} starting from the initial point $P_1$ can't connect to $P_8$ for $\gamma\in (1,\gaSix]$ and can't connect to $P_6$ for $\gamma\in [2,3]$. We will use $z$ only (instead of $\la$) in the following sections. The corresponding range of $z$ for $P_*$ is given by
\begin{align}
	\mathcal{Z}(\gamma;P_*)=\begin{cases}
		&[z_m(\gamma),z_M(\gamma)] \quad \text{ when } P_*=P_6,\\
		&(0,z_M(\gamma)]\quad \text{ when } P_*=P_8.
	\end{cases}
\end{align}

\subsection{Connection from the sonic point to the initial point $P_1$}\label{sec:existence}

By Theorem \ref{analytic}, the following 
 problem
	\begin{align}\label{ivp p*}
	\begin{cases}
		&\cfrac{dC}{dV} = \cfrac{F(V,C;\gamma,z)}{G(V,C;\gamma,z)},\quad V\in[V_1,V_*],\\
		&C(V_*) = C_*,\\
		&\frac{dC}{dV}(V_*)= c_1,
	\end{cases}
\end{align}
where $c_1$ is as defined in \eqref{c1},
has a local analytic solution. To prove the existence of a trajectory connecting $P_1$ and $P_*$, we will first show that the local analytic solution obtained from Theorem \ref{analytic} extends smoothly as a strictly monotone decreasing solution to the left to $C:[V_1,V_*]\to\R_+$. Secondly, we will show that for any $\gamma\in(1,3]$, there exists a $z_{std}\in\mathcal{Z}(\gamma;P_*)$ such that the local analytic solution from either $P_6$ or $P_8$ can be extended smoothly to $P_1$ by using a continuity argument.

\begin{lemma}\label{local solution to the left}
Let $P_*=(V_*,C_*)$ be either $P_6$ or $P_8$ and suppose $C:[V_*-\epsilon,V_*]\to\R_+$ is the local, analytic solution to \eqref{ivp p*} guaranteed by Theorem \ref{analytic}. Then the solution extends smoothly to the left to $C:[V_1,V_*]\to\R_+$.
\end{lemma}

\begin{proof} We argue by contradiction. Suppose that the maximal time of existence of the solution is $(V_0,V_*]$ for some $V_0\geq V_1$. By Remark \ref{sign of dC/dV to the left}, $C(V)>C(V_*)$ on ($V_0,V_*)$. Moreover, as the right hand side of the ODE \eqref{ivp p*} is locally Lipschitz away from zeros of $G$, we see that the only obstruction to continuation past $V_0$ is blow-up of $C$, i.e., $\limsup_{V\to V_0^+}C(V)=\infty$.

Now, from the explicit forms of $F$ and $G$, we observe that there exists $M=M(\ga,z)>0$ such that if $C\geq M$, for all $V\in[V_1,V_*-\epsilon]$, we have 
\begin{align*}
G(V,C;\gamma,z) = &\,C^2[(m+1)V+2mz]-V(1+V)(\lambda+V)\leq \frac12 C^2[(m+1)V+2mz]<0,\\
F(V,C;\ga,z) = &\,C\big\{C^2[1+\frac{mz}{(1+V)}]- a_1(1+V)^2+a_2(1+V)-a_3\big\} \leq \frac32 C^3[1+\frac{mz}{(1+V)}]
\end{align*}
and $F>0$.
Thus, 
\begin{align*}
\frac{dC}{dV}=\frac{F(V,C;\gamma,z)}{G(V,C;\gamma,z)}\geq 3 \frac{C[1+\frac{mz}{(1+V)}]}{[(m+1)V+2mz]}.
\end{align*}
Thus, as $V$ is contained in a bounded set, we see that there exists a constant $A>0$ such that whenever $C(V)\geq M$, we have
$$\frac{d \log C}{dV}\geq -A,$$
and hence $C$ is necessarily bounded on $(V_0,V_*)$, contradicting the assumption.
\end{proof}

By Theorem \ref{analytic} and Lemma \ref{local solution to the left}, \eqref{ivp p*} has a smooth solution on $[V_1, V_*]$. 
We use $C(V;\gamma,z,P_*)$ to denote the solution of \eqref{ivp p*} at $V\in[V_1,V_*]$. By the fundamental theorem of calculus,
\begin{align}\label{C(V;gamma,z,P_*)}
	C(V;\gamma,z,P_*) = \int_{V_*(\gamma,z)}^{V} \frac{dC(V;\gamma,z)}{dV}  dV +C_*(\gamma,z) = \int_{V_*(\gamma,z)}^{V} \cfrac{F(V,C;\gamma,z)}{G(V,C;\gamma,z)}  dV +C_*(\gamma,z).
\end{align}
It is clear from the expressions for $(V_6,C_6)$ and $(V_8,C_8)$ in \eqref{P6}, \eqref{P8} as well as the continuous dependence of the local, analytic solution on $\ga$, $z$ (\textit{cf.}  Remark \ref{continuous dependence of the local analytic solution on gamma, z}), and the continuity properties of $F$ and $G$ that this is a continuous function with respect to $\gamma$ and $z$. 
 Moreover, the initial value $V_1$ only depends on $\gamma$. Hence, for any fixed $\gamma$ and a sonic point $P_*\in\{P_6,P_8\}$, if we can show that there exists a $\underline{z}\in\mathcal{Z}(\gamma;P_*)$ such that $C(V_1;\gamma,z,P_*)\leq C_1$ and a $\overline{z}\in\mathcal{Z}(\gamma;P_*)$ such that $C(V_1;\gamma,z,P_*)\geq C_1$, then we conclude that there exists a $z_{std}$ such that $\underline{z}\leq z_{std}\leq \overline{z}$ and $C(V_1;\gamma,z_{std},P_*)= C_1$. This motivates the introduction of upper and lower solutions: 

\begin{definition}(Upper and lower solution). Let $\gamma\in (1,3]$ and $z\in\mathcal{Z}(\gamma;P_*)$. Let $C(\cdot;\gamma,z,P_*):[V_1,V_*]\to\R_+$ be the analytic solution obtained from Theorem \ref{analytic} and Lemma \ref{local solution to the left}. We say that $\overline{z}(\gamma;P_*)$ gives an upper solution for $P_*$ if
	\begin{align}\label{upper solution def}
		C(V_1;\gamma,\overline{z}(\gamma;P_*),P_*) > C_1.
	\end{align} 
	We say that $\underline{z}(\gamma;P_*)$ gives a lower solution for $P_*$ if
	\begin{align}\label{lower solution def}
		C(V_1;\gamma,\underline{z}(\gamma;P_*),P_*) < C_1.
	\end{align} 

\end{definition}

The proof of the existence of an analytic solution connecting $P_1$ and either $P_6$ or $P_8$ proceeds as follows. 
We will first show that $P_6$ always admits a lower solution and $P_8$ always admits an upper solution. It will then follow that, depending on whether $C(V;\ga,z_M,P_6)=C(V;\ga,z_M,P_8)$ gives a lower or an upper solution (or connects to $P_1$), at least one of $P_6$ and $P_8$ has both an upper and a lower solution, thus concluding the proof.

First we show the existence of a lower solution for $P_6$.

\begin{lemma}\label{lowersolutionP6}
	Let $\gamma\in(1,3]$. Then there exists $\underline{z}(\gamma;P_6)\in\mathcal{Z}(\gamma;P_6)$ such that  $C(V;\gamma,\underline{z}(\gamma),P_6)$ is a lower solution for $P_6$.
\end{lemma}
\begin{proof}
	This  follows simply from  Remark \ref{whyweneedzm} and the monotonicity of $C(V;\gamma,{z},P_6)$. When $z=z_m(\gamma)$, we have $V_6(z_m(\gamma)) = V_1(\gamma)$ and $C(V_1;\gamma,z_m(\gamma),P_6)=C_6(z_m(\gamma)) < C_1(\gamma)$ by \eqref{ineq:C6lam}. Thus, $C(V;\gamma,z_m(\gamma),P_6)$ gives a lower solution for $P_6$.
\end{proof}

Next we show  the existence of an upper solution for $P_8$.

\begin{lemma}\label{uppersolutionP8}
Let $\gamma\in(1,3]$. Then there exists $\overline{z}(\gamma;P_8)\in\mathcal{Z}(\gamma;P_8)$ such that  $C(V;\gamma,\overline{z}(\gamma;P_8),P_8)$ is an upper solution for $P_8$.
\end{lemma}
\begin{proof}
	By Lemma \ref{C1V1}, $C_1(\gamma)\leq C_1(3)=\frac{\sqrt{3}}{2}<1$ for each $\gamma\in(1,3]$. By Lemma \ref{Lemma V_6/8 monoton}, $C_8(\ga,z)$ is monotone decreasing with respect to $z$ for $0<z\leq z_M$. Since $C_8(\gamma,0) = 1$, there exists a sufficiently small $\overline{z}(\gamma;P_8)\in\mathcal{Z}(\gamma;P_8)$ such that $C_8(\overline{z}(\gamma;P_8))>C_1$. By the monotonicity of $C(V;\gamma,\overline{z}(\gamma;P_8),P_8)$ with respect to $V$, we conclude that $C(V;\gamma,\overline{z}(\gamma;P_8),P_8)> C_1(\ga)$. Thus, $C(V;\gamma,\overline{z}(\gamma;P_8),P_8)$ is an upper solution for $P_8$. 
\end{proof}

We  now prove the main result of this section.

\begin{theorem}\label{existence of zstd}
	Let $\gamma\in(1,3]$. Then there exists a $P_*\in\{P_6,P_8\}$ and a corresponding $z_{std}(\gamma;P_*)\in\mathcal{Z}(\gamma;P_*)$ such that the local analytic solution obtained  from  Theorem \ref{analytic}  extends smoothly from $P_*$ to $P_1$.
	\end{theorem}
\begin{proof}
	By Lemma \ref{local solution to the left}, the domain of the local analytic solution extends smoothly (analytically) to $V=V_1$. It remains to show that for each $\gamma\in (1,3]$ 
	there exist $P_*$ and $z_{std}(\gamma;P_*)\in\mathcal{Z}(\gamma;P_*)$ such that $C(V_1;\gamma,z_{std}(\gamma),P_*) = C_1$. Recall that when $z=z_M$, $P_6$ coincides with $P_8$. Therefore, for $P_*= P_6=P_8$ with $z=z_M$, there are three possibilities: 
	\begin{enumerate}
		\item If $C(V_1;\gamma,z_M(\gamma),P_*) < C_1(\gamma)$, then $z=z_M(\gamma)$ gives a lower solution for $P_8$. Then, by using the continuity argument and Lemma \ref{uppersolutionP8}, there exists a $z_{std}\in(\overline{z}(\gamma;P_8),z_M(\gamma))$ such that $C(V_1;\gamma,z_{std},P_8)=C_1$.
		\item If $C(V_1;\gamma,z_M(\gamma),P_*) = C_1(\gamma)$, then $z=z_M(\gamma)$ gives the solution.
		\item If $C(V_1;\gamma,z_M(\gamma),P_*) > C_1(\gamma)$, then $z=z_M(\gamma)$ gives an upper solution for $P_6$. Then, by using the continuity argument and Lemma \ref{lowersolutionP6}, there exists a $z_{std}\in(\underline{z}(\gamma;P_6),z_M(\gamma))$ such that $C(V_1;\gamma,z_{std},P_6)=C_1$.
	\end{enumerate}
This concludes the proof. 	
\end{proof}

\subsection{No connection to $P_8$ for $\gamma\in (1,\gaSix]$} \label{sec:P6only}

Now that we have established the existence of an analytic solution to \eqref{ODE} connecting $P_1$ to either $P_6$ or $P_8$ for each $\gamma\in(1,3]$, we seek to understand better the nature of the solutions in order to connect the solution through the triple point to the origin. The first step in showing this is to prove that, for $\ga\in(1,\ga_\star]$, for $\ga_\star$ defined below in \eqref{gaSix}, the connection must be to $P_6$. 

For notational convenience,  
we define a constant
\begin{align}\label{k(r,z)}
	k(\gamma,z) = -\frac{(1+V_6(\ga,z))^2}{V_6(\ga,z)},
\end{align}
and, for each $\ga\in(1,2]$, $z\in[z_m,z_M]$, we define a barrier function 
\begin{equation}
B_k(V)=\sqrt{-k(\ga,z)V}.
\end{equation}

\begin{lemma}\label{lower barrier P6}
For any $\gamma\in (1,2]$ and $z\in[z_m,z_M]$, the curve $C=B_k(V)$ is a lower barrier (in the sense of Definition \ref{def:barrier}) for the solution of 
	\begin{align}\label{ivp p6-epsilon}
		\begin{cases}
			&\cfrac{dC}{dV} = \cfrac{F(V,C;\gamma,z)}{G(V,C;\gamma,z)},\quad V\in[V_1,\overline{V}],\\
			&C(\overline{V})= \overline{C},
		\end{cases}
	\end{align}
	where $\overline{V}\in(V_1,V_6)$ and $\overline{C}>B_k(\overline{V})$. 
	In particular, for any $z\in[z_g,z_M]$ the curve  $C=B_k(V)$ is a lower barrier for the solution of the 
	problem \eqref{ivp p*} with $P_*=P_6$. Recall that $z_g\in\mathcal{Z}(\ga;P_6)$ is the value of $z$ such that $P_4=P_6$, see \eqref{zg}.
\end{lemma}

\begin{proof}
	We begin by showing that the second claim follows from the first one.  
	Observe that, assuming the first part of the lemma is proved, it is sufficient to verify that there exists some interval $[\overline{V},V_6)$  
	such that the solution $C(V)$ to the  
	problem \eqref{ivp p*} satisfies $C(V)>B_k(V)$ for $V\in(\overline{V},V_6)$. 
	This claim follows one we verify that the derivative at $V_6$  
	satisfies the inequality
	\begin{align}\label{C'(V6)<C6/2V6}
			\frac{d C}{d V}(V_6)=c_1 < -\frac{1}{2}\sqrt{\cfrac{k(\gamma,z)}{-V_6}} = \frac{1+V_6}{2V_6}.
		\end{align}
		The proof of this inequality for $\ga\in(1,2]$ and $z\in[z_g,z_M]$ is given in Appendix \ref{proofofC'(V6)<C6/2V6}. 
		
		We therefore focus on proving the first claim. Suppose that $C(V)$ is a solution to the problem \eqref{ivp p6-epsilon}.
We will apply the barrier argument (\textit{cf.} \ref{barrier_argument}) to show that as the initial point $C(\overline{V})>B_k(\overline{V})$, this inequality is propagated by the ODE.
As the solution to the ODE \eqref{ivp p6-epsilon} remains monotone by Lemma \ref{sign of dC/dV to the left}, it is clear that it cannot meet a sonic point.
	Our goal is to show that for any $\gamma\in (1,2]$, $z\in[z_g,z_M]$ and $V\in[V_1,V_6)$,
	\begin{align}\label{barrier inequality}
		\cfrac{F(V,\sqrt{-k(\gamma,z)V};\gamma,z)}{G(V,\sqrt{-k(\gamma,z)V};\gamma,z)}+\frac{1}{2}\sqrt{\cfrac{k(\gamma,z)}{-V}}<0.
	\end{align}
	Since $G(V,\sqrt{-k(\gamma,z)V};\gamma,z)<0$ for any $\gamma\in (1,2]$, $z\in[z_m,z_M]$ and $V\in[V_1,V_6)$ by Lemma \ref{sign of dC/dV to the left}, it is sufficient to show that
	\begin{align}\label{barrier inequality 22}
		F(V,\sqrt{-k(\gamma,z)V};\gamma,z)+\frac{1}{2}\sqrt{\cfrac{k(\gamma,z)}{-V}}G(V,\sqrt{-k(\gamma,z)V};\gamma,z)>0.
	\end{align}
	By direct computations, we obtain 
	\begin{align}
		& \cfrac{2}{\sqrt{-k(\gamma,z)V}} \left( F(V,\sqrt{-k(\gamma,z)V};\gamma,z)+\frac{1}{2}\sqrt{\cfrac{k(\gamma,z)}{-V}}G(V,\sqrt{-k(\gamma,z)V};\gamma,z)\right)\notag\\
		&=(m-1-m\gamma)V^2+[-2-m(\gamma-1)+m \gamma z(\gamma-2)+(m-1)k(\gamma,z)]V+\cfrac{2mk(\gamma,z)z}{1+V}-1-m\gamma z \notag\\
		&=: \Barrier_k(V,z,m). \label{def_Bk}
	\end{align}
	Since $\cfrac{\sqrt{-k(\gamma,z)V}}{2}>0$, \eqref{barrier inequality 22} is equivalent to  
	the positivity of $\Barrier_k(V,z,m)$:  
	\begin{align}
		\Barrier_k(V,z,m) 
		>0\label{sign_Bk}
	\end{align}
	for any $\gamma\in (1,2]$, $z\in[z_m,z_M]$ and $V\in[V_1,V_6)$.

	As  $\Barrier_k(V_6,z,m)=0$ for any $\gamma\in (1,2]$ and $z\in[z_m,z_M]$ (due to $B_k(V_6)=C_6$ and the vanishing of $F$ and $G$ at $(V_6,C_6)$), we will conclude that $\Barrier_k(V,z,m)>0$ by demonstrating that 
	for $V\in[V_1,V_6)$, 
	\begin{equation}\label{demonstrate00}
	\frac{\partial\Barrier_k(V,z,m)}{\partial V}<0.
	\end{equation}
	The $V$ derivative of $\Barrier_k$ is given by 
	\begin{align*}
		\frac{\partial\Barrier_k(V,z,m)}{\partial V} = 2(m-1-m\gamma)V-2-m(\gamma-1)+m\gamma z(\gamma-2)+(m-1)k(\gamma,z)-\cfrac{2mk(\gamma,z)z}{(1+V)^2}.
	\end{align*}
	When $m=1$, for any $V\in[V_1,V_6)$,
	\begin{align*}
		\frac{\partial\Barrier_k(V,z,1)}{\partial V} = -2\gamma V-1-\gamma+(\gamma-2)\gamma z+\cfrac{2(1+V_6)^2z}{V_6(1+V)^2}<-2\gamma V_1-1-\gamma+(\gamma-2)\gamma z+\cfrac{2z}{V_6} =: I,
	\end{align*}
	where we have used $-2\gamma V<-2\gamma V_1$ and $\cfrac{2(1+V_6)^2z}{V_6(1+V)^2}<\cfrac{2z}{V_6}$ for any $V\in(V_1,V_6)$.
	Recalling $\eqref{V after shock}$ and $\eqref{P6}$, we deduce 
	\begin{align*}
		I 
		&= -\frac{(\gamma-1)^2}{\gamma+1}+[(\gamma-2)\gamma+\frac{2}{V_6}]z
		<-\frac{(\gamma-1)^2}{\gamma+1}+[(\gamma-2)\gamma+\frac{2}{V_1}]z
		=-\frac{(\gamma-1)^2}{\gamma+1}+[\gamma(\gamma-3)-1]z<0
	\end{align*}
	for any $\gamma\in(1,2]$ and $z\in[z_m,z_M]$, which in turn leads to \eqref{demonstrate00} for $m=1$. 
	
	When $m=2$, for any $V\in[V_1,V_6)$,
	\begin{align*}
		\frac{\partial\Barrier_k(V,z,2)}{\partial V}&= 2(1-2\gamma)V-2\gamma+2(\gamma-2)\gamma z-\frac{(1+V_6)^2}{V_6}+\frac{4(1+V_6)^2z}{V_6(1+V)^2}\\
		&<2(1-2\gamma)V_1-2\gamma+2(\gamma-2)\gamma z-\frac{(1+V_6)^2}{V_6}+\frac{4z}{V_6} =: II + III 
	\end{align*}
	since $(1-2\gamma) V<(1-2\gamma) V_1$ and $\cfrac{2(1+V_6)^2z}{V_6(1+V)^2}<\cfrac{2z}{V_6}$ for any $V\in(V_1,V_6)$, where $II $ and $III$ denote 
	\begin{align*}
		II&:=2(1-2\gamma)V_1-2\gamma+2(\gamma-2)\gamma z \ \ \text{ and }\ \ 
		III:=-\frac{(1+V_6)^2}{V_6}+\frac{4z}{V_6}.
	\end{align*}
	By \eqref{V after shock},
	\begin{align*}
		II = -\frac{2(\gamma-1)(\gamma-2)}{\gamma+1}+2(\gamma-2)\gamma z.
	\end{align*}
	Using \eqref{P6} and \eqref{w(z)}, we rewrite $III$ as 
	\begin{align*}
		III 
		&=\frac{-1-(\gamma-2)^2z^2-w^2-2(\gamma-2)z+2w+2(\gamma-2)zw+16z}{4V_6}\\
		&=\frac{-(1+(\gamma-2)^2z^2-2(\gamma+2)z)-w^2+2w+2(\gamma-2)zw+(12-4\gamma)z}{4V_6}\\
		&=\frac{2(1-w)w+4(3-\gamma)z+2[(\gamma-2)w+2]z}{4V_6}.
	\end{align*}
By Remark \ref{bounds of w(z)} and the fact that $-1<V_6<0$  for any $\gamma\in(1,2]$ and $z\in[z_m,z_M]$, we have
	\begin{align*}
		III&<-\frac{2(1-w)w+4(3-\gamma)z+2[(\gamma-2)w+2]z}{4}<-\frac{4(3-\gamma)+2[(\gamma-2)w+2]}{4}z
		<(\gamma-\frac{7}{2})z.
	\end{align*}
	Then, $II+III$ is bounded by 
	\begin{align*}
		II+III < -\frac{2(\gamma-1)(\gamma-2)}{\gamma+1}+2(\gamma-2)\gamma z+(\gamma-\frac{7}{2})z< -\frac{2(\gamma-1)(\gamma-2)}{\gamma+1}+4(\gamma-2)\gamma z
	\end{align*}
	where we have used $\gamma-\frac{7}{2}<2(\gamma-2)\gamma$ for any $\gamma\in(1,2]$.  
	Since $z_m = \frac{\gamma-1}{(2\gamma-1)(\gamma+1)}$ by \eqref{zm},
	\begin{align*}
		II+III < -\frac{2(\gamma-1)(\gamma-2)}{\gamma+1}+4(\gamma-2)\gamma z& \le -\frac{2(\gamma-1)(\gamma-2)}{\gamma+1}+\frac{4\gamma(\gamma-1)(\gamma-2)}{(2\gamma-1)(\gamma+1)} \\
		&= \frac{2(\gamma-1)(\gamma-2)}{\gamma+1}(\frac{2\gamma}{2\gamma-1}-1)\leq 0.
	\end{align*}
	Therefore,  \eqref{demonstrate00} holds for $m=2$, $\gamma\in (1,2]$, $z\in[z_m,z_M]$ and $V\in[V_1,V_6)$, thereby completing the proof. 
\end{proof}

	Next, we establish a uniform upper barrier for the forward solution trajectory of \eqref{ODE} with the initial value $P_1$ for a particular range of $\gamma$ and $z\in[z_m(\gamma),z_M(\gamma)]$ to demonstrate there is no connection from $P_1$ to $P_8$.

By \eqref{k(r,z)} and \eqref{V68(z_M) C68(z_M)},
\begin{align}\label{k(r,zM)}
	k(\gamma, z_M) = -\frac{(1+V_6(z_M))^2}{V_6(z_M)} = \frac{\gamma}{2+\sqrt{2\gamma}}.
\end{align}
 We define $\gaSix$ to be the value such that
\begin{align}\label{gaSix}
	C_1(\gaSix)=\sqrt{-k(\gaSix, z_M)V_1(\gaSix)}.
\end{align}

\begin{lemma}\label{upper barrier C1}
	For any $\gamma\in (1,\gaSix]$ and $z\in[z_m,z_M]$, the curve $B_{k_M}(V)=\sqrt{-k(\gamma,z_M)V}$ is an upper barrier for the solution of 
	\begin{align}\label{ivp I V6}
		\begin{cases}
			&\cfrac{dC}{dV} = \cfrac{F(V,C;\gamma,z)}{G(V,C;\gamma,z)},\quad V\in[V_1,V_6),\\
			&C(V_1) = C_1.
		\end{cases}
	\end{align}
\end{lemma}
\begin{proof}

We begin by verifying that $P_1=(V_1,C_1)$ lies on or below the curve defined by $B_{k_M}(V)$.
	Note that 
	\begin{align*}
		C_1 - \sqrt{-k(\gamma, z_M)V_1} =  \frac{\sqrt{2\gamma(\gamma-1)}}{\gamma+1}-\sqrt{\frac{\sqrt{2}\gamma}{(\sqrt{\gamma}+\sqrt{2})(\gamma+1)}}
		=\cfrac{\frac{2\gamma(\gamma-1)}{\gamma+1}-\frac{\sqrt{2}\gamma}{\sqrt{\gamma}+\sqrt{2}}}{\sqrt{2\gamma(\gamma-1)}+\sqrt{\frac{\sqrt{2}\gamma(\gamma+1)}{\sqrt{\gamma}+\sqrt{2}}}}\,.
	\end{align*}
	Since 
	\begin{align*}
		\frac{d}{d\gamma}\left[\frac{2\gamma(\gamma-1)}{\gamma+1}-\frac{\sqrt{2}\gamma}{\sqrt{\gamma}+\sqrt{2}}\right] &= \frac{2(\gamma^2+2\gamma-1)}{(\gamma+1)^2}-\frac{\sqrt{2\gamma}+4}{2(\sqrt{\gamma}+\sqrt{2})^2}\\
		&=-\frac{4}{(\gamma+1)^2}-\frac{1}{\sqrt{2}(\sqrt{\gamma}+\sqrt{2})}-\frac{1}{(\sqrt{\gamma}+\sqrt{2})^2}+2>0
	\end{align*}
	because $-\frac{4}{(\gamma+1)^2}>-1$, $-\frac{1}{\sqrt{2}(\sqrt{\gamma}+\sqrt{2})}>-\frac{1}{3}$, and $-\frac{1}{(\sqrt{\gamma}+\sqrt{2})^2}>-\frac{1}{2}$ for any $\gamma\in(1,\gaSix]$, and since  $C_1(\gaSix)=\sqrt{-k(\gaSix, z_M)V_1}$ by the definition of $\gaSix$, 
	we deduce that $C_1 -  \sqrt{-k(\gamma,z_M)V_1}\leq 0$ when $\gamma\leq\gaSix$ with equality only when $\gamma=\gaSix$. Hence, $P_1$ is located below the curve $B_{k_M}(V)$ for $\gamma< \gaSix$, and $P_1$ lies on the curve  $B_{k_M}(V)$ when $\gamma=\gaSix$.

	We will now employ a barrier argument (\textit{cf}.\eqref{barrier_argument}) to establish that the curve $B_{k_M}(V)$ serves as an upper barrier for the solution 
	of \eqref{ivp I V6}. Specifically, we will show  
	that for all $\gamma\in (1,\gaSix]$, $z\in[z_m,z_M]$, and $V\in[V_1,V_6)$,
	\begin{align*}
		\cfrac{F(V,\sqrt{-k(\gamma,z_M)V};\gamma,z)}{G(V,\sqrt{-k(\gamma,z_M)V};\gamma,z)}+\frac{1}{2}\sqrt{\cfrac{k(\gamma,z_M)}{-V}}<0.
	\end{align*}
		 By Lemma \ref{sign of dC/dV to the left}, $G(V,\sqrt{-k(\gamma,z_M)V};\gamma,z)<0$ for any $\gamma\in (1,\gaSix]$, $z\in[z_m,z_M]$ and $V\in[V_1,V_6)$.
	Hence, using the same 
	procedure as outlined in Lemma \ref{lower barrier P6} and recalling \eqref{def_Bk}, 
	it is enough to show that
	\begin{align*}
		\Barrier_{k(z_M)}(V,z,m)&:=(m-1-m\gamma)V^2+[-2-m(\gamma-1)+m(\gamma-2)\gamma z+(m-1)k(\gamma,z_M)]V\\
		&\quad+\cfrac{2mk(\gamma,z_M)z}{1+V}-1-m\gamma z>0.
	\end{align*}
	As  $-1<V_6<0$ for any $z\in[z_m,z_M]$,  by Lemma \ref{Lemma V_6/8 monoton}, we have
	\begin{align}
		\frac{\partial k(\gamma,z)}{\partial z} = (\frac{1}{V_6^2}-1)\frac{\partial V_6}{\partial z}>0.
	\end{align} 
	Thus, we have 
	\begin{align}\label{increasing of k}
		k(\gamma,z_M)\ge k(\gamma,z)\quad \text{for any }z\in[z_m,z_M].
	\end{align}
	When $m=1$, recalling \eqref{def_Bk} and using \eqref{increasing of k} and \eqref{sign_Bk}, we deduce that 
	\begin{align*}
		\Barrier_{k(z_M)}(V,z,1) 
		\ge \Barrier_k(V,z,1) >0
	\end{align*}
	for any $z\in[z_m,z_M]$ and $V\in[V_1,V_6)$.

When $m=2$, by \eqref{upper and lower bound for V/C_6/8}, we have $-1<V_6\leq \frac{-\sqrt{2}}{\sqrt{\gamma}+\sqrt{2}}\leq -\frac{1}{2}$ for any $\gamma\in(1,2]$ and $z\in[z_m,z_M]$. Thus,
	\begin{align*}
		k(\gamma,z_M)V + \frac{4zk(\gamma,z_M)}{1+V}  
		= \frac{k(\gamma,z_M)}{1+V}(V^2+V+4z) > \frac{k(\gamma,z_M)}{1+V}(V_6^2+V_6+4z) 
	\end{align*}
	for any $V\in[V_1,V_6)$. By direct computations, we obtain 
	\begin{align*}
		V_6^2+V_6+4z 
		&=\frac{2(\gamma-2)^2z^2+2(2-\gamma)zw+2(6-\gamma)z}{4}>0.
	\end{align*}
	Therefore, by \eqref{increasing of k} and and \eqref{sign_Bk}, we obtain $\Barrier_{k(z_M)}(V,z,2)\ge \Barrier_k(V,z,2)> 0$, thereby completing the proof. 
\end{proof}

\begin{proposition}\label{onlypass P6}
For any $\gamma\in(1,\gaSix]$, the analytic solution to \eqref{ODE} which connects $P_1$ to either $P_6$ or $P_8$, guaranteed by Theorem \ref{existence of zstd} with the initial condition $C(V_1)=C_1$, can only connect to $P_6$. 
\end{proposition}
\begin{proof}
 When $z=z_M$, we have 
 $P_6 = P_8$, thus obviating the need for further discussion.  
 If $z\in(0,z_M)$, by Theorem \ref{existence of zstd}, it is equivalent to demonstrating that the solution trajectory can not connect to $P_8$. We will discuss $z\in(0,z_m]$ and $z\in[z_m,z_M)$ separately.

Let $z\in(0,z_m]$ be given.  
We observe that when $C_8(z)\geq C_1|_{\ga=2}$, the solution trajectory cannot connect to $P_8$, since the solution of \eqref{ODE} with $C(V_1)=C_1$ is decreasing by 
Lemma \ref{sign of dC/dV to the left}. We further note  
that $z = \frac{2}{3(\gamma+4)}$ leads to $C_8(z) = C_1(2)$. By Lemma \ref{C1V1} and Lemma \ref{Lemma V_6/8 monoton}, 
$C_8(z)\geq C_1|_{\ga=2}>C_1(\gamma)$ for any $\gamma\in(1,\gaSix]$ and $z\leq \frac{2}{3(\gamma+4)}$.  
On the other hand, it is easy to check $z_m<\frac{2}{3(\gamma+4)}$: 
	\begin{align*}
		z_m-\frac{2}{3(\gamma+4)} = \frac{\gamma-1}{(2\gamma-1)(\gamma+1)}-\frac{2}{3(\gamma+4)} = \frac{(2-\gamma)(\gamma-5)}{(\gamma-\frac{1}{2})(\gamma+1)(\gamma+4)} <0,
	\end{align*}
and hence, the conclusion follows for $z\in(0,z_m]$.

When $z\in[z_m,z_M)$, we have $C(V_6;\gamma, z)<\sqrt{-k(\gamma,z_M) V_6(z)}$ by Lemma \ref{upper barrier C1}. 
Therefore, in order to show that this solution can not connect to $P_8$, it is sufficient to show that 
	\begin{align*}
		\sqrt{-k(\gamma,z_M)V_6(z)} < C_8(z), \ \text{ equvalently } \  
		-k(\gamma,z_M)V_6(z)-C_8^2(z)<0.
	\end{align*}
	Since $-k(\gamma,z_M)V_6(z_M)-C_8^2(z_M)=C_6(z_M)^2-C_8(z_M)^2=0$  
	by \eqref{V68(z_M) C68(z_M)} and \eqref{k(r,zM)}, the proof will be complete upon  
	showing that $-k(\gamma,z_M)V_6(z)-C_8^2(z)$ is monotone increasing in $z$. Now, differentiating with respect to $z$ (for any fixed $\gamma\in (1,\gaSix)$),  
	and recalling $\frac{dC_8}{dz}<0$,
	\begin{align*}
		\frac{d}{dz}\left(-k(\gamma,z_M)V_6(z)-C_8^2(z)\right) = -\sqrt{\frac{\gamma}{2}}C_8(z_M)\frac{dV_6}{dz}-2C_8(z)\frac{dC_8}{dz}
		>C_8(z_M)\left(-\sqrt{\frac{\gamma}{2}}\frac{dV_6}{dz}-2\frac{dC_8}{dz}\right).
	\end{align*}
	The inner bracket is
	\begin{align*}
		-\sqrt{\frac{\gamma}{2}}\frac{dV_6}{dz}-2\frac{dC_8}{dz} &= -\sqrt{\frac{\gamma}{2}}(\frac{\gamma-2}{2}+\frac{1}{2}\frac{(\gamma+2)-(\gamma-2)^2z}{w})-(\gamma-2-\frac{(\gamma+2)-(\gamma-2)^2z}{w})\\
		&= (2-\gamma)(\frac{1}{2}\sqrt{\frac{\gamma}{2}}+1)+\frac{(\gamma+2)-(\gamma-2)^2z}{w}(1-\frac{1}{2}\sqrt{\frac{\gamma}{2}})>0
	\end{align*}
	where we have used $z<z_M=\frac{1}{\gamma+2+2\sqrt{2\gamma}}<\frac{1}{4}$ and $\gaSix<2$ to conclude the positivity. 	
\end{proof}

\subsection{No connection to $P_6$ for $\gamma\in [2,3]$} \label{sec:P8only}

In this subsection, we shall employ another  
barrier function $B_s(V)$ 
to demonstrate that for $\gamma\in[2,3]$, the solution trajectory originating at $P_1$ and propagated 
by \eqref{ODE} can only establish a connection with $P_8$.

We define
\begin{align*}
	B_s(V) = -\sqrt{\frac{\gamma}{2}}V.
\end{align*}
From 
\eqref{V68(z_M) C68(z_M)}, we observe that
\begin{align}\label{-sqrt(gamma/2)V8=C8}
	\frac{C_8(z_M)}{V_8(z_M)} = -\sqrt{\frac{\gamma}{2}}.
\end{align}
First we will show that the solution trajectory of \eqref{ODE}  
starting from $P_1$  remains above the curve $B_s(V)$ for  $V\in[V_1,-\sqrt{\frac{2}{\gamma}}C_8(z))$.

\begin{lemma}\label{lower barrier sqrt(gamma/2)V}
	For any $\gamma\in [2,3]$ and $z\in (0,z_M]$, the curve $B_s(V)=-\sqrt{\frac{\gamma}{2}}V$ is a lower barrier for the solution of
	\begin{align}\label{ivp I gamma [2,3]}
		\begin{cases}
			&\cfrac{dC}{dV} = \cfrac{F(V,C;\gamma,z)}{G(V,C;\gamma,z)},\quad V\in[V_1,-\sqrt{\frac{2}{\gamma}}C_8(z)),\\
			&C(V_1) = C_1.
		\end{cases}
	\end{align}
\end{lemma}

\begin{proof}
	To show that $B_s(V)$ is a lower barrier of the solution of \eqref{ivp I gamma [2,3]}, we first verify that the initial point $P_1=(V_1,C_1)$ lies on or above the curve $B_s(V)$. This follows from 
	\begin{align*}
		C_1+\sqrt{\frac{\gamma}{2}}V_1 = \frac{\sqrt{2\gamma(\gamma-1)}}{\gamma+1}-\sqrt{\frac{\gamma}{2}}\frac{2}{\gamma+1} = \frac{\sqrt{2\gamma}}{\gamma+1}(\sqrt{\gamma-1}-1)\geq 0
	\end{align*}
	for any $\gamma\in[2,3]$ where the equality holds when $\gamma=2$. 
	
	Next, we  employ a barrier argument (\textit{cf.}  \eqref{barrier_argument}) to show that $B_s(V)$ is a lower barrier for the solution trajectory of \eqref{ivp I gamma [2,3]}. Specifically, we aim to prove that for any $\gamma\in [2,3]$, $z\in(0,z_M]$, and $V\in[V_1,-\sqrt{\frac{2}{\gamma}}C_8(z))$,
	\begin{align}\label{barrier I gamma 1-2}
		\cfrac{F(V,-\sqrt{\frac{\gamma}{2}}V;\gamma,z)}{G(V,-\sqrt{\frac{\gamma}{2}}V;\gamma,z)}+\sqrt{\frac{\gamma}{2}} 
		>0.
	\end{align}
	By Lemma \ref{sign of dC/dV to the left}, $G(V,-\sqrt{\frac{\gamma}{2}}V;\gamma,z)<0$ for any $\gamma\in [2,3]$, $z\in(0,z_M]$ and $V\in[V_1,-\sqrt{\frac{2}{\gamma}}C_8(z))$, and so it is sufficient to prove that
	\begin{align}
		F(V,-\sqrt{\frac{\gamma}{2}}V;\gamma,z)+\sqrt{\frac{\gamma}{2}}
		G(V,-\sqrt{\frac{\gamma}{2}}V;\gamma,z)<0.
	\end{align}
	By \eqref{F(V,C)} and \eqref{G(V,C)}, we have
	\begin{align*}
		F(V,-\sqrt{\frac{\gamma}{2}}V;\gamma,z)+\sqrt{\frac{\gamma}{2}}G(V,-\sqrt{\frac{\gamma}{2}}V;\gamma,z) 	
		= -m\sqrt{\frac{\gamma}{2}}V^2\Big[(-\gamma+\frac{1}{2})V+\frac{z\gamma V}{2(1+V)}+\frac{(\gamma z-1)(\gamma-1)}{2}-z\gamma\Big].
	\end{align*}
	Since $-m\sqrt{\frac{\gamma}{2}}V^2<0$, it is sufficient to show that 
	\begin{align}\label{barBar P8 only}
		\Barrier(V,z):=(-\gamma+\frac{1}{2})V+\frac{z\gamma V}{2(1+V)}+\frac{(\gamma z-1)(\gamma-1)}{2}-z\gamma>0
	\end{align}
	for any $\gamma\in [2,3]$, $z\in(0,z_M]$ and $V\in[V_1,-\sqrt{\frac{2}{\gamma}}C_8(z))$.
	
	By Lemma \ref{Lemma V_6/8 monoton}, $C_8(z) > C_8(z_M)$ for any $z\in(0,z_M)$. Thus, 
	\begin{align*}
		-\sqrt{\frac{2}{\gamma}}C_8(z) < -\sqrt{\frac{2}{\gamma}}C_8(z_M) = V_8(z_M).
	\end{align*} 
	Therefore, if we can establish the validity of \eqref{barBar P8 only} for all $V\in[V_1,V_8(z_M))$, it trivially holds for all $V\in[V_1,-\sqrt{\frac{2}{\gamma}}C_8(z))$.
	Notice that for any $z\in(0,z_M]$ and $V\in[V_1,V_8(z_M))$, we have
	\begin{align*}
		\frac{\partial \Barrier(V,z)}{\partial V} = -\gamma+\frac{1}{2}+\frac{z\gamma}{2(1+V)^2}\leq -\gamma+\frac{1}{2}+\frac{\gamma z_M}{2(1+V_1)^2}
		=\left(\frac{(\gamma+1)^2}{2(\gamma-1)^2}z_M-1\right)\gamma+\frac{1}{2}.
	\end{align*} 
Given that $z_M(\gamma) = \frac{1}{\gamma+2+2\sqrt{2\gamma}}$ and $\frac{\gamma+1}{\gamma-1} = 1+\frac{2}{\gamma-1}$ are both positive and monotonically decreasing functions in $\gamma$, it follows that $\frac{(\gamma+1)^2}{(\gamma-1)^2}z_M(\gamma)-1$ is also monotone decreasing  in $\gamma$.  
Hence for all $\gamma\in[2,3]$, we have 
	\begin{align*}
		\frac{(\gamma+1)^2}{2(\gamma-1)^2}z_M-1 \leq -\frac{7}{16} \  \ \text{ and } \ \  \left(\frac{(\gamma+1)^2}{2(\gamma-1)^2}z_M-1\right)\gamma+\frac{1}{2} \leq -\frac{7}{8}+\frac{1}{2} = -\frac{3}{8}<0,
	\end{align*} 
	which implies that for any $\gamma\in[2,3]$, $z\in(0,z_M]$ and $V\in[V_1,V_8(z_M))$,
	\begin{align*}
		\Barrier(V,z) > \Barrier(V_8(z_M),z).
	\end{align*}
	
	To finish the proof of \eqref{barrier I gamma 1-2}, it is now sufficient to show that $ \Barrier(V_8(z_M),z)\geq 0$. By \eqref{V68(z_M) C68(z_M)}, $V_8(z_M)$ is independent of $z$ so that 
	\begin{align*}
		\frac{\partial\Barrier(V_8(z_M),z)}{\partial z} = \frac{\gamma V_8(z_M) }{2(1+V_8(z_M))}+\frac{\gamma(\gamma-3)}{2} <0.
	\end{align*}
	Hence, we obtain 
	\begin{align*}
		\Barrier(V_8(z_M),z) \geq \Barrier(V_8(z_M),z_M),
	\end{align*}
	where the equality holds when $z=z_M$. By \eqref{-sqrt(gamma/2)V8=C8} and Lemma \ref{triple roots},
	\begin{align*}
		\Barrier(V_8(z_M),\gamma,z_M) = \cfrac{F(V_8(z_M),C_8(z_M),\gamma,z_M)+\sqrt{\frac{\gamma}{2}}G(V_8(z_M),C_8(z_M),\gamma,z_M)}{-m\sqrt{\frac{\gamma}{2}}V_8(z_M)^2} =0.
	\end{align*}
	In conclusion, we have shown that for $\gamma\in[2,3]$, $z\in(0,z_M]$ and $V\in[V_1,V_8(z_M))$, $\Barrier(V,z) > 0$, thereby completing the proof. 
	\end{proof}

\begin{proposition}\label{onlypass P8}
For any $\gamma\in[2,3]$, the analytic solution to \eqref{ODE} connecting $P_1$ to either $P_6$ or $P_8$, guaranteed by Theorem \ref{existence of zstd}, can only connect to $P_8$. 
\end{proposition}
\begin{proof}
	When $z=z_M$, the points $P_6$ and $P_8$ coincide, rendering any further discussion unnecessary.
	For $z\in(0,z_M)$, by Theorem \ref{existence of zstd}, it is equivalent to showing that the solution trajectory cannot connect to $P_6$.
	
	By \eqref{upper and lower bound for V/C_6/8} and Lemma \ref{Lemma V_6/8 monoton}, for any $\gamma\in[2,3]$ and $z\in(0,z_M)$, it holds 
	\begin{align*}
		C_6(z)< C_8(z_M).
	\end{align*}
	 Moreover, from \eqref{P6} and \eqref{P8}, 
	\begin{align*}
		V_6(z)+\sqrt{\frac{2}{\gamma}}C_8(z) 
		&=\frac{\sqrt{\frac{2}{\gamma}}-1+(\sqrt{\frac{2}{\gamma}}+1)(\gamma-2)z+(\sqrt{\frac{2}{\gamma}}-1)w}{2}<\frac{\sqrt{\frac{2}{\gamma}}-1+(\sqrt{\frac{2}{\gamma}}+1)(\gamma-2)z_M}{2}=0
	\end{align*}
	since $z_M = \frac{1}{(\sqrt{\gamma}+\sqrt{2})^2}$, which shows that $(V_1,V_6(z)) \subset (V_1,-\sqrt{\frac{2}{\gamma}}C_8(z))$. Thus, $P_6$  always lies below the curve $\{B_s(V)\,|\,V\in (V_1,-\sqrt{\frac{2}{\gamma}}C_8(z))\}$ for $z\in(0,z_M)$.
	On the other hand, by Lemma \ref{lower barrier sqrt(gamma/2)V}, the solution trajectory of \eqref{ivp I gamma [2,3]} is always above $B_s(V)$ on $ (V_1,-\sqrt{\frac{2}{\gamma}}C_8(z))$.  
	Therefore, we conclude that the solution cannot connect to $P_6$.
\end{proof}

\section{Solving to left: $P_6$}\label{sec:P6toleft}

In this section, we will refine our analysis and the existence result around $P_6$  for $\gamma\in(1,2]$ by deriving  
an appropriate upper bound for the backwards solution of \eqref{ODE} starting from $P_6$  
to determine a more precise sonic window for $z_{std}(\gamma;P_6)$. In addition, we will prove that, for $\ga\in(1,2]$, the value $z_{std}(\gamma;P_6)$ is unique when it exists.

\subsection{Existence for $\gamma\in (1,2]$}

Recall from Section \ref{subsec:mainresult} that, for $\ga\in(1,2]$, $z_g$ is defined to be the value of $z$ such that $V_4(z_g)=V_6(z_g)$. In this subsection, we rigorously demonstrate that for $\gamma \in (1,2]$ and $z \in [z_m,z_g]$, the analytic solution to \eqref{ODE} backwards from $P_6$  guaranteed by Theorem \ref{analytic} and defined on the domain $[V_1,V_6]$ by Lemma \ref{local solution to the left} is indeed 
a lower solution for $P_6$, which yields an improvement of the range of $z_{std}(\gamma;P_6)$ to $[z_g, z_M]$. 
 We remark that 
$z_m<z_g$  for any $\gamma\in(1,2]$ and $m=1,2$. A proof of this simple fact may be found in  Appendix \ref{zm<zg}.

In what follows, recalling the definitions \eqref{G(V,C)}--\eqref{a_1234&z}, 
we use the notation
\begin{align}
	G(V,C) &= C^2g_1(V)-g_2(V),\label{Ggg}\\
	F(V,C) &= C(C^2f_1(V)-f_2(V))\label{Fff}
\end{align}
where 
\begin{align}
	g_1(V) &= (m+1)V+2mz,\label{g1}\\
	g_2(V) &=V(1+V)(m\gamma z+1+V),\label{g2}\\
	f_1(V) &= 1+\frac{mz}{(1+V)},\label{f1}\\
	f_2(V)&= a_1(1+V)^2-a_2(1+V)+a_3.\label{f2}
\end{align} We rewrite $\frac{dC}{dV} = \frac{F(V,C)}{G(V,C)}$ as 
\begin{align}\label{L}
	\cfrac{d\log{C}}{dV}=\frac{1}{C} \frac{dC}{dV}= \cfrac{C^2f_1(V)-f_2(V)}{C^2g_1(V)-g_2(V)}.
\end{align}
 For $\ga\in(1,2]$, $z_g=z_g(\gamma)$  is defined to be the value such that $$V_4(\gamma,z_g)=V_6(\gamma,z_g).$$ 
 In fact, there exists $\ga_g\in(2,3)$ such that $z_g$ defined in this way is well-defined for $\gamma\in (1,\gamma_g]$, while for $\ga\in(\ga_g,3]$, $V_4$ meets $V_8$ at $z_g$ (defined in an equivalent manner). A detailed discussion of  $\gamma_g$ and $z_g$ is given in  \cite{Lazarus81}. However, for our analysis, we require an understanding of $z_g$ only in the range $\ga\in(1,2]$. The value $z_g$ admits an explicit representation as  
 \begin{align}\label{zg}
	z_g =\begin{cases}
		\cfrac{\sqrt{\gamma^2+(\gamma-1)^2}-\gamma}{\gamma(\gamma-1)} \quad &\text{ when } m=1,\\
		\cfrac{\sqrt{(2\gamma^2-\gamma+1)^2 +2\gamma(\gamma-1)[4\gamma(\gamma-1)+\frac{8}{3}]}-(2\gamma^2-\gamma+1)}{\gamma[4\gamma(\gamma-1)+\frac{8}{3}]}\quad &\text{ when } m=2.
	\end{cases}
\end{align}
We claim that for any $\gamma\in(1,2]$, $z\in[z_m,z_g]$ gives a lower solution for $P_6$. Recalling the definition of a lower solution, \eqref{lower solution def}, 
it is enough to show that 
\begin{align}
		\log C(V_1;\gamma,z,P_6):=-\int_{V_1}^{V_6(z)} \frac{d\log C}{dV} dV + \log C_6(z) <  \log C_1.
\end{align}
Solving this inequality directly is not a trivial task, since the integral is implicit as the integrand involves not only $V$ but also $C$ (\textit{cf.} \eqref{L}).  
To simplify our approach and avoid the complications associated with this implicit integral, we will derive 
an explicit lower bound for $\frac{d\log C}{dV}$ for any $\gamma\in(1,2]$, $z\in[z_m,z_g]$ and $V\in[V_1,V_6(z))$.

\begin{lemma}\label{g1f2-g2f1}
	For any $\gamma\in(1,2]$ and $z\in[z_m,z_g]$, the solution obtained from Theorem \ref{analytic} and Lemma \ref{local solution to the left} satisfies 
	\begin{align}\label{upper bound of solution V4 zg}
			-\int_{V_1}^{V_6(z)}\frac{d\log C}{dV} dV <-\int_{V_1}^{V_6(z)}\frac{f_1(V)}{g_1(V)} dV.
	\end{align}
\end{lemma}
\begin{proof}
	By direct computations, we have 
	\begin{align*}
		\cfrac{d\log{C}}{dV} - \frac{f_1(V)}{g_1(V)} = \cfrac{C^2f_1(V)-f_2(V)}{C^2g_1(V)-g_2(V)}-\frac{f_1(V)}{g_1(V)}
		=-\frac{g_1(V)f_2(V)-f_1(V)g_2(V)}{[C^2g_1(V)-g_2(V)]g_1(V)}.
	\end{align*}
	We will show this 
	function is positive for any $\gamma\in(1,2]$, $z\in[z_m,z_g]$, and $V\in[V_1,V_6(z))$.
	
	By \eqref{upper and lower bound for V/C_6/8} and the fact that for $\ga\in(1,2]$, $z<z_M<\frac{1}{5}$, we have $g_1(V)<0$ for $V\in[V_1,V_6)$. On the other hand, by Lemma \ref{sign of dC/dV to the left}, $G(V,C)=C^2g_1(V)-g_2(V)<0$ for any $V\in[V_1,V_6)$. Therefore, it is sufficient to show that 
	\begin{align*}
		q(V):=g_1(V)f_2(V)-f_1(V)g_2(V) < 0.
	\end{align*}
	Note that $q(V)$ is a cubic polynomial in $V$. Also, $F(V,C)=G(V,C)=0$ at $P_4$, $P_6$ and $P_8$ which implies $C_k^2g_1(V_k)=g_2(V_k)$ and $C_k^2f_1(V_k)=f_2(V_k)$ for $k=4,6,8$. Consequently, $V_4$, $V_6$ and $V_8$ are three roots of $g_1f_2-g_2f_1=0$. Thus,
	\begin{align*}
		q(V) =\frac{m[(m + 1)(\gamma- 1) + 2]}{2} (V-V_4)(V-V_6)(V-V_8).
	\end{align*}
	According to \eqref{upper and lower bound for V/C_6/8}, we have $V_6\leq V_8$ with the equality when $z=z_M$. Therefore, the sign of $q(V)$ depends on the location of $V_4$. If we can show that $V_4\geq V_6$ for $z\in[z_m,z_g]$,  then 
	$q(V)<0$ for any $V\in[V_1,V_6)$. We claim that
	\begin{align}\label{V6<V4}
		V_4(z)\geq V_6(z) \ \text{ for } \  z\in[z_m,z_g]
	\end{align}
	where the equality holds when $z=z_g$. By using \eqref{P4}, we have
	\begin{align*}
		\frac{dV_4(z)}{dz} = \frac{d}{dz} \left( \frac{-2m\gamma z-2}{(m+1)\gamma+1-m} \right)= \frac{-2m\gamma}{(m+1)\gamma+1-m} <0,
	\end{align*}
	which implies $V_4$ is a decreasing function in $z$. By Lemma \ref{Lemma V_6/8 monoton}, $V_6(z)$ is an increasing function in $z$. 
	From the definition of $z_g$ \eqref{zg}, $V_4(z_g)=V_6(z_g)$ for any $\gamma\in(1,2]$. We have shown  \eqref{V6<V4}, which leads to 
	\begin{align}\label{f1/g1 <0}
		\cfrac{d\log{C}}{dV} - \frac{f_1(V)}{g_1(V)} >0.
	\end{align}
	This completes the proof of \eqref{upper bound of solution V4 zg}.
\end{proof}

Motivated by Lemma \ref{g1f2-g2f1}, we define 
\begin{align}
	\delta(V_1;z) := -\int_{V_1}^{V_6(z)}\frac{f_1(V;z)}{g_1(V;z)}dV + \log C_6(z) \label{ln(widetilde(C))} 
\end{align}
where we have used the notations $f_1(V;z)$ and $g_1(V;z)$ for $f_1(V)$ and $g_1(V)$ to emphasize the dependence of $f_1$ and $g_1$ on $z$. By Lemma \ref{g1f2-g2f1}, we have for any $\gamma\in(1,2]$ and $z\in(z_m,z_g]$,
\begin{align*}
	\log C(V_1;\gamma,z,P_6) <	\delta(V_1;z).
\end{align*}
Our next step is to show that for any $\gamma\in(1,2]$, $z=z_g$ gives a lower solution for $P_6$.
\begin{lemma}\label{zglower}
	For any $\gamma\in (1,2]$ and $z=z_g$, 
	\begin{align}\label{zginequality}
	\delta(V_1;z_g)<\log C_1.
	\end{align}
\end{lemma}

\begin{proof}
We first evaluate $\delta(V_1;z)$  
in \eqref{ln(widetilde(C))} by  
using \eqref{f1} and \eqref{g1}  to calculate the integral explicitly as
\begin{align}
	\delta(V_1;z)  
	&= \cfrac{(m^2-m)z+(m+1)}{(2mz-m-1)(m+1)}\log\frac{(m+1)V_6+2mz}{(m+1)V_1+2mz}+\cfrac{mz}{2mz-m-1}\log\frac{1+V_1}{1+V_6}+\log (1+V_6)\label{ln(widetilde(C)(V1))}.
\end{align}
By \eqref{V after shock} and \eqref{C after shock}, 
\begin{align}\label{1+V1, logC1}
	1+V_1 &= \frac{\gamma-1}{\gamma+1}\text{ and } \log C_1 = \frac{1}{2} \log (1+V_1) + \frac{1}{2}\log \frac{2\gamma}{\gamma+1}.
\end{align}
To study the remainder of the expression for $\delta(V_1;z)$, we treat $m=1$ and $m=2$ separately. 
 
When $m=1$,  
by \eqref{P4} and \eqref{zg},
\begin{align}
	V_6(z_g) & =V_4(z_g)= -\frac{\gamma z_g+1}{\gamma} = -z_g-\frac{1}{\gamma}\label{V4m1},\\
	z_g &= \frac{\sqrt{\gamma^2+(\gamma-1)^2}-\gamma}{\gamma(\gamma-1)} = \frac{\gamma-1}{\gamma(\sqrt{\gamma^2+(\gamma-1)^2}+\gamma)}<\frac{1}{2}.\label{zgm1}
\end{align}
Together with \eqref{ln(widetilde(C)(V1))} and \eqref{1+V1, logC1}, we then have
\begin{align*}
    \delta(V_1;z_g)-\log C_1=\frac{1}{2}\log\frac{(1+V_4)(\gamma+1)}{2\gamma}+\frac{1}{2(z_g-1)}\log[\frac{(V_4+z_g)(1+V_1)}{(V_1+z_g)(1+V_4)}] =: \frac{1}{2}(I+II).
\end{align*}
For $I$, we have $1+V_4<1$ and $\frac{\gamma+1}{2\gamma}<1$ for any $\gamma\in (1,2]$. Thus, $I<0$.
From \eqref{zgm1}, $\frac{1}{z_g-1}>-2$. Moreover,  
\begin{align*}
	I+II &< \log\frac{(1+V_4)(\gamma+1)}{2\gamma}-2\log[\frac{(V_4+z_g)(1+V_1)}{(V_1+z_g)(1+V_4)}]\\
	 &=\log [\frac{(\gamma+1)(\gamma-1)}{2\gamma^2}(2-z_g(\gamma+1))^2(1-\frac{1}{(\sqrt{\gamma^2+(\gamma-1)^2}+\gamma)})^3]\\
	 &<\log \frac{(\gamma+1)(\gamma-1)}{\gamma^2}+\log\frac{2-z_g(\gamma+1)}{2}+\log[(2-z_g(\gamma+1))(1-\frac{1}{\sqrt{5}+2})^3]\\
	 &<\log\frac{2(\sqrt{5}+1)^3}{(\sqrt{5}+2)^3}<0
\end{align*}
where we have used that $2-z_g(\gamma+1)<2$ and $\cfrac{1}{(\sqrt{\gamma^2+(\gamma-1)^2}+\gamma)}$ is a decreasing function. This concludes the proof in the case $m=1$.

\

When $m=2$, for any $\gamma\in(1,2]$, by \eqref{V after shock}, \eqref{P4}, \eqref{zM} and \eqref{zg},
\begin{align}
	V_4 & = -\frac{4\gamma z_g+2}{3\gamma-1},\label{V4m2}\\
	z_g &=  \frac{2(\gamma-1)}{\sqrt{(2\gamma^2-\gamma+1)^2 +2\gamma(\gamma-1)[4\gamma(\gamma-1)+\frac{8}{3}]}+(2\gamma^2-\gamma+1)}.\label{zgm2}
\end{align}
We first claim that $z_g<\frac{1}{8}$. By direct computations,  
for any $\gamma\in(1,2]$,
\begin{align*}
	&(16(\gamma-1)-(2\gamma^2-\gamma+1))^2-(2\gamma^2-\gamma+1)^2 -2\gamma(\gamma-1)[4\gamma(\gamma-1)+\frac{8}{3}] \\
	&= 8(\gamma-1)(-\gamma^3-7\gamma^2+\frac{106}{3}\gamma-36)<0
\end{align*}
where the cubic polynomial is negative for any $\gamma\in(1,2]$ as shown in Proposition \ref{poly: lemma 5.2-1}. This then implies 
$$16(\gamma-1)-(2\gamma^2-\gamma+1) < \sqrt{(2\gamma^2-\gamma+1)^2 +2\gamma(\gamma-1)[4\gamma(\gamma-1)+\frac{8}{3}]},$$
 and hence
 $$\frac{2(\gamma-1)}{\sqrt{(2\gamma^2-\gamma+1)^2 +2\gamma(\gamma-1)[4\gamma(\gamma-1)+\frac{8}{3}]}+(2\gamma^2-\gamma+1)} < \frac{1}{8},$$
that is,
\begin{align}
	z_g<\frac{1}{8},\label{zgm2upperbound} 
\end{align}
 as claimed.
By 
\eqref{ln(widetilde(C)(V1))}, \eqref{1+V1, logC1}, \eqref{V4m2} and \eqref{zgm2}, we have
\begin{align*}
    &\delta(V_1;z_g)-\log C_1 \\
 &=\frac{2z_g+3}{2(12z_g-9)}\log [\frac{(3V_4+4z_g)^2}{(3V_1+4z_g)^2}\frac{2\gamma}{\gamma-1}]+\frac{6z_g-9}{4(12z_g-9)}\log\frac{(1+V_4)^4}{(1+V_1)^4} +\frac{7z_g-3}{12z_g-9}\log\frac{\gamma-1}{2\gamma}\\
 &<\frac{2z_g+3}{2(12z_g-9)}\log [\frac{(3V_4+4z_g)^2}{(3V_1+4z_g)^2}\frac{2\gamma}{\gamma-1}]+\frac{6z_g-9}{4(12z_g-9)}\log\frac{(\gamma-1)(1+V_4)^4}{2\gamma(1+V_1)^4}
\end{align*}
where we have used \eqref{zgm2upperbound} in the inequality.  
Now, we  show that both terms are negative. We compute
\begin{align*}
	\frac{(3V_4+4z_g)^2}{(3V_1+4z_g)^2}\frac{2\gamma}{\gamma-1} =(\cfrac{\frac{-12\gamma z_g-6}{3\gamma-1}+4z_g}{\frac{-6}{\gamma+1}+4z_g})^2\frac{2\gamma}{\gamma-1}=\frac{(-6-4z_g)^2}{(-6+4z_g(\gamma+1))^2}\frac{2\gamma(\gamma+1)^2}{(3\gamma-1)^2(\gamma-1)}.
\end{align*}
Note 
that $\frac{2\gamma(\gamma+1)^2}{(3\gamma-1)^2(\gamma-1)}>1$ for $\gamma\in(1,2]$.
Also, $z_g>0$ implies
\begin{align*}
	\frac{(-6-4z_g)^2(\gamma+1)^2}{(-6+4z_g(\gamma+1))^2}>1.
\end{align*}
Hence, 
\begin{align}
	\frac{2z_g+3}{2(12z_g-9)}\log [\frac{(3V_4+4z_g)^2}{(3V_1+4z_g)^2}\frac{2\gamma}{\gamma-1}]<0
\end{align}
because $12z_g-9<0$. As for the second term, we first note that 
\begin{align*}
	\frac{(\gamma-1)(1+V_4)^4}{2\gamma(1+V_1)^4} &= \frac{(\gamma+1)^4}{2\gamma}\cfrac{(3(\gamma-1)-4\gamma z_g)^4}{(\gamma-1)^3(3\gamma-1)^4}\\
	&=\frac{81(\gamma+1)^4}{(3\gamma-1)^4}\frac{\gamma-1}{2\gamma}(1-\frac{8}{3(\sqrt{(\frac{2\gamma^2-\gamma+1}{\gamma})^2 +8[(\gamma-1)^2+\frac{2(\gamma-1)}{3\gamma}]}+\frac{2\gamma^2-\gamma+1}{\gamma}})^4.
\end{align*} 
Moreover, for any $\gamma\in(1,2]$, 
\begin{align*}
	(\frac{2\gamma^2-\gamma+1}{\gamma} )' &= 2-\frac{1}{\gamma^2}>0 \ \ \text{ and } \ \ 
	((\gamma-1)^2+\frac{2(\gamma-1)}{3\gamma})' = \frac{2}{3\gamma^2}+2\gamma-2>0,
\end{align*}
which implies 
\begin{align*}
	&(1-\frac{8}{3(\sqrt{(\frac{2\gamma^2-\gamma+1}{\gamma})^2 +8[(\gamma-1)^2+\frac{2(\gamma-1)}{3\gamma}]}+\frac{2\gamma^2-\gamma+1}{\gamma}})^4< (1-\frac{8}{3(\sqrt{(\frac{7}{2})^2 +8[1+\frac{1}{3}]}+\frac{7}{2}})^4
	< \frac{1}{4}.
\end{align*}
Therefore, for any $\gamma\in (1,2]$, we deduce that  
\begin{align*}
	\frac{(\gamma-1)(1+V_4)^4}{2\gamma(1+V_1)^4} <\frac{81(\gamma-1)(\gamma+1)^4}{8\gamma(3\gamma-1)^4}<1
\end{align*}
where the last inequality is shown in Proposition \ref{poly: lemma 5.2-2}. We then have  
\begin{align}
	\frac{6z_g-9}{4(12z_g-9)}\log\frac{(\gamma-1)(1+V_4)^4}{2\gamma(1+V_1)^4} <0,
\end{align}
thereby completing the proof. 
\end{proof}

\begin{proposition}\label{prop:zglower}
	For any $\gamma\in(1,2]$, $z\in[z_m,z_g]$ gives a lower solution for $P_6$.
\end{proposition}
\begin{proof}
	We first show $\delta(V_1;z)$  
	is an increasing function. 
 From \eqref{ln(widetilde(C))} and \eqref{ln(widetilde(C)(V1))} 
	\begin{align*}
		\frac{d \delta(V_1;z)}{d z} 
		=& -\int_{V_1}^{V_6(z)}\frac{-2m+m(m-1)V}{(1+V)((m+1)V+2mz)^2} dV -\frac{dV_6(z)}{dz}\frac{f_1(V_6(z);z)}{g_1(V_6(z);z)} +\frac{1}{C_6(z)}\frac{dC_6(z)}{dz}>0
	\end{align*}
	where we have used $\frac{dV_6(z)}{dz}=\frac{dC_6(z)}{dz}>0$ by Lemma \ref{Lemma V_6/8 monoton} and $\frac{f_1(V_6(z);z)}{g_1(V_6(z);z)}<0$ by \eqref{f1/g1 <0}. 
	By  Lemma \ref{g1f2-g2f1} and Lemma \ref{zglower}, we then have
	\begin{align*}
		\log C(V_1;\gamma,z,P_6) <\delta(V_1;z) \leq \delta(V_1;z_g) < \log C_1
	\end{align*}
	for any $\gamma\in(1,2]$ and $z\in[z_m,z_g]$. This finishes the proof.
\end{proof}

\subsection{Uniqueness of $z_{std}$ for $P_6$ when $\gamma\in (1,2]$}

Recall $z_{std}(\gamma;P_6)$ is the value of $z$ such that the solution $C(V;\gamma,z_{std}(\gamma;P_6),P_6 )$ (\textit{cf.} \eqref{C(V;gamma,z,P_*)}) satisfies 
	\begin{align}
		\begin{cases}
			&\cfrac{d C}{d V} = \cfrac{F(V,C;\gamma,z_{std})}{G(V,C;\gamma,z_{std})},\\
			& C(V_1) = C_1,\quad C(V_6) = C_6,\\
			& \frac{d C}{d V}(V_6,C_6)=c_1.
		\end{cases}
	\end{align} 
By Proposition \ref{onlypass P8}, we know that for $\gamma\in [2,3]$ the solution can only connect to $P_8$ and therefore, we focus on   
$\gamma\in(1,2]$ for further analysis of $z_{std}(\gamma;P_6)$. 
Within this range of $\gamma$, 
 we demonstrate  the uniqueness of $z_{std}$ for $P_6$. 
This is achieved by showing 
that for any fixed $\gamma\in(1,2]$, the solution trajectories $C(V;\gamma,z,P_6)$ of \eqref{ivp p*} starting from $P_6$ 
do not intersect for different values of $z\in [z_g,z_M]$. In particular, at most one such trajectory can connect to $P_1$.

\begin{lemma}\label{uniquenessP6}
	For any $\gamma\in(1,2]$ fixed, the solution trajectories  $C(V;\gamma,z,P_6)$ do not intersect for different $z\in[z_g,z_M]$ in the interval $[V_1,V_6)$.
\end{lemma}
\begin{proof}
	We argue by contradiction. 
For any fixed $\gamma\in(1,2]$, we write $C'(V,C,z)= \frac{dC(V,C,z)}{dV}$. We suppose that there exist $z_s,z_t\in[z_g,z_M]$ and $z_s<z_t$ such that  $C(V;\gamma,z_s,P_6)$ and  $C(V;\gamma,z_t,P_6)$ intersect at a point $(V_0,C_0)$ where $V_1\leq V_0<V_6(z_s)$. By the continuity of the solution curves with respect to both $V$ and $z$ (see Remark \ref{continuous dependence of the local analytic solution on gamma, z}), we may assume without loss of generality that $(V_0, C_0)$ is the first such intersection point to the left of $P_6(z_s)$ and $P_6(z_t)$. In particular,  there are no other intersection points within the triangular region enclosed  
by the curves $\{(V,C(V;\gamma,z_s,P_6))\,|\, V\in[V_0,V_6(z_s)]\}$, $\{(V,C(V;\gamma,z_t,P_6))\,|\, V\in[V_0,V_6(z_t)]\}$ and $\{(V_6(z),C_6(z))\,|\, z\in[z_s,z_t]\}$. 			
	Then, we have $C(V_0,z) = C_0$ for all $z\in[z_s,z_t]$ so that for all $z\in(z_s,z_t)$,
	\begin{align}\label{dC/dz(V_0,C_0,z)=0}
		\frac{\partial C}{\partial z}(V_0,C_0,z)=0
	\end{align} 	
	and $C'(V_0,C_0,z_s) \leq C'(V_0,C_0,z_t)$. By the Mean Value Theorem, there exists a $\tilde z\in(z_s,z_t)$ such that $\frac{\partial C'}{\partial z} (V_0, C_0,\tilde z)\geq 0$. We will show $\frac{\partial C'}{\partial z}(C_0,V_0,z)< 0$ for all $z\in(z_s,z_t)$ to reach the contradiction. 
	
	By direct computations from the explicit forms of $F$ and $G$ from \eqref{G(V,C)}--\eqref{F(V,C)} and using  \eqref{dC/dz(V_0,C_0,z)=0}, we have, for any  $z\in(z_s,z_t)$ 
	\begin{align*}
		&\frac{\partial C'}{\partial z}(C_0,V_0,z) \\ 
		&=\cfrac{C_0(\frac{mC_0^2}{1+V_0}+\frac{m\gamma(\gamma-3)}{2}(1+V_0)-\frac{m\gamma(\gamma-1)}{2})G(V_0,C_0,z)-(2mC_0^2-m\gamma V_0(1+V_0))F(V_0,C_0,z)}{G^2(V_0,C_0,z)}.
	\end{align*}
	Since $z_s,z_t\in[z_g,z_M]$, it is sufficient to show that 
	\begin{align}
			\frac{\partial C'}{\partial z}(C_0,V_0,z) <0
	\end{align}
	for any $z\in(z_g,z_M)$.
		Substituting \eqref{F(V,C)} and \eqref{G(V,C)} into the above formula and simplifying the expression, we arrive at 
	\begin{align*}
		\frac{\partial C'}{\partial z}(C_0,V_0,z) = \cfrac{mC_0(C_0^2-(1+V_0)^2)[((m-1)V_0-2)C_0^2+(m+1)\frac{\gamma-1}{2}\gamma V_0^2 (1+V_0)] }{G^2(V_0,C_0,z)(1+V_0)}.
	\end{align*}
	Notice that, for any $V_0\in[V_1,V_6)$ and $z\in (z_g,z_M)$
	\begin{align*}
		C_0>0 ,\quad C_0^2-(1+V_0)^2>0,\quad G^2(V_0,C_0,z)(1+V_0)>0.
	\end{align*}
	Thus in order to show $\frac{\partial C'}{\partial z}(C_0,V_0,z)<0$, it is enough to show
	\begin{align*}
		h(V_0,C_0,z):= ((m-1)V_0-2)C_0^2+(m+1)\frac{\gamma-1}{2}\gamma V_0^2 (1+V_0)<0.
	\end{align*}
	By Lemma \ref{lower barrier P6}, $C=\sqrt{\frac{(1+V_6)^2}{V_6}V}$ is a lower barrier of $C(V;\gamma, z, P_6)$ 
	with $z\in (z_g,z_M)$. Hence, 
	\begin{align*}
		h(V_0,C_0,z) &\leq \frac{(1+V_6)^2}{V_6}((m-1)V_0-2)V_0+(m+1)\frac{\gamma-1}{2}\gamma V_0^2 (1+V_0)\\
		&=V_0\Big[\frac{(1+V_6)^2}{V_6}((m-1)V_6-2) +(m+1)\frac{\gamma-1}{2}\gamma V_0 (1+V_0) \Big]\\
		&< \frac{V_0(1+V_6)}{V_6}\Big[(1+V_6)((m-1)V_6-2) +(m+1)\frac{\gamma-1}{2}\gamma V_6^2 \Big]
	\end{align*}
	because $V_0<V_6< V_6(z_M|_{\gamma=2})=-\frac{1}{2}$ by Lemma \ref{Lemma V_6/8 monoton} and $m=1, 2$.  
	Denote
	\begin{align*}
		q(V_6) :=(1+V_6)((m-1)V_6-2) +(m+1)\frac{\gamma-1}{2}\gamma V_6^2 .
	\end{align*}
	Our goal is to show $q(V_6)<0$. Again, by Lemma \ref{Lemma V_6/8 monoton}, we have $-\frac{1}{2}=V_6(z_M|_{\gamma=2})> V_6> V_6(z_g) = V_4(z_g)$ for any $z\in (z_g,z_M)$. Thus, if $q(V_4(z_g))\leq 0$ and $q(-\frac{1}{2})\leq 0$, then $q(V_6)<0$ for all $z\in (z_g,z_M)$ since the coefficient of $V_6^2$ is positive. 
	We first note that 
	\begin{align*}
		q(-\frac{1}{2})=-1-\frac{m-1}{4}+\frac{(m+1)(\gamma-1)\gamma}{8} <0
	\end{align*}
	for any $\gamma\in(1,2]$ and $m=1,2$. For $q(V_4(z_g))$, we claim that $V_4(z_g)$ is a negative zero of $q$. To this end, by using $V_4(z_g)= V_6(z_g)$ and $C_4(z_g)= C_6(z_g)$, we rewrite $q(V_4(z_g))$ as 
	\begin{align*}
	q(V_4(z_g))&=\frac{1}{1+V_4(z_g)} \left[ (1+V_6(z_g))^2 ((m-1) V_4(z_g) -2) + (m+1)\tfrac{\gamma-1}{2}\gamma (1+V_4(z_g)) V_4(z_g)^2\right] \\
	&=\frac{1}{1+V_4(z_g)} \left[ C_4(z_g)^2( (m-1) V_4(z_g) -2) + (m+1)\tfrac{\gamma-1}{2}\gamma (1+V_4(z_g)) V_4(z_g)^2\right] .
	\end{align*}
	Using \eqref{P4}, we replace $ C_4(z_g)$ by $H( V_4(z_g))$ to obtain 
	\begin{align*}
	q(V_4(z_g))=\frac{V_4(z_g)}{ (m+1) V_4(z_g) + 2mz} \tilde q (V_4(z_g)),
	\end{align*}
	where 
	\begin{align*}
	\tilde q (V_4(z_g)) &= (V_4(z_g)+1+m\gamma z_g) ( (m-1) V_4(z_g) -2) + (m+1)\tfrac{\gamma-1}{2}\gamma V_4(z_g) ((m+1) V_4(z_g) + 2mz_g) \\
	&=
	\begin{cases}
	2\gamma (\gamma-1)V_4(z_g)^2 +2 (\gamma(\gamma-1)z_g -1) V_4(z_g) - 2(1+\gamma z_g), & m=1, \\ 
	(\tfrac{9}{2}\gamma(\gamma-1)+1) V_4(z_g)^2 + (2\gamma(3\gamma-2) z_g -1) V_4(z_g) - 2(1+2\gamma z_g), & m=2.
	\end{cases}
	\end{align*}
	It is routine to check that $\tilde q$ has two roots $-z_g -\frac{1}{\gamma}$, $\frac{1}{\gamma-1}$ when $m=1$ and $- \frac{4\gamma z_g+2}{3\gamma-1}$ and $\frac{2}{3\gamma-2}$ when $m=2$. Therefore by \eqref{V4m1} and \eqref{V4m2}, $\tilde q (V_4(z_g))=0$ and $q (V_4(z_g))=0$.  This finishes the proof of $q(V_6)<0$ and  $\frac{\partial C'}{\partial z}(C_0,V_0,z)<0$ for all $z\in (z_g, z_M)$, which contradicts our assumption. 
Therefore we conclude that $C(V;\gamma,z_s,P_6)$ can not intersect $C(V;\gamma,z_t,P_6)$ if $z_s \neq z_t$. 
	\end{proof}

\section{Solving to left: $P_8$}\label{sec:P8toleft}
In this section, we again employ suitable barrier functions to delineate a more precise range for $z$ in which $z_{std}(\gamma;P_8)$ resides. By Proposition \ref{onlypass P6}, we may focus on $\gamma\in (\gaSix,3]$ for further analysis of $z_{std}(\gamma;P_8)$.

\subsection{Conditional existence for $\gamma\in(\gaSix, \gaOne]$}
As described in Section \ref{subsec:mainresult}, we define $\gaOne$ to be the value such that $C_1 = \sqrt{-V_1}$. A simple calculation then establishes that
\begin{align}\label{gaOne}
	\gaOne = 1+\sqrt{2}.
\end{align} 
By Lemma \ref{C1V1}, $P_1$ is below $B_1(V)=\sqrt{-V}$ when $\gamma\leq\gaOne$. We will show that the solution trajectory, originating at $P_1$ and propagated by \eqref{ODE}, remains below $C=B_1(V)$ within a specific range of $V$, the exact bounds of which will be established subsequently, when $\gamma\leq\gaOne$.

For each $\ga$, we define $z_1$ to be the value such that $C_8(z_1) = \sqrt{-V_8(z_1)}$. Then 
\begin{align}\label{z1}
	z_1 = \frac{\sqrt{5}-1}{2(1+\sqrt{5}+\gamma)}.
\end{align}
Since $V_8(z_1)=C_8(z_1)-1$ and $1-C_8^2(z_1)=C_8(z_1)$, 
\begin{align}
	V_8(z_1) = \frac{\sqrt{5}-3}{2},\quad C_8(z_1) = \frac{\sqrt{5}-1}{2}.\label{V8C8z1}
\end{align}
For any fixed $\gamma\in(1,3]$, we write $C_8=C_8(z)$, $C'_8(z) = \frac{d C_8(z)}{dz}$ and $C''_8(z) = \frac{d^2 C_8(z)}{dz^2}$. We first show the concavity of $C_8^2$ with respect to $z$, which will be crucial for subsequent arguments. 

\begin{lemma}\label{4C8C8''+(C8')^2<0}
	For any fixed $\gamma\in(1,3]$, and any $z\in(0,z_M)$, 
	\begin{align*}
		C_8C''_8+(C'_8)^2<0.
	\end{align*}
\end{lemma}
\begin{proof}
	By direct computations, we have 
	\begin{align*}
		C''_8 = -\frac{4\gamma}{w^3}<0.
	\end{align*}
	For any $\gamma\in (1,3]$ and $z\in(0,z_M)$, recalling \eqref{P6}, we have
	\begin{align*}
		1+(\gamma-2)z+w = 2C_6+2w>2 w .
	\end{align*}
	Hence,
		\begin{align*}
			4C_8C_8''+4(C_8')^2 &= (1+(\gamma-2)z+w)\frac{-8\gamma}{w^3} + 4(C_8')^2
			\leq  (2C'_8)^2-\frac{16\gamma}{w^2} = (2C'_8-\frac{4\sqrt{\gamma}}{w})(2C'_8+\frac{4\sqrt{\gamma}}{w}).
	\end{align*}
	By Lemma \ref{Lemma V_6/8 monoton}, $C_8'<0$ and thus $2C'_8-\frac{4\sqrt{\gamma}}{w}<0$. We will show $2C'_8+\frac{4\sqrt{\gamma}}{w}>0$. By using $0<w<1$ (\textit{cf.} Remark \ref{bounds of w(z)}), we see that 
	\begin{align*}
		(2C'_8+\frac{4\sqrt{\gamma}}{w})w &= 4\sqrt{\gamma}-\gamma-2+(\gamma-2)w+(\gamma-2)^2 z>4\sqrt{\gamma}-\gamma-3= (3-\sqrt{\gamma})(\sqrt{\gamma}-1)>0
	\end{align*}
	for any $\gamma\in (1,3]$, thereby completing the proof.
\end{proof}
	
\begin{lemma}\label{leftgammabar}
	For any $\gamma\in (\gaSix,\gaOne]$, $z\in(0,z_1(\gamma)]$ gives a upper solution for $P_8$. 
\end{lemma}
\begin{proof}
	In view of Lemma \ref{uppersolutionP8}, we may determine the existence of $\overline{z}(\gamma;P_8)$ such that for all $z\leq \overline{z}(\gamma;P_8)$, it holds that $C_8(\gamma,z)\geq C_1(\gamma)$ and each $z\in (0, \overline{z}(\gamma;P_8)]$ gives rise to an upper solution for $P_8$. However, the expression for such $\overline{z}(\gamma;P_8)$, derived from the equation $C_8(\gamma, \overline{z}(\gamma;P_8))=C_1(\gamma)$, is intricate and inconvenient. For that reason,  
	we  use 
	another value of $z$, which has a closed-form representation. 
	Specifically, we define $\tilde{z}_m(\gamma)$  
	to be the value satisfying
	\begin{align}\label{C8zm'}
		C_8(\gamma,\tilde{z}_m(\gamma)) = C_1(\gaOne)=\sqrt{2-\sqrt{2}}
	\end{align}
	so that 
	\begin{align*}
		\tilde{z}_m(\gamma)=\frac{-2^{\frac{1}{4}} \sqrt{1+\sqrt{2}}+2^{\frac{3}{4}}\sqrt{1+\sqrt{2}}}{2\sqrt{2}+ 2^{\frac{5}{4}} \sqrt{1+\sqrt{2}}+\gamma}.
	\end{align*}
	It is easy to check that 
	for any $\gamma\in (\gaSix,\gaOne]$,
	\begin{align}\label{lower bound of zm'}
		\tilde{z}_m(\gamma) \geq \tilde{z}_m(\gaOne)>\frac{2}{25}.
	\end{align}
	By Lemma \ref{Lemma V_6/8 monoton} and Lemma \ref{C1V1}, for any $\gamma\in (\gaSix,\gaOne]$ and $z\leq \tilde{z}_m(\gamma)$, we have
	\begin{align*}
		C_8(\gamma,\tilde{z}_m(\gamma)) = 
		C_1(\gaOne) \geq C_1(\gamma) =
		 C_8(\gamma,\overline{z}(\gamma;P_8)).
	\end{align*}
	It then follows that $\tilde{z}_m(\gamma)\leq \overline{z}(\gamma;P_8)$ and hence, each $z\in(0,\tilde{z}_m(\gamma)]$ serves as an upper solution for $P_8$.

To complete the proof, we must demonstrate that each $z\in(\tilde{z}_m(\gamma),z_1(\gamma)]$ serves as an upper solution for $P_8$. To achieve this, we employ the barrier function $B_1(V)=\sqrt{-V}$. 
	It suffices to establish that,  for any $\gamma\in (\gaSix,\gaOne]$ and $z\in(\tilde{z}_m(\gamma),z_1(\gamma)]$, $B_1(V)=\sqrt{-V}$ is  
	a lower barrier for the solution $C(V)$ of \eqref{ivp p*} with $P_*=P_8$. 
	By Lemma \ref{sign of dC/dV to the left}, $\frac{d C}{dV}<0$ to the left of the triple point $P_8$. Therefore, $$C(V)> C_8=B_1(-C_8^2)\geq B_1(V)\quad \text{ for } V\in[-C_8^2,V_8)$$   
and  $C(V_8)>B_1(V_8)$ for $z\in(\tilde{z}_m(\gamma),z_1(\gamma))$, 
$C(V_8)=B_1(V_8)$ for $z=z_1(\gamma)$,  and  
\begin{align}\label{ineq of c1(z1)}	
\frac{d C}{d V}|_{V=V_8(z_1),C=C_8(z_1),z=z_1}<- \frac{1}{2\sqrt{-V_8(z_1)}},
\end{align}
which 
guarantees the existence of $\overline{V}<V_8$ sufficiently close to $V_8$ 
so that also for $z=z_1$, the solution $C(V)$ to \eqref{ivp p*} enjoys 
$C(V)>B_1(V)$ for $V\in[\overline{V},V_8)$. The proof of  \eqref{ineq of c1(z1)} is given  
in Lemma \ref{Lemma: ineq of c1(z1)}. 
	Hence, by the barrier argument \eqref{barrier_argument}, we want to  
	show  
	\begin{align}
		\frac{F(V,\sqrt{-V};\gamma,z)}{G(V,\sqrt{-V};\gamma,z)} + \frac{1}{2\sqrt{-V}}<0. 
	\end{align} 
	Since this inequality is nothing but \eqref{barrier inequality} with $k(\gamma,z)$ replaced by $1$, and $G(V,\sqrt{-V};\gamma,z)< 0$ for any $V\in[V_1,-C_8^2(z))$ by Lemma \ref{sign of dC/dV to the left}, our goal is to show the positivity of the following function (\textit{cf.} \eqref{def_Bk}):
	\begin{align}
		\Barrier_1(V,z,m):= 
		 (m-1-m\gamma)V^2+(-3+2m-m\gamma+m(\gamma-2)\gamma z)V-m\gamma z-1+\frac{2mz}{1+V}\label{BP81}
	\end{align} 
	for each $\gamma\in(\gaSix,\gaOne]$, $z\in(\tilde{z}_m(\gamma),z_1(\gamma)]$, $V\in[V_1,-C_8^2(z))$ and $m=1,2$. Our strategy is the following: 
	\begin{enumerate}
		\item For $m=1,2$, we will show that it is enough to check the sign of $(1-C^2_8(z))\Barrier_1(-C^2_8(z),z,m)$.
		\item For $m=1,2$, we will show that $\frac{d^2}{dz^2}[(1-C^2_8(z))\Barrier_1(-C^2_8(z),z,m)]<0$ so that $(1-C^2_8(z))\Barrier_1(-C^2_8(z),z,m)$ is a concave function. Since $\Barrier_1(-C^2_8(z_1),z_1,m)=0$ by the definition of $z_1$,  it is sufficient to check the sign of $(1-C^2_8(\tilde{z}_m(\gamma)))\Barrier_1(-C^2_8(\tilde{z}_m(\gamma),\tilde{z}_m(\gamma),m)>0$. 	\end{enumerate}
For any fixed $\gamma\in(\gaSix,\gaOne]$, we write $C_8=C_8(z)$, $C'_8(z) = \frac{d C_8(z)}{dz}$ and $C''_8(z) = \frac{d^2 C_8(z)}{dz^2}$.

\

	$\underline{\text{Step 1}}:$ When $m=1$, for each 
	 $z\in(\tilde{z}_m(\gamma),z_1(\gamma)]$ and $V\in[V_1,-C_8^2(z))$, using $-\gamma(2V+1) \le -\gamma(2V_1+1) $ and $-\frac{2z}{(1+V)^2} < -2 z$, we derive  
	\begin{align*}
		\frac{\partial \Barrier_1(V,z,1)}{\partial V} 
		&<-\gamma(2V_1+1)-1+[(\gamma-2)\gamma-2]z 
		=\frac{-(\gamma-1)^2}{\gamma+1}+[(\gamma-2)\gamma-2]z<0
	\end{align*}
	where 
	we have used $(\gamma-2)\gamma-2<0$ for any $\gamma\in(\gaSix,\gaOne]$. Since $ \Barrier_1(V,z,1)$ is a decreasing function in $V$ and $1-C^2_8(\gamma,z)>0$, it is sufficient to check the sign of $(1-C^2_8(z))\Barrier_1(-C^2_8(z),z,1)$. 
	\
	
	When $m=2$, we compute $(1+V)\Barrier_1(V,z,2)$ to obtain
	\begin{align*}
		(1+V)\Barrier_1(V,z,2) = (1-2\gamma)V^3+2(1-2\gamma+(\gamma-2)\gamma z)V^2+(-2\gamma+2(\gamma-3)\gamma z)V-2\gamma z-1+4z.
	\end{align*}
	We next show that $	\frac{\partial}{\partial V}[(1+V)\Barrier_1(V,z,2)]<0$. Note that 
	\begin{align*}
		\frac{\partial}{\partial V}[(1+V)\Barrier_1(V,z,2)] = 3(1-2\gamma)V^2+4(1-2\gamma+(\gamma-2)\gamma z)V+(-2\gamma+2(\gamma-3)\gamma z),
	\end{align*}
	which is a quadratic polynomial of $V$. If the discriminant of the polynomial is negative for any $\gamma\in (\gaSix,\gaOne]$, $z\in(\tilde{z}_m(\gamma),z_1(\gamma)]$, then $\frac{d}{dV}[(1+V)\Barrier_1(V,z)]$ will always be negative because $3(1-2\gamma)<0$. The discriminant of the polynomial is given by
	\begin{align*}
		\Delta &= 16(1-2\gamma+(\gamma-2)\gamma z)^2-12(1-2\gamma)(-2\gamma+2(\gamma-3)\gamma z)=: 8 p(z)
	\end{align*} where 
	\begin{align*}
		p(z):=2(\gamma-2)^2\gamma^2z^2+(1-2\gamma)\gamma(\gamma+1) z+(1-2\gamma)(2-\gamma).
	\end{align*}
	When $\gamma=2$, it is clear that $p(z)<0$. When $\gamma\neq 2$, $p(z)$ is a quadratic polynomial in $z$. It has a local minimum at $z = \frac{(2\gamma-1)(\gamma+1)}{4(\gamma-2)^2\gamma} > \frac{1}{4}>z_1$. 
	Thus, by \eqref{lower bound of zm'}, to verify the negativity of $\Delta$, it is sufficient to check the negativity of $p(\frac{2}{25})$. This condition is checked in Proposition \ref{poly: Lemma 6.2-1}. Therefore, we have shown 
	\begin{align*}
		\frac{\partial}{\partial V}[(1+V)\Barrier_1(V,z,2)] <0,
	\end{align*}
	and hence, to show that $\Barrier_1(V,z,2)>0$, it is enough to check the sign of $(1-C^2_8(z))\Barrier_1(-C^2_8(z),z,2)$.

\smallskip 

$\underline{\text{Step 2}}:$ Our next goal is to show $(1-C^2_8(z))\Barrier_1(-C^2_8(z),z,m)> 0$ for any  
$z\in(\tilde{z}_m(\gamma),z_1(\gamma)]$, $V\in[V_1,-C_8^2(z))$, and $m=1,2$. We will first show $(1-C^2_8(z))\Barrier_1(-C^2_8(z),z,m)$ is a concave function in $z$.\\
	For notational convenience, we will write 
	\begin{align}
		\Jm(z) := -C^2_8(z).
	\end{align}
	By using Lemma \ref{Lemma V_6/8 monoton}, \eqref{V8C8z1} and \eqref{C8zm'}, we obtain
	\begin{align}\label{range of Jm(z)}
	\sqrt{2}-2=\Jm(\tilde{z}_m(\gamma))<	\Jm(z)\leq \Jm(z_1) = -\frac{3-\sqrt{5}}{2}.
	\end{align}
	Note that by using Lemma \ref{Lemma V_6/8 monoton} and Lemma \ref{4C8C8''+(C8')^2<0}, we obtain
	\begin{align}
		\Jm'(z) &= -C_8C'_8(z)>0,\label{1st dev of Jm(z)}\\
		 \Jm''(z)&= -[C_8C''_8(z)+(C'_8)^2]>0.\label{2nd dev of Jm(z)}
	\end{align}
	We rewrite $(1-C^2_8(z))\Barrier_1(-C^2_8(z),z,m)$ as 
	\begin{align*}
		(1+\Jm(z))\Barrier_1(\Jm(z),z,m)
		=&\,(m-1-m\gamma)\Jm^3(z)+(3m-4-2m\gamma+m(\gamma-2)\gamma z)\Jm^2(z)\\
		&+(2m-4-m\gamma+m(\gamma-3)\gamma z)\Jm(z)+m(2-\gamma) z-1,
	\end{align*}
	and compute the second $z$ derivative to obtain 
	\begin{align*}
		\frac{d^2}{dz^2}[(1+\Jm(z))\Barrier_1(\Jm(z),z,m)] := \Jm''(z) A(\gamma,z,m)+\Jm'(z) B(\gamma,z,m),
	\end{align*}
	where
	\begin{align*}
		A(\gamma,z,m) =& \,3(m-1-m\gamma)\Jm^2(z)+2[3m-4-2m\gamma+m(\gamma-2)\gamma z]\Jm(z)+2m-4-m\gamma+m(\gamma-3)\gamma z,\\
		B(\gamma,z,m) =& \,6(m-1-m\gamma)\Jm(z)\Jm'(z)+2[3m-4-2m\gamma+m(\gamma-2)\gamma z]\Jm'(z)\\
		&+4m(\gamma-2)\gamma \Jm(z)+2m(\gamma-3)\gamma.
	\end{align*}
	We claim $A(\gamma,z,m)$ and $B(\gamma,z,m)$ are negative. We first check $B(\gamma,z,m)<0$. We decompose $B(\gamma,z,m)$ into two parts
	\begin{align*}
		B(\gamma,z,m) = \Jm'(z)B_1(\gamma,z,m) + B_2(\gamma,z,m),
	\end{align*}
	where 
	\begin{align*}
		B_1(\gamma,z,m) &=6(m-1-m\gamma)\Jm(z)+6m-8-4m\gamma+2m(\gamma-2)\gamma z,\\
		B_2(\gamma,z,m) & 
		= 2m\gamma[ ( 2(\gamma-2) \Jm(z)+(\gamma-3)].
	\end{align*}
	For $B_1(\gamma,z,m)$, by using $m-1-m\gamma<0$, \eqref{range of Jm(z)}, $|(\gamma-2)\gamma|\leq 1$, and $z<z_M<\frac{1}{5}$, we obtain
	\begin{align*}
		B_1(\gamma,z,m) &< 6(m-1-m\gamma)\Jm(\tilde{z}_m(\gamma))+6m-8-4m\gamma+\frac{2m}{5} \\
		&= 4 - 6 \sqrt{2} - \frac{m(28-30 \sqrt{2}) }{5} + m(8  - 6 \sqrt{2} )\gamma <0
	\end{align*}
	for any $m=1,2$ and $\gamma\in(\gaSix,\gaOne]$.

	For $B_2(\gamma,z,m)$, when $\gamma\geq 2$, it is clear that $B_2(\gamma,z,m)<0$. When $\gamma\in(\gaSix,2)$, by using \eqref{range of Jm(z)},
	\begin{align*}
		B_2(\gamma,z,m)  
		< 2m\gamma (2(\gamma-2)\Jm(\tilde{z}_m(\gamma))+\gamma-3) = 2m\gamma[(2\sqrt{2}-3)\gamma+5-4\sqrt{2}]<0.
	\end{align*}
	Hence, we have shown that $B_2(\gamma,z,m)<0$ for any $m=1,2$, $\gamma\in(\gaSix,\gaOne]$ and $z\in(\tilde{z}_m(\gamma),z_1(\gamma)]$.

	Regarding $A(\gamma,z,m)$, we observe that
	\begin{align*}
		\frac{d A(\gamma,z,m)}{dz} = \Jm'(z)B_1(\gamma,z,m)+\frac{1}{2} B_2(\gamma,z,m) <0.
	\end{align*}
	Therefore, in order to show $A(\gamma,z,m)<0$, it is enough to verify $A(\gamma,\tilde{z}_m(\gamma),m)<0$. When $m=1$, by using \eqref{range of Jm(z)} and $B_2(\gamma,z,1)<0$, we have
	\begin{align*}
		A(\gamma,\tilde{z}_m(\gamma),1) &= -3\gamma\Jm^2(\tilde{z}_m(\gamma))+2(-1-2\gamma)\Jm(\tilde{z}_m(\gamma))-2-\gamma+\frac{\tilde{z}_m(\gamma)}{2}B_2(\gamma,\tilde{z}_m(\gamma),1)\\
		&<(8\sqrt{2}-11)\gamma+2(1-\sqrt{2})<0.
	\end{align*}
	When $m=2$, by using \eqref{range of Jm(z)}, \eqref{lower bound of zm'} and $(2\sqrt{2}-3)\gamma+5-4\sqrt{2}<0$, we have
	\begin{align*}
		A(\gamma,\tilde{z}_m(\gamma),2) &= 3(1-2\gamma)\Jm^2(\tilde{z}_m(\gamma))+2(2-4\gamma)\Jm(\tilde{z}_m(\gamma))-2\gamma+2\gamma[2(\gamma-2)\Jm(\tilde{z}_m(\gamma))+\gamma-3]\tilde{z}_m(\gamma)\\
		&<(16\sqrt{2}-22)\gamma+10-8\sqrt{2}+\frac{4\gamma[(2\sqrt{2}-3)\gamma+5-4\sqrt{2}]}{25}\\
		&=\frac{2}{25}\Big[(4\sqrt{2}-6)\gamma^2+(192\sqrt{2}-265)\gamma+125-100\sqrt{2}\Big]=:p(\gamma).
	\end{align*}
	Since $p(\gamma)$ has a global maximum at $\gamma = \frac{192\sqrt{2}-265}{12-8\sqrt{2}}>3>\gaOne$ and $p(\gaOne)=\frac{484-346\sqrt{2}}{25}<0$, $p(\gamma)<0$ for any $\gamma\in(\gaSix,\gaOne]$. We conclude that  
	$A(\gamma,z,m)<0$.
	
	Hence, $(1+\Jm(z))\Barrier_1(\Jm(z),z,m)$ is a concave function in $z$. It is then enough to check the sign for the function at two ends points of $z$. By definition of $z_1$ (\textit{cf.} \eqref{z1}), $(1+\Jm(z_1))\Barrier_1(\Jm(z_1),z_1,m)=0$. 
	At $z= \tilde{z}_m(\gamma)$, since 
	$1+\Jm(\tilde{z}_m(\gamma)= \sqrt2 -1 >0$, we only need to show $\Barrier_1(\Jm(\tilde{z}_m(\gamma)),\tilde{z}_m(\gamma),m)>0$.  
By direct computations,
\begin{align*}
&\Barrier_1(\Jm(\tilde{z}_m(\gamma)),\tilde{z}_m(\gamma),m)\\ &=m(3\sqrt{2}-4)\gamma+(1-\sqrt{2})(2m-1) +m[(\sqrt{2}-2)\gamma^2+(3-2\sqrt{2})\gamma+2\sqrt{2}+2]\tilde{z}_m(\gamma)\\
&>m(3\sqrt{2}-4)\gamma+(1-\sqrt{2})(2m-1) +\frac{2m}{25}[-(2-\sqrt{2})\gamma^2+(3-2\sqrt{2})\gamma+2\sqrt{2}+2]>0,
\end{align*}
where we have used $(\sqrt{2}-2)\gamma^2+(3-2\sqrt{2})\gamma+2\sqrt{2}+2>0$ for each $\gamma\in(\gaSix,\gaOne]$ and \eqref{lower bound of zm'} in the second line, while the positive sign of the last inequality is shown in Proposition \ref{poly: Lemma 6.2-2} and Proposition \ref{poly: Lemma 6.2-3} for $m=1$ and $m=2$ respectively.
\end{proof}

\subsection{Conditional existence for $\gamma\in(\gaOne, 3]$}

In this subsection, we will employ the barrier function $B_{\frac{3}{2}}(V)=\sqrt{-\frac{3}{2}V}$ to delineate a narrower and more precise range for the potential location of $z_{std}(\gamma;P_8)$. With 
the choice of the barrier function, we define $z_2(\gamma)$ to be the value such that $P_8$ lies on the curve $C=B_{\frac{3}{2}}(V)$: 
\begin{align}\label{z2}
	z_2 = \frac{\sqrt{33}-3}{6+2\sqrt{33}+4\gamma}.
\end{align}
Since $V_8(z_2)=C_8(z_2)-1$ and $1-\frac{2}{3}C_8^2(z_2)=C_8(z_2)$,
\begin{align}\label{C8z2}
	V_8(z_2) = \frac{\sqrt{33}-7}{4},\quad C_8(z_2) = \frac{\sqrt{33}-3}{4}.
\end{align}

\begin{lemma}\label{leftlargergammabar}
	For any $\gamma\in (\gaOne,3]$, any $z\in(0,z_2(\gamma)]$ gives an upper solution for $P_8$.
\end{lemma}

\begin{proof}
	Similar to our approach in Lemma \ref{leftgammabar}, we define $\overline{z}(\gamma;P_8)$ as  the solution to $C_8(\gamma, \overline{z}(\gamma;P_8))=C_1(\gamma)$ and 
	we further introduce $\hat{z}_m(\gamma)$, which satisfies the equation:
		\begin{align}\label{C8hatzm}
			C_8(\ga, \hat{z}_m(\ga)) &= C_1(3)= \frac{\sqrt{3}}{2}
		\end{align}
		so that 
		\begin{align*}
			\hat{z}_m(\gamma)=\frac{\sqrt{3}}{12+8\sqrt{3}+2\gamma}.
		\end{align*}
		It is easy to check that for any $\gamma\in(\gaOne,3]$,
		\begin{align}\label{lower bound of hatz_m}
			\frac{1}{8}>z_M(2)>z_M(\gamma)>\hat{z}_m(\gamma) \geq \hat{z}_m(3)>\frac{1}{20}.
		\end{align}
		By Lemma \ref{Lemma V_6/8 monoton} and Lemma \ref{C1V1}, for any $\gamma\in (\gaOne,3]$ and $z\leq \hat{z}_m(\gamma)$, we have
		\begin{align*}
			C_8(\gamma,\hat{z}_m(\gamma)) = C_1(3) \geq C_1(\gamma) = C_8(\gamma, \overline{z}(\gamma;P_8)),
		\end{align*}
		and it follows that $\hat{z}_m(\gamma)\leq \overline{z}(\gamma;P_8)$. Consequently, each $z\in(0,\hat{z}_m(\gamma)]$ serves as an upper solution for $P_8$.
		
		To conclude the proof, we need to establish that each $z\in(\hat{z}_m(\gamma),z_2(\gamma)]$ gives an upper solution for $P_8$. To this end, we will employ the barrier function $B_{\frac{3}{2}}(V)$.  
	Hence, it suffices to establish that,  for any $\gamma\in (\gaOne,3]$ and $z\in(\hat{z}_m(\gamma),z_2(\gamma)]$, $B_{\frac{3}{2}}(V)$ is a lower barrier for the solution $C(V)$ of \eqref{ivp p*} with $P_*=P_8$. 
By Lemma \ref{sign of dC/dV to the left}, $\frac{d C}{dV}<0$ to the left of the triple point $P_8$. Therefore, $C(V)> B_{\frac{3}{2}}(V)$ for $V\in[-\frac{2}{3}C_8^2,V_8]$ and $z\in(\hat{z}_m(\gamma),z_2(\gamma))$, while $C(V_8)= B_{\frac32} (V_8) $ for $z=z_2(\gamma)$ and it satisfies  
\begin{align}\label{ineq of c1(z2)}	
\frac{d C}{d V}|_{V=V_8(z_2),C=C_8(z_2),z=z_2}<- \frac{1}{2\sqrt{-V_8(z_2)}}, 
\end{align}
so that 
$C(V)>B_{\frac32}(V)$ for $V\in[\overline{V},V_8)$ for some $\overline{V}<V_8$.  
The proof of  \eqref{ineq of c1(z2)} is given in Lemma \ref{Lemma: ineq of c1(z2)}. 
	Now, by using the barrier argument (\textit{cf.}  \eqref{barrier_argument}), it is sufficient to show that
	\begin{align}
		 \frac{F(V,\sqrt{-\frac{3}{2}V};\gamma,z)}{G(V,\sqrt{-\frac{3}{2}V};\gamma,z)}+\frac{1}{2}\sqrt{\frac{-3}{2V}}<0.
	\end{align}
		We observe that this inequality is \eqref{barrier inequality} with $k(\gamma,z)$ replaced by $\frac{3}{2}$. As $G(V,\sqrt{-\frac{3}{2}V};\gamma,z)< 0$ for any $V\in[V_1,-\frac{2}{3}C_8^2(z))$ by Lemma \ref{sign of dC/dV to the left}, it suffices to show that 	for any $\gamma\in(\gaOne,3]$, $z\in(\hat{z}_m(\gamma),z_2(\gamma)]$, $V\in[V_1,-\frac{2}{3}C_8^2(z))$, and $m=1,2$, the following function (\textit{cf.} \eqref{def_Bk}) is positive.
	\begin{align}\label{BP82}
		\Barrier_{\frac32}(V,z,m):=
		 (m-1-m\gamma)V^2+(-4+2m+\frac{m+1}{2}-m\gamma+m(\gamma-2)\gamma z)V-m \gamma z-1+\frac{3mz}{1+V}.
	\end{align}
The strategy is the following: 
	\begin{enumerate}
		\item For $m=1,2$, we will show that it is enough to check the sign of $(1-\frac{2}{3}C^2_8(z))\Barrier_{\frac32}(-\frac{2}{3}C^2_8(z),z,m)$.
		\item For $m=1,2$, we will show that $\frac{d^2}{dz^2}[(1-\frac{2}{3}C^2_8(z))\Barrier_{\frac32}(-\frac{2}{3}C^2_8(z),z,m)]<0$ so that $(1-\frac{2}{3}C^2_8(z))\Barrier_{\frac32}(-\frac{2}{3}C^2_8(z),z,m)$ is a concave function. Since $\Barrier_{\frac32}(-\frac{2}{3}C^2_8(z_2),z_2,m)=0$ by the definition of $z_2$, it is sufficient to check the sign of $(1-\frac{2}{3}C^2_8(\hat{z}_m(\gamma)))\Barrier_{\frac32}(-\frac{2}{3}C^2_8(\hat{z}_m(\gamma)),\gamma,\hat{z}_m(\gamma),m)>0$.
	\end{enumerate}
	
	\smallskip 
	
	$\underline{\text{Step 1}}:$
	First of all, by Lemma \ref{C1V1},
	\begin{align}\label{V1>sqrt(2)-2}
		V_1(\gamma) \geq V_1(\gaOne) = \sqrt{2}-2.
	\end{align}
	When $m=1$, 
	by using \eqref{V1>sqrt(2)-2} and $0<1+V<1$, we have
	\begin{align}
		\frac{\partial \Barrier_{\frac32}(V,z,1)}{\partial V} 
		&< -2\gamma V_1-1-\gamma+[(\gamma-2)\gamma -3]z<(3-2\sqrt{2})\gamma-1+[(\gamma-2)\gamma -3]z<0. \label{sign_BV}
	\end{align} 
	Therefore, to establish the positivity of $\Barrier_{\frac32}(V,z,1)$, it suffices to show
	\begin{align*}
		(1-\frac{2}{3}C^2_8(z)) \Barrier_{\frac32}(-\frac{2}{3}C^2_8(z),z,1)\geq 0.
	\end{align*}

	As for $m=2$, we compute
	\begin{align*}
		(1+V)\Barrier_{\frac32}(V,z,2) = (1-2\gamma)V^3+(\frac{5}{2}-4\gamma+2(\gamma-2)\gamma z)V^2+(\frac{1}{2}-2\gamma+2(\gamma-3)\gamma z)V+2(3-\gamma) z-1.
	\end{align*}
For any $\gamma\in(\gaOne,3]$, $z\in(\hat{z}_m(\gamma),z_2(\gamma)]$ and $V\in[V_1,-\frac{2}{3}C_8^2(z))$, by using \eqref{lower bound of hatz_m}, 
	\begin{align*}
		\frac{\partial^2}{\partial V^2}[(1+V)\Barrier_{\frac32}(V,z,2)] &= 6(1-2\gamma)V+5-8\gamma+4(\gamma-2)\gamma z \\
		&<\frac{12(2\gamma-1)}{\gamma+1}+5-8\gamma+\frac{(\gamma-2)\gamma}{2}=\frac{\gamma^3-17\gamma^2+40\gamma-14}{2(\gamma+1)}<0
	\end{align*}
	where the negative sign of the cubic polynomial is shown in Proposition \ref{poly: Lemma 6.3-1}.
 Thus, $(1+V)\Barrier_{\frac32}(V,z,2)$ is a concave function in $V$. It is then enough to check the signs of $\Barrier_{\frac32}(V_1,z,2)$ and $\Barrier_{\frac32}(-\frac{2}{3}C^2_8(z),z,2)$. 
	We now compute $\Barrier_{\frac32}(V_1,z,2)$,
	\begin{align*}
		\Barrier_{\frac32}(V_1,z,2) 
		&=\frac{3\gamma (\gamma-3)}{(\gamma+1)^2} + \frac{6(\gamma^2(3-\gamma)+\gamma+1)}{\gamma^2-1}z\\ 
		&> \frac{3\gamma (\gamma-3)}{(\gamma+1)^2} + \frac{6(\gamma^2(3-\gamma)+\gamma+1)}{20(\gamma^2-1)}=\frac{3(-\gamma^4+12\gamma^3-36\gamma^2+32\gamma+1)}{10(\gamma-1)(\gamma+1)^2}>0
	\end{align*}
	where we have used \eqref{lower bound of hatz_m} in the second line, while the positive sign of the last inequality is shown in Proposition \ref{poly: Lemma 6.3-2}.
	Thus, in order to show $(1+V)\Barrier_{\frac32}(V,z,2)>0$ for any $\gamma\in(\gaOne,3]$, $z\in(\hat{z}_m(\gamma),z_2(\gamma)]$ and $V\in[V_1,-\frac{2}{3}C_8^2(z))$, it is sufficient to show that
	\begin{align*}
		(1-\frac{2}{3}C^2_8(z))\Barrier_{\frac32}(-\frac{2}{3}C^2_8(z),z,2) \geq 0.
	\end{align*}
	
	\smallskip 
	
	$\underline{\text{Step 2}}:$ Our goal is to show $(1-\frac{2}{3}C^2_8(z))\Barrier_{\frac32}(-\frac{2}{3}C^2_8(z),\gamma,z,m)\geq 0$ for any  $\gamma\in(\gaOne,3]$, $z\in(\hat{z}_m(\gamma),z_2(\gamma)]$, $V\in[V_1,-\frac{2}{3}C_8^2(z))$, and $m=1,2$. We will first show $(1-\frac{2}{3}C^2_8(z))\Barrier_{\frac32}(-\frac{2}{3}C^2_8(z),\gamma,z,m)$ is a concave function in $z$. For notational convenience, we will denote 
	\begin{align}
		\Im(z) := -\frac{2}{3}C^2_8(z).
	\end{align}
	By using Lemma \ref{Lemma V_6/8 monoton}, \eqref{C8z2} and \eqref{C8hatzm}, we obtain
	\begin{align}\label{range of Im(z)}
	-\frac{1}{2}=\Im(\hat{z}_m(\gamma))<	\Im(z)\leq \Im(z_2) = -\frac{7-\sqrt{33}}{4}.
	\end{align}
	Note that by using Lemma \ref{Lemma V_6/8 monoton} and Lemma \ref{4C8C8''+(C8')^2<0}, we obtain
	\begin{align}
		\Im'(z) &= -\frac{4}{3}C_8C'_8(z)>0,\label{1st dev of Im(z)}\\
		 \Im''(z)&= -\frac{4}{3}[C_8C''_8(z)+(C'_8)^2]>0.\label{2nd dev of Im(z)}
	\end{align}
	We rewrite $(1-\frac{2}{3}C^2_8(z))\Barrier_{\frac32}(-\frac{2}{3}C^2_8(z),\gamma,z,m)$ as 
	\begin{align*}
		(1+\Im(z))\Barrier_{\frac32}(\Im(z),z,m)&=(m-1-m\gamma)\Im^3(z)+(3m+\frac{m+1}{2}-5-2m\gamma+m(\gamma-2)\gamma z)\Im^2(z)\\
		&\ \ +(2m+\frac{m+1}{2}-5-m\gamma+m(\gamma-3)\gamma z)\Im(z)+m(3-\gamma) z-1,
	\end{align*}
	and  compute the second $z$ derivative to obtain
	\begin{align*}
		\frac{d^2}{dz^2}[(1+\Im(z))\Barrier_{\frac32}(\Im(z),z,m)] := \Im''(z) A(\gamma,z,m)+\Im'(z) B(\gamma,z,m),
	\end{align*}
	where
	\begin{align*}
		A(\gamma,z,m) =&\, 3(m-1-m\gamma)\Im^2(z)+2[3m+\frac{m+1}{2}-5-2m\gamma+m(\gamma-2)\gamma z]\Im(z)\\
		&+2m+\frac{m+1}{2}-5-m\gamma+m(\gamma-3)\gamma z,\\
		B(\gamma,z,m) =&\, 6(m-1-m\gamma)\Im(z)\Im'(z)+2[3m+\frac{m+1}{2}-5-2m\gamma+m(\gamma-2)\gamma z]\Im'(z)\\
		&+4m(\gamma-2)\gamma \Im(z)+2m(\gamma-3)\gamma.
	\end{align*}
	We claim $A(\gamma,z,m)$ and $B(\gamma,z,m)$ are negative. We will first show $B(\gamma,z,m)<0$. We decompose $B(\gamma,z,m)$ into two parts
	\begin{align*}
		B(\gamma,z,m) =\Im'(z) B_1(\gamma,z,m) + B_2(\gamma,z,m),
	\end{align*}
	where 
	\begin{align*}
		B_1(\gamma,z,m) &=6(m-1-m\gamma)\Im(z)+7m-9-4m\gamma+2m(\gamma-2)\gamma z,\\
		B_2(\gamma,z,m) &=4m(\gamma-2)\gamma \Im(z)+2m(\gamma-3)\gamma .
	\end{align*}
	Clearly, $B_2(\gamma,z,m)<0$. Regarding $B_1(\gamma,z,m)$, 
	 by using \eqref{lower bound of hatz_m} and \eqref{range of Im(z)}, 
	 we obtain
	 \begin{align*}
	 	B_1(\gamma,z,m) &< 6(m-1-m\gamma)\Im(\hat{z}_m(\gamma))+7m-9-4m\gamma+2m\gamma z_M= 4m-6-\frac{3m\gamma}{4}<0
	 \end{align*}
	 for any $m=1,2$ and $\gamma\in(\gaOne,3]$.

	For $A(\gamma,z,m)$, we notice that
	\begin{align*}
		\frac{d A(\gamma,z,m)}{dz} = \Im'(z) B_1(\gamma,z,m)+\frac{1}{2} B_2(\gamma,z,m) <0.
	\end{align*}
	Thus, to show $A(\gamma,z,m)<0$, it is enough to verify  $A(\gamma,\hat{z}_m(\gamma),m)<0$.  
	Using \eqref{lower bound of hatz_m} and \eqref{range of Im(z)},
		\begin{align*}
		A(\gamma,\hat{z}_m(\gamma),m) = \frac{m\gamma}{4}-\frac{m+3}{4}-m\gamma\hat{z}_m(\gamma) < \frac{m\gamma}{5}-\frac{m+3}{4} < 0
	\end{align*}
	for  $m=1,2$ and $\gamma\in(\gaOne,3]$. 
Hence, $(1+\Im(z))\Barrier_{\frac32}(\Im(z),z,m)$ is a concave function in $z$. It is enough to check the sign 
at two ends points of $z$. By the definition \eqref{z2} of $z_2$, $(1+\Im(z_2))\Barrier_{\frac32}(\Im(z_2),\gamma,z_2,m)=0$. For $(1+\Im(\hat{z}_m(\gamma)))\Barrier_{\frac32}(\Im(\hat{z}_m(\gamma)),\hat{z}_m(\gamma),m)$, we evaluate 
	\begin{align*}
		(1+\Im(\hat{z}_m(\gamma)))\Barrier_{\frac32}(\Im(\hat{z}_m(\gamma))),\hat{z}_m(\gamma)),m) = \begin{cases}
			&\frac{\gamma-2}{8}+\frac{(12-\gamma^2)\hat{z}_m(\gamma)}{4} \text{ when }m=1,\\
			&\frac{\gamma-3}{4}+\frac{(12-\gamma^2)\hat{z}_m(\gamma)}{2} \text{ when }m=2.
		\end{cases}
	\end{align*}
	Obviously, $(1+\Im(\hat{z}_m(\gamma)))\Barrier_{\frac32}(\Im(\hat{z}_m(\gamma))),\hat{z}_m(\gamma)),1)>0$. For $(1+\Im(\hat{z}_m(\gamma)))\Barrier_{\frac32}(\Im(\hat{z}_m(\gamma))),\hat{z}_m(\gamma)),2)$, by using \eqref{lower bound of hatz_m}, we  have
	\begin{align*}
		(1+\Im(\hat{z}_m(\gamma)))\Barrier_{\frac32}(\Im(\hat{z}_m(\gamma))),\hat{z}_m(\gamma)),2) > \frac{\gamma-3}{4}+\frac{12-\gamma^2}{40} = \frac{-\gamma^2+10\gamma-18}{40}>0
	\end{align*}
	for any $\gamma\in(\gaOne,3]$. This completes the proof.
\end{proof}

\section{Solving to right}\label{sec:solntoright}

In the previous sections, for each $\ga\in(1,3]$, we established  
the existence of $z_{std}(\gamma;P_*)$ and a  
range of $z$ which $z_{std}(\gamma;P_*)$ must belong to.  
This $z_{std}$ allows the solution $C(V;\gamma,z_{std}(\gamma;P_*),P_*)$ of \eqref{ODE} to pass smoothly from $P_1$ through  
the 
triple point $P_*$, where $P_*$ is either $P_6$ or $P_8$. 
The remaining goal 
is to extend this smooth solution from the triple point $P_*$ to the origin $P_0$ while ensuring that it remains within the second quadrant in the phase plane (that is, that we retain both $V<0$ and $C>0$ up to time the flow meets the origin). To prove this property for the solution associated to $z_{std}$, we will in fact prove the stronger property that the local solution to the right of the sonic $P_*$ always extends to $P_0$ within the second quadrant for all $z$ within the range containing $z_{std}(\gamma;P_*)$. 

For notational convenience, we define unified notation for the various possible ranges of $z$ containing $z_{std}$ from the results of Sections \ref{sec:leftbasic}--\ref{sec:P8toleft}.
\begin{align}
	\mathring{\mathcal{Z}}(\gamma;P_*) = \begin{cases}
		&(z_g(\gamma),z_M(\gamma)] \quad\text{for }\gamma\in(1,2] \text{ at }P_6,\\
		&(z_1(\gamma),z_M(\gamma)] \quad\text{for }\gamma\in(\gaSix,\gaOne] \text{ at }P_8,\\
		&(z_2(\gamma),z_M(\gamma)] \quad\text{for }\gamma\in(\gaOne,3] \text{ at }P_8.
	\end{cases}
\end{align}
In this section, we will show that for any $\gamma\in(1,3]$ and $z\in\mathring{\mathcal{Z}}(\gamma;P_*)$, the local analytic solutions around $P_*$ constructed in Theorem \ref{analytic} continue to the origin $P_0$ in the second quadrant.

From the phase portrait analysis, three possibilities arise for the extension of the local analytic solution.
	\begin{enumerate}
		\item The trajectory intersects the negative $V$-axis before reaching $V=0$.
		\item The trajectory intersects the positive $C$-axis when $V=0$.
		\item The trajectory converges to $P_0$ within the second quadrant.
	\end{enumerate}
To rule out the first two possibilities, we will use suitable barrier functions 
to establish an invariant region for the solutions ensuring convergence to $P_0$ within the second quadrant. We begin with 
the extension for the local analytic solution to the right of the triple point.

\begin{lemma}\label{local solution to the right} Let $\gamma\in(1,3]$ and $z\in\mathring{\mathcal{Z}}$ be given and 
let $P_*=(V_*,C_*)$ be either $P_6$ or $P_8$. Consider the local, analytic solution $C:[V_*,V_*+\epsilon]\to\R_+$, guaranteed by Theorem \ref{analytic}. This solution  extends smoothly
to the right within the second quadrant onto the domain 
$C:[V_*,V_0)\to\R_+$, where $V_0 = \min\{C^{-1}(0),0\}$. Furthermore, except at the triple point, the solution enjoys  
 $\frac{dC}{dV}<0$, $F<0$, $G>0$, $D<0$.
\end{lemma}
\begin{proof}
	The result follows 
	from our choice of $c_1$ (\textit{cf.}  \eqref{c1}), $F_C(V_*,C_*)>0$ (\textit{cf.}  \eqref{F_C}), $G_C(V_*,C_*)<0$ (\textit{cf.}  \eqref{G_C}).
\end{proof}

Next, we will eliminate the first possibility.
\begin{lemma}\label{lemma:Cneq0}
	For any $\gamma\in(1,3]$ and $z\in\mathring{\mathcal{Z}}$, the solution constructed in 
	Lemma \ref{local solution to the right} 
	does not intersect the negative $V$-axis (i.e., in the notation of that Lemma, $V_0=0$). 
\end{lemma}
\begin{proof}
	We argue by contradiction. 
	Suppose the solution intersects the negative $V$-axis before reaching $V=0$, and let $(\overline{V},0)$ denote the  point of intersection of the solution trajectory with the $V$-axis. 
Consider the initial value problem:
\begin{equation*}
	\begin{cases}
	&\cfrac{dC(V)}{dV} = \cfrac{F(V,C;\gamma,z)}{G(V,C;\gamma,z)},\\
	& C(\overline{V})=0.
\end{cases}
\end{equation*}
In a small rectangular region around $(\overline{V}, 0)$, it is evident that $\frac{F(V,C;\gamma,z)}{G(V,C;\gamma,z)}$ is continuously differentiable. By 
the standard theorem for existence and uniqueness of solutions to ODEs with locally Lipschitz right hand side (e.g.~\cite{Teschl12}), this initial value problem possesses a unique solution on the interval $(\overline{V}-\epsilon, \overline{V}+\epsilon)$ for sufficiently small $\epsilon$. However, 
we see trivially that $C(\overline{V}) \equiv 0$ solves this problem, and so must be the unique solution, leading to a contradiction. 
\end{proof}

We remark that Lemma \ref{lemma:Cneq0}  
yields $V_0=0$ in Lemma \ref{local solution to the right}. 
The rest of this section is devoted to ruling out the second possibility by the barrier argument. 

\subsection{Connecting $P_6$ to the origin}

In this subsection, we prove that the local analytic solutions around $P_6$ constructed in Theorem \ref{analytic} for any $\gamma\in (1,2]$ and $z\in(z_g(\gamma),z_M(\gamma)]$ continue to the origin and stay below the barrier curve $B_1(V)=\sqrt{-V}$. In particular, this implies that the solution trajectories to the right of the sonic point will stay between $C=0$ and $B_1(V)$ and must therefore strictly decrease to the origin.

\begin{lemma}\label{right of P6}
	For any $\gamma\in (1,2]$ and $z\in(z_g(\gamma),z_M(\gamma)]$, the solution constructed in  
	Lemma \ref{local solution to the right} with $P_*=P_6$, always lies below the curve $ B_1(V)= \sqrt{-V}$ for $V\in[V_6,0)$.
\end{lemma}

\begin{proof} 
	According to \eqref{upper and lower bound for V/C_6/8}, 
	$P_6$ consistently remains below the curve $C=B_1(V)$ for any $\gamma\in (1,2]$. In other words, $C_6<B_1(V_6)$. Additionally, as per Lemma \ref{local solution to the right}, $\frac{dC}{dV}<0$ holds for $V\in [V_6,0)$. Observing that $-C_6^2>V_6$, we therefore see that the inequality $C(V)<B_1(V)$ holds trivially for $V\in[V_6,-C_6^2]$. Consequently, our analysis can be confined to $V\in(-C_6^2,0)$.
	Thus, employing the barrier argument \eqref{barrier_argument}, we verify the validity of the following inequality for any $\gamma\in(1,2]$, $z\in(z_g(\gamma),z_M(\gamma)]$, and $V\in(-C_6^2,0)$:
	\begin{align}\label{sign_check}
		\frac{F(V,\sqrt{-V};\gamma,z)}{G(V,\sqrt{-V};\gamma,z)}+\frac{1}{2\sqrt{-V}} <0.
	\end{align}
	Observe that the inequality is \eqref{barrier inequality} with $k(\gamma,z)$ replaced by $1$. So, by Lemma \ref{local solution to the right} which assures that $G(V,\sqrt{-V};\gamma,z)> 0$ for any $V\in(-C_6^2,0)$, our task reduces to proving 
	$$\Barrier_1(V,z,m)=\frac{2F(V,\sqrt{-V};\gamma,z)}{\sqrt{-V}}+\frac{G(V,\sqrt{-V};\gamma,z)}{-V}<0,$$
	for $m=1,2$, $\ga\in(1,2]$, $z\in(z_g(\ga),z_M(\ga)]$, $V\in(-C_6^2,0)$, 
where we recall from \eqref{BP81} that
\begin{align*}
	\Barrier_1(V,z,m) 
		&=(m-1-m\gamma)V^2+(-3+2m-m\gamma+m(\gamma-2)\gamma z)V-m\gamma z-1+\frac{2mz}{1+V}.
	\end{align*} 
	We note first that
	\begin{align}
		\frac{\partial \Barrier_1(V,z,m)}{\partial V} &= (m-1-m\gamma)(1+2V)+m-2+m(\gamma-2)\gamma z-\frac{2mz}{(1+V)^2}\notag\\
		&<(m-1-m\gamma)(1-2C_6^2)<0\label{positivity of dBar/dV}
	\end{align}
	for $V\in[-C_6^2,0)$, since $C_6 \le\frac{\sqrt{\gamma}}{\sqrt{\gamma}+\sqrt{2}} \leq\frac{1}{2}$ by using \eqref{upper and lower bound for V/C_6/8}. 
	 Thus, it is sufficient to check the negativity of $\Barrier_1(-C^2_6(\gamma,z),z, m)$. 

For any fixed $\gamma\in(1,2]$, we write $C_6=C_6(z)$ and $C_6'(z) = \frac{dC_6}{dz}$.
		By  
		\eqref{positivity of dBar/dV}, Lemma \ref{Lemma V_6/8 monoton} and the chain rule, 
 we check the sign of the $z$-derivative of $\Barrier_1(-C^2_6(z),z,m)$: 
	\begin{align*}
		\frac{d \Barrier_1(-C^2_6(z),z,m)}{d z} =  \frac{\partial \Barrier_1 }{\partial V}(-C^2_6(z),z,m)(-2C_6(z)C_6'(z)) + \frac{\partial \Barrier_1 }{\partial z}(-C^2_6(z),z,m)>0,		
	\end{align*}
	where we also used $\frac{\partial \Barrier_1 }{\partial z}(-C^2_6(z),z,m) =-m(\gamma-2)\gamma C^2_6(z) + m(\frac{2}{1-2C^2_6(z)}-\gamma)>0$ for any $\gamma\in(1,2]$. 
	As a consequence, it is now enough to show that $\Barrier_1(-C^2_6(\gamma,z_M),z_M)<0$ to accomplish \eqref{sign_check}. 
	By \eqref{V68(z_M) C68(z_M)} and \eqref{z_{max}}, we have $C_6(\gamma,z_M) = \sqrt{\gamma z_M}$ and hence, we obtain 
	\begin{align*}
		&\Barrier_1(-C^2_6(\gamma,z_M),z_M,m)\\ &= (m-1-m\gamma)\gamma^2z^2_M-(-3+2m-m\gamma+m(\gamma-2)\gamma z_M)\gamma z_M-m\gamma z_M-1+\frac{2mz_M}{1-\gamma z_M}\\
		&= (3m-1)\gamma^2z^2_M -2m\gamma^3z^2_M+3(1-m)\gamma z_M+m\gamma^2z_M-1+\frac{2mz_M}{1-\gamma z_M}.
	\end{align*}
	When $m=1$, we have
	\begin{align}
		\Barrier_1(-C^2_6(\gamma,z_M),z_M,1) 
		&=2(1-\gamma) \gamma^2z^2_M+\frac{\gamma^2z_M-\gamma^3z^2_M-1+(\gamma+2) z_M}{1-\gamma z_M}.\label{eq:B1C62m=1}
	\end{align}
	With $\gamma>1$, the first term in the equation above is negative. To find an explicit expression for the second term, 
	first, by \eqref{z_{max}}, we see that 
	\begin{align*} 
		1-\gamma z_M = 1-\frac{\gamma}{\gamma+2+2\sqrt{2\gamma}}>0
	\end{align*}
and
	\begin{align*}
		\gamma^2z_M-\gamma^3z^2_M-1+(\gamma+2) z_M &= \frac{\gamma^2(\sqrt{\gamma}+\sqrt{2})^2-\gamma^3-(\sqrt{\gamma}+\sqrt{2})^4+(\gamma+2)(\sqrt{\gamma}+\sqrt{2})^2}{(\sqrt{\gamma}+\sqrt{2})^4}\\
		&=\frac{(\gamma^2-2\sqrt{2\gamma})(\sqrt{\gamma}+\sqrt{2})^2-\gamma^3}{(\sqrt{\gamma}+\sqrt{2})^4}<0,
	\end{align*}
	where we have used $\gamma\leq 2$. Therefore, from \eqref{eq:B1C62m=1}, we deduce that $\Barrier_1(-C^2_6(\gamma,z_M),z_M,1)<0$ and hence $\Barrier_1(V,z,1)<0$ for all $V\in(-C_6^2,0)$, $z\in(z_g(\ga),z_M(\ga)]$.

	When $m=2$, we have
	\begin{align}
		\Barrier_1(-C^2_6(\gamma,z_M),z_M,2) 
		&=4(1-\gamma)\gamma^2z^2_M+\frac{4\gamma^2 z^2_M-2\gamma z_M+2\gamma^2z_M-\gamma^3z^3_M-2\gamma^3z^2_M-1+4z_M}{1-\gamma z_M}.\label{eq:B1C62m=2}
	\end{align}
	Since $\gamma>1$, the first term 
	is again negative. For the second term, we again apply \eqref{z_{max}} to rearrange the numerator 
	as 
	\begin{align}
		&4\gamma^2 z^2_M-2\gamma z_M+2\gamma^2z_M-\gamma^3z^3_M-2\gamma^3z^2_M-1+4z_M\notag\\ 
		 &= (4\gamma z_M-1)\gamma z_M-\gamma^3z^3_M-2\gamma^3z_M^2+2(\gamma^2+1-\sqrt{2\gamma}-\gamma)z_M\notag\\
		 &=\frac{3\gamma-2-2\sqrt{2\gamma}}{\gamma+2+2\sqrt{2\gamma}}\gamma z_M-\gamma^3z^3_M+2(\gamma^2+1-\sqrt{2\gamma}-\gamma-\gamma^3z_M)z_M.\label{eq:frog1}
	\end{align}
	The first two terms are negative because $\gamma\in(1,2]$. As for the last term, when $\gamma\in(1,\sqrt[3]{2}]$, we have
	\begin{align*}
		\gamma^2-\sqrt{2\gamma}<0.
	\end{align*}
	Thus, $\Barrier_1(-C^2_6(\gamma,z_M),z_M,2)<0$ when $\gamma\in(1,\sqrt[3]{2}]$.
	When $\gamma\in(\sqrt[3]{2},2]$, by using $z_M\geq\frac{1}{8}$, we obtain
	\begin{align}
		\gamma^2+1-\sqrt{2\gamma}-\gamma-\gamma^3z_M &< \gamma^2+1-\sqrt{2\gamma}-\gamma-\frac{\gamma^3}{8}.\label{eq:frog2}
	\end{align}
	This upper bound is increasing as a function of $\ga$ as
	\begin{align*}
		\frac{d}{d\gamma}\left(\gamma^2+1-\sqrt{2\gamma}-\gamma-\frac{\gamma^3}{8}\right) = -\frac{3\gamma^2}{8}+2\gamma-1-\frac{1}{\sqrt{2\gamma}}
		\geq -\frac{3\gamma^2}{8}+2\gamma-1-\frac{1}{\sqrt{2\sqrt[3]{2}}}>0.
	\end{align*}
	Hence, for any $\gamma\in(\sqrt[3]{2},2]$, we have
	\begin{align*}
		\gamma^2+1-\sqrt{2\gamma}-\gamma-\frac{\gamma^3}{8} \leq 0,
	\end{align*}
	which, combined with \eqref{eq:B1C62m=2}, \eqref{eq:frog1} and \eqref{eq:frog2}, leads to $\Barrier_1(-C^2_6(\gamma,z_M),z_M,2)<0$ in the remaining range $\ga\in[\sqrt[3]{2},2]$ and hence we have established $\Barrier_1(V,z,2)<0$ for any $V\in [-C^2_6,0)$, $z\in(z_g(\ga),z_M(\ga)]$. This concludes the proof. 
\end{proof}

\subsection{Connecting $P_8$ to the origin}

In this subsection we prove analogous results around $P_8$ for $\gamma\in (\gaSix, 3]$ to those of the previous subsection, that is, we show that the local analytic solutions around $P_8$ for $z\in \mathring{\mathcal{Z}}(\gamma;P_8)$ converge to the origin $P_0$ within the second quadrant by employing a barrier argument. We split into two cases: $\gamma\in(\gaSix,\gaOne]$ and $\gamma\in (\gaOne, 3]$. 

\subsubsection{$P_8$ for $\gamma\in(\gaSix,\gaOne]$}

\begin{lemma}\label{P8rightga6toga1}
	For any $\gamma\in (\gaSix,\gaOne]$ and $z\in(z_1(\gamma),z_M(\gamma)]$, the solution constructed in 
	Lemma \ref{local solution to the right} with $P_*=P_8$, always lies below the curve $ B_1(V)= \sqrt{-V}$ for $V\in(V_8,0)$.
\end{lemma}
\begin{proof}
	By the definition of $z_1(\gamma)$ (\textit{cf.}  \eqref{z1}) and Lemma \ref{Lemma V_6/8 monoton}, $C_8\leq B_1(V_8)$ for any $\gamma\in (\gaSix,\gaOne]$ and $z\in(z_1(\gamma),z_M(\gamma)]$. Since $\frac{dC}{dV}<0$ by Lemma \ref{local solution to the right}, the solution stays below the curve $B_1(V)$ for $V\in(V_8,-C_8^2]$. Thus,  
	it suffices to show the claim for  
	$V\in(-C_8^2,0)$. By 
	employing the barrier argument  \eqref{barrier_argument}, we will establish 
	the following inequality for any $\gamma\in(\gaSix,\gaOne]$, $z\in(z_1(\gamma),z_M(\gamma)]$, and $V\in(-C_8^2,0)$: 
	\begin{align}
		\frac{F(V,\sqrt{-V};\gamma,z)}{G(V,\sqrt{-V};\gamma,z)}+\frac{1}{2\sqrt{-V}} <0.
	\end{align}
	By 
	Lemma \ref{local solution to the right}, $G(V,\sqrt{-V};\gamma,z)>0$ for all $V\in(-C_8^2,0)$ and hence, as in Lemma \ref{leftgammabar} ({\textit{cf.} }\eqref{BP81}), we again see that it is sufficient to  
	prove that
	\begin{align*}
		\Barrier_1(V,z,m)=(m-1-m\gamma)V^2+(-3+2m-m\gamma+m(\gamma-2)\gamma z)V-m\gamma z-1+\frac{2mz}{1+V}<0.
	\end{align*}
	We first derive an upper bound of $\frac{\partial \Barrier_1}{\partial V} $: 
	\begin{align*}
		\frac{\partial \Barrier_1(V,z,m)}{\partial V} &= 2(m-1-m\gamma)V-3+2m-m\gamma+m(\gamma-2)\gamma z-\frac{2mz}{(1+V)^2}\\
		&<(m-1-m\gamma)(2V+1)+m-2+[(\gamma-2)\gamma-2] mz\\
		&<(m-1-m\gamma)(2V+1)
	\end{align*}
	because $(\gamma-2)\gamma-2<0$ for $\gamma\in(\gaSix,\gaOne]$.
	By \eqref{upper and lower bound for V/C_6/8}, we observe that for any $\gamma\in(\gaSix,\gaOne]$,
	\begin{align}
		\frac{9}{20}<\frac{\sqrt{\gaSix}}{\sqrt{\gaSix}+\sqrt{2}}<\frac{\sqrt{\gamma}}{\sqrt{\gamma}+\sqrt{2}}\leq C_8 \leq C_8(z_1) = \frac{\sqrt{5}-1}{2}. \label{C8upperlower}
	\end{align} 
	Thus, $2V+1>-2C_8^2+1>0$, which  implies
	\begin{align}\label{ineq:B1P8right}
		\frac{\partial \Barrier_1(V,z,m)}{\partial V}<0.
	\end{align}
	It is therefore sufficient to show that $\Barrier_1(-C^2_8(\gamma,z),z,m)<0$. 
	We use the same notation as in Lemma \ref{leftgammabar}, denoting $\Jm(z) = -C^2_8(\gamma,z)$. 
	For any fixed $\gamma\in(\gaSix,\gaOne]$, we write $\Jm'(z) = \frac{d\Jm(z)}{dz}$ and $\Jm''(z) = \frac{d^2\Jm(z)}{dz^2}$. We then derive the $z$ derivative of $ \Barrier_1(\Jm(z),z,m)$ as 
	\begin{align*}
		\frac{\partial \Barrier_1(\Jm(z),z,m)}{\partial z} =\frac{\partial \Barrier_1}{\partial V} (\Jm(z),z,m) \Jm'(z) + \frac{\partial \Barrier_1}{\partial z}(\Jm(z),z,m)
			=: A(\gamma,z,m) \Jm'(z)+mB(\gamma,z),
	\end{align*}
	where 
	\begin{align*}
	&A(\gamma,z,m) : = 2(m-1-m\gamma)\Jm(z)-3+2m-m\gamma+m(\gamma-2)\gamma z,\\
	&B(\gamma, z):=\frac{2}{1+\Jm(z)}-\gamma+(\gamma-2)\gamma \Jm(z)-\frac{2z\Jm'(z)}{(1+\Jm(z))^2}.
	\end{align*}
	We claim that both $A(\gamma,z,m)$ and $B(\gamma,z)$ are negative. For  $A(\gamma,z,m)$, in the case $m=1$, we have 
	\begin{align*}
		A(\gamma,z,1) = -2\gamma \Jm(z)-1-\gamma+(\gamma-2)\gamma z < -\gamma(1+2\Jm(z))-1+|\gamma-2|\gamma z_M < 0,
	\end{align*}
	where we have used $|\gamma-2|\gamma z_M < 1$ in the last inequality. When $m=2$, by \eqref{C8upperlower} and $z<z_M(\gamma)< z_M(\gaSix)<\frac{1}{5}$ for any $\gamma\in(\gaSix,\gaOne]$ and $z\in(z_1(\gamma),z_M(\gamma)]$, we see
	\begin{align*}
		A(\gamma,z,2) = (1-2\gamma)(1+2\Jm(z))+2(\gamma-2)\gamma z < \frac{1-2\gamma+2|\gamma-2|\gamma}{5}<0
	\end{align*}
	for any $\gamma\in(\gaSix,\gamma_1]$. Hence, $A(\gamma,z,m)<0$.

	For $B(\gamma,z)$, we first observe that when $\gamma\in(\gaSix,\gaOne]$, by \eqref{C8upperlower},
	\begin{align}
		-\gamma+(\gamma-2)\gamma \Jm(z) = -\gamma[1-(\gamma-2)\Jm(z)] < 0 .\label{eq:toad1}
	\end{align}
	For the remaining two terms,  by Lemma \ref{Lemma V_6/8 monoton} and Lemma \ref{4C8C8''+(C8')^2<0}, we have  $\Jm'(z)>0$ and $\Jm''(z)>0$. Therefore, for any $z\in(z_1(\gamma),z_M(\gamma)]$, 
	\begin{align*}
		\frac{d}{dz}[1+\Jm(z)-z\Jm'(z)] =  -z \Jm''(z) <0 .
	\end{align*}
	Hence, using also \eqref{eq:toad1},
	\begin{align*}
		B(\gamma, z)<\frac{2}{1+\Jm(z)} -\frac{2z\Jm'(z)}{(1+\Jm(z))^2} = \frac{2}{(1+\Jm(z))^2} (1+\Jm(z)-z\Jm'(z)) < \frac{2}{(1+\Jm(z))^2} (1+\Jm(z_1)-z_1\Jm'(z_1)) .
	\end{align*}
	Recalling \eqref{P8} and \eqref{V8C8z1}, we have
	\begin{align*}
		C_8(z_1) = \frac{\sqrt{5}-1}{2} \text{ and } w(z_1) = \sqrt{5}-2-(\gamma-2)z_1>0,
	\end{align*}
	where the positivity of $w(z_1)$ is due to Remark \ref{bounds of w(z)}.
	By direct computation, using \eqref{P8}, \eqref{w(z)}, and  \eqref{V8C8z1}, 
	we obtain
	\begin{align*}
		1+\Jm(z_1)-z_1\Jm'(z_1) &= 1- C_8^2(z_1)+2z_1C_8(z_1)C_8'(z_1)\\
		&=\frac{C_8(z_1)}{w(z_1)} [\sqrt{5}-2+[(\sqrt{5}-4)\gamma-2(\sqrt{5}-2)]                                                                                                                                                                                                                           z_1]\\
		&=\frac{C_8(z_1)}{w(z_1)}\frac{(5-3\sqrt{5})\gamma+4\sqrt{5}-8}{2(\sqrt{5}+1+\gamma)}<0
	\end{align*}
	for any $\gamma\in(\gaSix,\gaOne]$, and therefore we obtain $B(\ga,z)<0$.

	Therefore, we have obtained $\frac{\partial }{\partial z}\Barrier_1(\Jm(z),z,m)<0$. 
	Hence, we conclude that for any fixed $\gamma\in (\gaSix,\gaOne]$, $m=1,2$, $z\in(z_1(\gamma),z_M(\gamma)]$, and $V\in (-C^2_8,0)$, applying also  \eqref{ineq:B1P8right}, we have $\Barrier_1(V,z,m)<\Barrier_1(\Jm(z),z,m)<\Barrier_1(\Jm(z_1),z_1,m)=0$ by the definition of $z_1$ (\textit{cf.}  \eqref{z1}).
	\end{proof}

\subsubsection{$P_8$ for $\gamma\in(\gaOne,3]$}

\begin{lemma}\label{P8rightga1to3}
	For any $\gamma\in (\gaOne,3]$ and $z\in(z_2(\gamma),z_M(\gamma)]$, the solution  
	constructed in  
	Lemma \ref{local solution to the right} with $P_*=P_8$, always lies below the curve $ B_\frac{3}{2}(V)= \sqrt{-\frac{3}{2}V}$ for $V\in(V_8,0)$.
\end{lemma}
\begin{proof}
	Using a similar argument as in Lemma \ref{P8rightga6toga1}, it suffices to verify the following inequality:
	\begin{align*}
	 \frac{F(V,\sqrt{-\frac{3}{2}V};\gamma,z)}{G(V,\sqrt{-\frac{3}{2}V};\gamma,z)}+\frac{1}{2}\sqrt{\frac{-3}{2V}}<0
	\end{align*}
	for any $m=1,2$, $\gamma\in (\gaOne,3]$, $z\in(z_2(\gamma),z_M(\gamma)]$ and $V\in(-\frac{2}{3}C_8^2,0)$. Given that $G(V,\sqrt{-\frac{3}{2}V},\gamma,z)>0$ for any $V\in(-\frac{2}{3}C_8^2,0)$ as shown in 
	Lemma \ref{local solution to the right}, and by the same calculations  
	as in Lemma \ref{leftlargergammabar} (\textit{cf. }\eqref{BP82}), it is sufficient to demonstrate that
	\begin{align*}
		\Barrier_{{\scriptstyle \frac32}}(V,z,m)=(m-1-m\gamma)V^2+(-4+2m+\frac{m+1}{2}-m\gamma+m(\gamma-2)\gamma z)V-m z\gamma-1+\frac{3mz}{1+V}<0
	\end{align*}
	for any $m=1,2$, $\gamma\in (\gaOne,3]$, $z\in(z_2(\gamma),z_M(\gamma)]$ and $V\in (-\frac{2}{3}C^2_8,0)$. The $V$ derivative of $\Barrier_{\frac32}$ is given by 
	\begin{align*}
		\frac{\partial \Barrier_{\frac32}(V,z,m)}{\partial V} =2(m-1-m\gamma)V+(-4+2m+\frac{m+1}{2}-m\gamma+m(\gamma-2)\gamma z)-\frac{3mz}{(1+V)^2}.
	\end{align*}
	When $m=1$, the same argument as in \eqref{sign_BV} implies $\frac{\partial \Barrier_{{\scriptstyle \frac32}}(V,z,1)}{\partial V}<0$.  
	When $m=2$, we have
	\begin{align*}
		\frac{\partial \Barrier_{\frac32}(V,z,2)}{\partial V} =2(1-2\gamma)V+\frac{3}{2}-2\gamma+2[(\gamma-2)\gamma -\frac{3}{(1+V)^2}]z.
	\end{align*}
	Note that $(\gamma-2)\gamma -\frac{3}{(1+V)^2}<0$ for $\gamma\in (\gaOne,3]$ and $V\in (-\frac{2}{3}C_8^2,0)$. 
	Moreover, 
	$2(1-2\gamma)V+\frac{3}{2}-2\gamma<0$, 
	since for any fixed $\gamma\in (\gaOne,3]$,
	\begin{align*}
		2(1-2\gamma)V+\frac{3}{2}-2\gamma < -2(1-2\gamma)\frac{2}{3}C_8^2+\frac{3}{2}-2\gamma
		=\frac{\sqrt{33}-4}{2}+(5-\sqrt{33})\gamma <0 .
	\end{align*}
	Hence, 
	\begin{align*}
		\frac{\partial \Barrier_{\frac32}(V,z,m)}{\partial V}<0.
	\end{align*}
	It is therefore sufficient to show that $\Barrier_{\frac32}(-\frac{2}{3}C^2_8(z),z)<0$.   
	Let $\Im(z) = -\frac{2}{3}C^2_8(z)$ and, for any fixed $\gamma$, we write $\Im'(z) = \frac{d \Im(z)}{dz}$ and $\Im''(z) = \frac{d^2 \Im(z)}{dz^2}$. By Lemma \ref{Lemma V_6/8 monoton} and Lemma \ref{4C8C8''+(C8')^2<0}, $\Im'(z)>0$ and $\Im''(z)>0$.
	The $z$ derivative of $ \Barrier_{\frac32}(\Jm(z),z,m)$ can therefore be written as 
	\begin{align*}
		\frac{d \Barrier_{\frac32}(\Im(z),z,m)}{d z} = \frac{\partial \Barrier_{\frac32}}{\partial V} (\Im(z),z,m)\Im'(z)+ \frac{\partial \Barrier_{\frac32}}{\partial z}(\Im(z),z,m)
		=:A(\gamma,z,m)\Im'(z)+mB(\gamma,z),
	\end{align*}
		where 
	\begin{align*}
	 &A(\gamma,z,m):=2(m-1-m\gamma)\Im(z)-4+2m+\frac{m+1}{2}-m\gamma+m\gamma(\gamma-2)z,\\
	 &B(\gamma,z):=(\gamma-2)\gamma  \Im(z)-\gamma+\frac{3(1+\Im(z))-3z\Im'(z) }{(1+\Im(z))^2} .
	\end{align*}
	We claim that both $A(\gamma,z,m)$ and $B(\gamma,z)$ are negative. We first check $A(\gamma,z,m)$. When $m=1$,  
	\begin{align}
		A(\gamma,z,1)=-2\gamma\Im(z)-1-\gamma+(\gamma-2)\gamma z = -\gamma(1+2\Im(z)) - 1+(\gamma-2)\gamma z <0,\label{ineq:P8Am=1}
	\end{align}
	where we have used $(\gamma-2)\gamma z\le (\gamma-2)\gamma z_M\le \frac{\gamma}{\gamma+2+2\sqrt{2\gamma}}\leq\frac{3}{5+2\sqrt{6}}$ and $1+2\Im(z)> 1+2\Im(z_2)= \frac{\sqrt{33}-5}{2}$ 
	in the last inequality.
	When $m=2$, using $1+2\Im(z)> 1+2\Im(z_2)= \frac{\sqrt{33}-5}{2}>\frac{3}{5+2\sqrt{6}}\geq \gamma z_M\geq \gamma z$, 
	\begin{align}
		A(\gamma,z,2)&=2(1-2\gamma)\Im(z)+\frac{3}{2}-2\gamma+2(\gamma-2)\gamma z = 2(2-\gamma)[1+2\Im(z)-\gamma z]-3(1+2\Im(z))+\frac{1}{2}\notag
		\\
		&\le -3(1+2\Im(z))+\frac{1}{2} \le -3 (\frac{\sqrt{33}-5}{2}) + \frac12= \frac{16-3\sqrt{33}}{2}<0.\label{ineq:P8Am=2}
	\end{align}
	For $B(\gamma,z)$, the first term is trivially negative for any $\gamma\in(\gaOne, 3]$ since $\Im(z)<0$. We will show that the remaining term is negative as well. 
	By using \eqref{V68(z_M) C68(z_M)}, we obtain that for any $z\in(z_2(\gamma),z_M(\gamma)]$,
	\begin{align*}
		0<C_8(z_2)=1-\frac{2}{3}C_8^2(z_2) < 1+\Im(z)\leq 1-\frac{2}{3}C_8^2(z_M) = 1-\frac{2\gamma z_M(\gamma)}{3}.
	\end{align*} Since $\frac{d}{dz}(-z\Im'(z)+1+\Im(z)) = -z\Im''(z)<0$, 
		\begin{align}
		 \frac{3(-z\Im'(z)+1+\Im(z))}{(1+\Im(z))^2} < \frac{3|-z_2\Im'(z_2)+1+\Im(z_2)|}{(1+\Im(z_2))^2}.\label{ineq:toad2}
	\end{align}
		Recalling \eqref{P8}, Remark \ref{bounds of w(z)} and \eqref{C8z2}, we have
	\begin{align*}
		C_8(z_2) =  \frac{\sqrt{33}-3}{4} \text{ and } w(z_2) =  \frac{\sqrt{33}-5}{2}-(\gamma-2)z_2>0.
	\end{align*}
	Thus, by a  direct computation, we obtain
	\begin{align}
		\frac{3|-z_2\Im'(z_2)+1+\Im(z_2)|}{(1+\Im(z_2))^2} &= \frac{3}{C_8^2(z_2)}|1-\frac{2}{3}C^2_8(z_2)+\frac{4}{3}z_2C_8(z_2)C_8'(z_2)|\notag\\
		&=\frac{1}{C_8(z_2)}\frac{|(33-7\sqrt{33})\gamma+12\sqrt{33}-48|}{(\sqrt{33}-7)\gamma+12}\notag\\
		&=\frac{4}{\sqrt{33}-3}|\sqrt{33}-\frac{48}{(\sqrt{33}-7)\gamma+12}|\notag\\
		&\leq \frac{4}{\sqrt{33}-3}(\sqrt{33}-\frac{48}{(\sqrt{33}-7)\gaOne+12})<\gamma \label{ineq:toad3}
	\end{align}
	for any $\gamma\in(\gaOne,3]$. 
	Therefore, combining \eqref{ineq:P8Am=1}--\eqref{ineq:toad3}, we have found
	\begin{align*}
	\frac{d \Barrier_{\frac32}(\Im(z),z,m)}{d z} =&\,A(\gamma,z,m)\Im'(z)+mB(\gamma,z)\\
	<&\, m\Big((\gamma-2)\gamma  \Im(z)-\gamma+\frac{3(1+\Im(z))-3z\Im'(z) }{(1+\Im(z))^2} \Big)\\
	<&\, m\Big( -\ga + \frac{3|-z_2\Im'(z_2)+1+\Im(z_2)|}{(1+\Im(z_2))^2}\Big)<0.
	\end{align*}
	Hence, for any fixed $\gamma\in (\gaOne,3]$, $m=1,2$, $z\in(z_2(\gamma),z_M(\gamma)]$, and $V\in (-\frac{2}{3}C^2_8,0)$, we have $\Barrier_{\frac32}(V,z,m)<\Barrier_{\frac32}(\Im(z),z,m)<\Barrier_{\frac32}(-\frac{2}{3}C_8^2(z_2),z_2,m)=0$ by the definition of $z_2$ (\textit{cf.}  \eqref{z2}).
	\end{proof}

\section{Proof of the main theorem}\label{sec:mainthmproof}

We now prove Theorem \ref{main theorem}.

$(i)$ follows by $(ii)$, $(iii)$, and $(iv)$.

By Theorem \ref{existence of zstd}, there exists a $z_{std}(\gamma)$ such that the local real analytic solution $C(V;\gamma,z,P_*)$ around $P_*$ for either $P_*=P_6$ or $P_*=P_8$ given by Theorem \ref{analytic}  extends on the left to $P_1$.

$(ii)$. Let $\gamma\in(1,\gaSix]$ be fixed. 
By Lemma \ref{onlypass P6}, any such $z_{std}(\ga)$ must connect $P_1$ to $P_6$ and, by Proposition \ref{prop:zglower}, $z_{std}(\gamma) \in (z_g(\gamma),z_M(\gamma)]$, that is, $z_{std}(\gamma) \in \mathring{\mathcal{Z}}(\gamma;P_6)$. Then, Lemma \ref{uniquenessP6} gives that in fact  $z_{std}(\gamma)$ is  unique. 
By Lemma \ref{right of P6}, this unique solution  extends to $P_0$ 
in the second quadrant to give a unique connection from $P_1$ to $P_0$ which passes through $P_6$ and is monotone.

$(iii)$. Let $\gamma\in(\gaSix,2)$ be fixed. Given such a $z_{std}(\ga)$, if $P_1$ is connected analytically to $P_6$, by Proposition \ref{prop:zglower}, we must have $z_{std}(\gamma) \in (z_g(\gamma),z_M(\gamma)]$ and so, by Lemma \ref{right of P6}, the solution extends to the right of $P_6$ to connect to $P_0$ within the second quadrant.
 On the other hand, if the solution connects $P_1$ to $P_8$ analytically, then by Lemma \ref{leftgammabar}, $z_{std}(\gamma)\in(z_1(\gamma),z_M(\gamma)]$, and, applying Lemma \ref{P8rightga6toga1}, the solution again extends inside the second quadrant to connect to $P_0$. Thus, in either case, we have $z_{std}(\gamma) \in \mathring{\mathcal{Z}}(\gamma;P_*)$ and have obtained a monotone analytic solution connecting $P_1$ to $P_0$ through a single triple point.

$(iv)$. Let $\gamma\in[2,3]$ be fixed. By Proposition \ref{onlypass P8}, the solution for $z=z_{std}(\ga)$ must connect $P_1$ to $P_8$ analytically and, by Lemmas \ref{leftgammabar} and \ref{leftlargergammabar}, $z_{std}(\gamma) \in \mathring{\mathcal{Z}}(\gamma;P_8)$. Therefore, by  Lemma \ref{P8rightga6toga1} and Lemma \ref{P8rightga1to3}, in each case, the solution  extends to the right in the second quadrant to connect $P_8$ to $P_0$ and we again have obtained an analytic, monotone solution connecting $P_1$ to $P_0$.

\

{\bf Acknowledgements.} JJ and JL are supported in part by the NSF grants DMS-2009458 and DMS-2306910. MS is supported by the EPSRC Post-doctoral Research Fellowship EP/W001888/1.

\begin{appendices}
\section{Calculation for Taylor Expansion}\label{caltaylorexp}
The purpose of this Appendix is to establish the proof of Lemma \ref{lemma:Taylorcoeffs}. To this end, we begin from \eqref{ODE} in the form
$$(1+V)F(C,V)- (1+V)C'(V)G(C,V)=0.$$
We write the left hand side of this equation as a power series in $v=V-V_*$ as
$$(1+V)F(C,V)- (1+V)C'(V)G(C,V)=\sum_{\ell=0}^\infty \mathcal{C}_\ell v^\ell.$$
Substituting in \eqref{Taylor Series}, \eqref{taylor expansion of dc/dv}, \eqref{v(C^2)_l(c)^3_l taylor expansion}, and \eqref{C^2C^3C'C^2 taylor expansion}, the first term of this identity expands as
\begin{allowdisplaybreaks}
\begin{align}
    &(1+v+V_*)F(V,C)\notag  \\ 
    =&\,(1+v+V_*)C(V)\big\{C^2(V)[1+\frac{mz}{(1+v+V_*)}]- a_1(1+v+V_*)^2+a_2(1+v+V_*)-a_3\big\}\notag\\
    =&\,vC^3(V)+(1+V_*+mz)C^3(V)+\Big[-a_1v^3+[-3a_1(1+V_*)+a_2]v^2+[-3a_1(1+V_*)^2+2a_2(1+V_*)-a_3]v\notag\\
    &+[-a_1(1+V_*)^3+a_2(1+V_*)^2-a_3(1+V_*)]\Big]C(V)\notag\\
   = &\,  \sum_{\ell =0}^{\infty}(c^3)_{\ell}v^{\ell+1}+(1+V_*+mz)\sum_{\ell =0}^{\infty}(c^3)_{\ell}v^{\ell}-a_1\sum_{\ell =0}^{\infty}c_{\ell}v^{\ell+3}+\Big[-3a_1(1+V_*)+a_2\Big]\sum_{\ell =0}^{\infty}c_{\ell}v^{\ell+2}\label{eq:(1+V)Fexp}\\
   &+\Big[-3a_1(1+V_*)^2+2a_2(1+V_*)-a_3\Big]\sum_{\ell =0}^{\infty}c_{\ell}v^{\ell+1}+\Big[-a_1(1+V_*)^3+a_2(1+V_*)^2-a_3(1+V_*)\Big]\sum_{\ell =0}^{\infty}c_{\ell}v^{\ell}.\notag
\end{align}
Expanding also the second term, we obtain
\begin{align}
  &(1+v+V_*)C'(V)G(V,C) \notag\\
  =&\,(1+v+V_*)C'(V)\bigg\{C^2(V)\Big[(m+1)(v+V_*)+2mz\Big]-(v+V_*)(1+v+V_*)(\lambda+v+V_*)\bigg\} \notag\\
  =& \,(m+1)v^2C'(V)C^2(V)+\Big[(m+1)(1+2V_*)+2mz\Big]vC'(V)C^2(V)+(1+V_*)\Big[(m+1)V_*+2mz\Big]C'(V)C^2(V)\notag\\
  &-\Big[v^4+(\lambda +2+4V_*)v^3
    +[6V_*^2+(3\lambda + 6)V_*+2\lambda +1]v^2+[4V_*^3+(3\lambda+6)V_*^2+(4\lambda+2)V_*+\lambda]v\notag\\
    &+V_*(1+V_*)^2(\lambda+V_*)\bigg]C'(V)\notag\\
  =& \,\frac{m+1}{3}\sum_{\ell =1}^{\infty}\ell(c^3)_{\ell}v^{\ell+1}+\frac{(m+1)(1+2V_*)+2mz}{3}\sum_{\ell =1}^{\infty}\ell(c^3)_{\ell}v^{\ell} + \frac{(1+V_*)\Big[(m+1)V_*+2mz\Big]}{3}\sum_{\ell =1}^{\infty}\ell(c^3)_{\ell}v^{\ell-1}\notag\\
	&-\sum_{\ell =1}^{\infty}\ell c_{\ell}v^{\ell+3}-(\lambda +2+4V_*)\sum_{1}^{\infty}\ell c_{\ell}v^{\ell+2}-\Big[6V_*^2+(3\lambda + 6)V_*+2\lambda +1\Big]\sum_{\ell =1}^{\infty}\ell c_{\ell}v^{\ell+1}\label{eq:(1+V)C'Gexp}\\
	&-\Big[4V_*^3+(3\lambda+6)V_*^2+(4\lambda+2)V_*+\lambda\Big]\sum_{\ell =1}^{\infty}\ell c_{\ell}v^{\ell}-V_*(1+V_*)^2(\lambda+V_*)\sum_{\ell =1}^{\infty}\ell c_{\ell}v^{\ell-1}.\notag
\end{align}
We now proceed to study the difference of \eqref{eq:(1+V)Fexp} and \eqref{eq:(1+V)C'Gexp} and to group terms at each order in $v$ to simplify the resulting identity for $\mathcal{C}_\ell$. First, at order zero, we have
\begin{align}
	\mathcal{C}_0=&\,(1+V_*+mz)c_0^3+\Big[-a_1(1+V_*)^3+a_2(1+V_*)^2-a_3(1+V_*)\Big]c_0-\frac{(1+V_*)\Big[(m+1)V_*+2mz\Big]}{3}(3c_0^2c_1)\notag\\
	&+V_*(1+V_*)^2(\lambda+V_*)c_1\notag\\
  =&\,\Big[c_0^2(1+V_*+mz)-a_1(1+V_*)^3+a_2(1+V_*)^2-a_3(1+V_*)\Big]c_0\notag\\
  &-\Big[c_0^2[(m+1)V_*+2mz]-V_*(1+V_*)(\lambda+V_*)\Big]\Big](1+V_*)c_1\notag\\
  =&\,F(V_*,C_*)(1+V_*)c_0-G(V_*,C_*)(1+V_*)c_1=0.\label{eq:Taylororder0}
\end{align}
Next, the first order coefficient in $v$ simplifies as 
\begin{align}
	\mathcal{C}_1=&\,c_0^3+(1+V_*+mz)3c_0^2c_1+\Big[-3a_1(1+V_*)^2+2a_2(1+V_*)-a_3\Big]c_0+\Big[-a_1(1+V_*)^3+a_2(1+V_*)^2-a_3(1+V_*)\Big]c_1\notag\\
  &-\frac{(m+1)(1+2V_*)+2mz}{3}3c_0^2c_1-\frac{(1+V_*)\Big[(m+1)V_*+2mz\Big]}{3}2(3c_0^2c_2+3c_0c_1^2)\label{eq:Taylororder1i}\\
  &+\Big[4V_*^3+(3\lambda+6)V_*^2+(4\lambda+2)V_*+\lambda\Big]c_1+V_*(1+V_*)^2(\lambda+V_*)2c_2.\notag
\end{align}
In order to simplify this identity, we recall that as $G(V_*,C_*)=0$, we have $c_0^2((m+1)+2mz)=V_*(1+V_*)(\la+V_*)$ and thus, recalling also \eqref{all partials}, we have the auxiliary identity
\begin{equation}\label{eq:Taylorsimpleaux}
-[(m+1)(1+2V_*)+2mz]c_0^2+\Big[4V_*^3+(3\lambda+6)V_*^2+(4\lambda+2)V_*+\lambda\Big]=-G_V(V_*,C_*)(1+V_*).
\end{equation}
Substituting this along with the other identities in \eqref{all partials} into \eqref{eq:Taylororder1i} and grouping terms, we find $\mathcal{C}_1$ simplifies to
\begin{align}
	\mathcal{C}_1=&\,2\Big\{V_*(1+V_*)^2(\lambda+V_*)-(1+V_*)\Big[(m+1)V_*+2mz\Big]c_0^2\Big\}c_2+\Big\{-2(1+V_*)\Big[(m+1)V_*+2mz\Big]c_0\Big\}c_1^2\notag\\
	&+\Big\{3(1+V_*+mz)c_0^2+\Big[-a_1(1+V_*)^3+a_2(1+V_*)^2-a_3(1+V_*)\Big]-[(m+1)(1+2V_*)+2mz]c_0^2\notag\\
	&\quad+\Big[4V_*^3+(3\lambda+6)V_*^2+(4\lambda+2)V_*+\lambda\Big]\Big\}c_1+c_0^3+\Big[-3a_1(1+V_*)^2+2a_2(1+V_*)-a_3\Big]c_0\notag\\
	=&-2G(V_*,C_*)(1+V_*)c_2-G_C(V_*,C_*)(1+V_*)c_1^2+[F_C(V_*,C_*)-G_V(V_*,C_*)](1+V_*)c_1+F_V(V_*,C_*)c_0\notag\\
	=&-G_C(V_*,C_*)(1+V_*)c_1^2+[F_C(V_*,C_*)-G_V(V_*,C_*)](1+V_*)c_1+F_V(V_*,C_*)c_0.\label{eq:Taylororder1}
\end{align}
Finally, we group coefficients at order $\ell\geq 2$, recalling the convention that $c_k=0$ if $k<0$, to obtain the coefficient
\begin{align}
\mathcal{C}_\ell=&\,(c^3)_{\ell-1}+\big(1+V_*+mz)(c^3)_{\ell}-a_1c_{\ell-3}+\big(-3a_1(1+V_*)+a_2\big)c_{\ell-2}\notag\\
&+\big(-3a_1(1+V_*)^2+2a_2(1+V_*)-a_3\big)c_{\ell-1}+\big(-a_1(1+V_*)^3+a_2(1+V_*)^2-a_3(1+V_*)\big)c_\ell\notag\\
&-\frac{m+1}{3}(\ell-1)(c^3)_{\ell-1}-\frac{(m+1)(1+2V_*)+2mz}{3}\ell(c^3)_\ell -\frac{(1+V_*)\big((m+1)V_*+2mz\big)}{3}(\ell+1)(c^3)_{\ell+1}\notag\\
&+(\ell-3)c_{\ell-3}+(\la+2+4V_*)(\ell-2)c_{\ell-2}+\big(6V_*^2+(3\la+6)V_*+2\la+1\big)(\ell-1)c_{\ell-1}\label{eq:Taylororderelli}\\
&+\big(4V_*^3+(3\la+6)V_*^2+(4\la+2)V_*+\la\big)\ell c_\ell +V_*(1+V_*)^2(\la+V_*)(\ell+1)c_{\ell+1}.\notag
\end{align}
Recalling that  
\begin{align*}
(c^3)_\ell=&\,\sum_{i+j+k=\ell}c_ic_jc_k=3c_0^2c_\ell +\sum_{\substack{i+j+k=\ell\\ i,j,k\leq \ell-1}}c_ic_jc_k\\
(c^3)_{\ell+1}=&\,\sum_{i+j+k=\ell+1}c_ic_jc_k=3c_0^2c_{\ell+1}+6c_0c_1c_{\ell} +\sum_{\substack{i+j+k=\ell+1\\ i,j,k\leq \ell-1}}c_ic_jc_k,\\
\end{align*}
we isolate the highest order terms in $c_\ell$ and $c_{\ell+1}$ from \eqref{eq:Taylororderelli} as
\begin{align*}
&c_{\ell+1}(\ell+1)\Big(3c_0^2\big(-\frac{(1+V_*)\big((m+1)V_*+2mz\big)}{3}+V_*(1+V_*)^2(\la+V_*)\Big)=-c_{\ell+1}(\ell+1)(V_*+1)G(V_*,C_*)=0,\\
&c_\ell\Big(3(1+V_*+mz)c_0^2-a_1(1+V_*)^3+a_2(1+V_*)^2-a_3(1+V_*)-\big((m+1)(1+2V_*)+2mz)\ell c_0^2\\
&\quad-2(1+V_*)\big((m+1)V_*+2mz\big)(\ell+1)c_0c_1+\big(4V_*^3+(3\la+6)V_*^2+(4\la+2)V_*+\la\big)\ell\Big)\\
&=c_\ell(1+V_*)\Big(F_C-\ell G_V-(\ell+1)c_1 G_C\Big)\\
&=A_\ell c_\ell,
\end{align*}
where we have again applied \eqref{all partials} and \eqref{eq:Taylorsimpleaux} and where $A_\ell$ is as defined in Lemma \ref{lemma:Taylorcoeffs}. Thus, substituting these into \eqref{eq:Taylororderelli} and grouping terms by order of $c_k$, we have obtained that
\begin{align}
\mathcal{C}_\ell=&\,A_\ell c_\ell - \frac{(1+V_*)\Big[(m+1)V_*+2mz\Big]}{3}(\ell +1)\sum_{\substack{i+j+k = \ell+1\\ i,j,k\leq \ell-1}}c_ic_jc_k\notag\\
&+\Big[(1+V_*+mz)-\frac{(m+1)(1+2V_*)+2mz}{3}\ell\Big]\sum_{\substack{i+j+k = \ell\\ i,j,k\leq \ell-1}}c_ic_jc_k\notag\\
  &+\Big[1-\frac{m+1}{3}(\ell -1)\Big]\sum_{\substack{i+j+k = \ell-1\\ i,j,k\geq 0}}c_ic_jc_k\notag\\
  &+\Big[\big[6V_*^2+(3\lambda + 6)V_*+2\lambda +1\big](\ell -1)-3a_1(1+V_*)^2+2a_2(1+V_*)-a_3\Big]c_{\ell-1}\label{eq:Taylororderell}\\
  &+\Big[(\lambda +2+4V_*)(\ell -2)-3a_1(1+V_*)+a_2\Big]c_{\ell-2}+\Big[\ell-3-a_1\Big]c_{\ell-3},\notag
\end{align}
which, recalling the definition of $B_\ell$ from Lemma \ref{lemma:Taylorcoeffs}, concludes the proof.\qed

\end{allowdisplaybreaks}

\section{$z_m<z_g$ for $\gamma\in(1,2]$}\label{zm<zg}
As the expression for $z_g$ depends on $m$, we will prove the inequality first in the case $m=1$ and then for $m=2$. Recall first that $z_m$ is defined by \eqref{zm},
\begin{align*}
	z_m = \cfrac{(\gamma-1)}{(2\gamma-1)(\gamma+1)} = \cfrac{(\gamma-1)}{2\gamma^2+\gamma-1}.
\end{align*}
On the other hand, for $z_g$,
when $m=1$, we have
\begin{align*}
	z_g = \frac{\gamma-1}{\gamma(\sqrt{\gamma^2+(\gamma-1)^2}+\gamma)}.
\end{align*}
Therefore, it is sufficient to check $\gamma^2+\gamma-1 > \gamma\sqrt{\gamma^2+(\gamma-1)^2}$ for $\gamma\in(1,2]$.
Since
\begin{align*}
	(\gamma^2+\gamma-1)^2-\gamma^2(\gamma^2+(\gamma-1)^2) = -(\gamma-1)(\gamma^2(\gamma-3)+1-\gamma)>0
\end{align*}
for any $\gamma\in(1,2]$, we conclude
\begin{align*}
	z_m < z_g
\end{align*}
for any $\gamma\in(1,2]$ as desired.

When $m=2$, we have
\begin{align*}
	 z_g=\frac{2(\gamma-1)}{\sqrt{(2\gamma^2-\gamma+1)^2 +2\gamma(\gamma-1)[4\gamma(\gamma-1)+\frac{8}{3}]}+(2\gamma^2-\gamma+1)}.
\end{align*}
Therefore, in order to show $z_g>z_m$, it is enough to check $2\gamma^2+3\gamma-3>\sqrt{(2\gamma^2-\gamma+1)^2 +2\gamma(\gamma-1)[4\gamma(\gamma-1)+\frac{8}{3}]}$. Since
\begin{align*}
	(2\gamma^2+3\gamma-3)^2 - (2\gamma^2-\gamma+1)^2 +2\gamma(\gamma-1)[4\gamma(\gamma-1)+\frac{8}{3}] = \frac{8}{3}(\gamma-1)(3-\gamma)(3\gamma^2-1)>0
\end{align*}
for any $\gamma\in(1,2]$, we again conclude the claimed inequality.

\section{Proof of \eqref{C'(V6)<C6/2V6}}\label{proofofC'(V6)<C6/2V6}

Let $\ga\in(1,2]$ and $z\in[z_g,z_M]$, where we recall that $z_g$ is defined as in \eqref{zg}. In this section, for notational convenience, we will use $G_C$, $G_V$, $F_C$, $F_V$, $R$ to represent their evaluations at $P_6=(V_6,C_6)$ where $G_C$, $G_V$, $F_C$, $F_V$, $R$ are given in \eqref{G_C}--\eqref{F_V} and \eqref{R} respectively. 

Recall from \eqref{c1} that the slope of the curve $C=C(V)$ at $P_6=(V_6,C_6)$ is given by 
\begin{align*}
	\frac{d C}{d V}|_{V=V_6,C=C_6} =c_1 = \frac{F_C-G_V+ R}{2G_C}.
\end{align*}
Thus  \eqref{C'(V6)<C6/2V6} holds if and only if 
\begin{align*}
	&\frac{F_C-G_V+ R}{2G_C} - \frac{1+V_6}{2V_6}\\
	&=\frac{V_6F_C-V_6G_V+ V_6R-(1+V_6)G_C}{2V_6G_C}<0.
\end{align*}
Note that $G_C<0$ by \eqref{G_C} and hence $V_6G_C>0$.  Therefore, showing \eqref{C'(V6)<C6/2V6} is equivalent to proving that the numerator is negative for each $\gamma\in (1,2]$ and $z\in[z_g,z_M]$. 
Using \eqref{R}, 
it is then sufficient to show 
\begin{align}
	&V_6^2[(F_C-G_V)^2+4F_VG_C]-[V_6F_C-V_6G_V-(1+V_6)G_C]^2>0,
\end{align} which is equivalent to 
\begin{align}	
	& \mathcal Q:= 4V_6^2F_VG_C-(1+V_6)^2G_C^2+2(1+V_6)V_6(F_C-G_V)G_C>0 \label{inequality of C'(V6)<(sqrt(kV))}.
\end{align}
The rest of this section is devoted to the proof of \eqref{inequality of C'(V6)<(sqrt(kV))}. We  first rewrite $ \mathcal Q$ by using various identities satisfied by $V_6, C_6, F_C, F_V, G_C, G_V$: 
\begin{align*}
	\mathcal Q=&\,G_C\Big[2V_6^2(m(\gamma-1)wC_6-2F_C)-C_6^2G_C+2C_6V_6(F_C+mwC_6+G_C)\Big]\\
	=&\,G_CV_6C_6\Big[-2V_6^2+2m[\gamma w+(\gamma-2)z]V_6+2m([2-\gamma]z+w)+2\Big]
\end{align*}
where we have used  $2F_V= m(\gamma-1) wC_6- 2F_C $ and  $-G_V= mwC_6 + G_C$ from \eqref{F_C+F_V=constant(G_C+G_V)} and \eqref{G_C+G_V=C_6} as well as $C_6= 1+V_6$ in the first line and used \eqref{F_C} and \eqref{G_C} in the second line. Recalling the formula for $V_6$ in \eqref{P6} and using $V_6^2= ((\gamma-2)z-1) V_6- 2z$, we next rearrange the bracket as a linear function in $w$ whose coefficients are polynomials in $\gamma,z$ so that 
\begin{align*}
	\mathcal Q	=& \,G_CV_6C_6\Big[A(\gamma,z,m)w+B(\gamma,z,m)\Big]
\end{align*}
where \begin{align*}
	A(\gamma,z,m) &= 2m-1-m\gamma+(\gamma-2+2m-3m\gamma+m\gamma^2)z,\\
	B(\gamma,z,m) &=1-m\gamma+(6m+(2+m)\gamma+2m\gamma^2)z+(4m-4+4(1-2m)\gamma+(5m-1)\gamma^2-m\gamma^3)z^2.
\end{align*}
Since $G_CV_6C_6 >0$, our aim is to show $A(\gamma,z,m)w+B(\gamma,z,m) >0$ for each $\gamma\in (1,2]$,  $z\in[z_g,z_M]$ and $m=1, 2$. We will treat $m=1$ and $m=2$ separately. 

\

\underline{Case 1: $m=1$.} When $m=1$, we have 
\begin{align*}
	A(\gamma,z,1)&=1-\gamma+\gamma(\gamma-2)z<0, \\
	B(\gamma,z,1)&=1-\gamma+(2\gamma^2+3\gamma+6)z-\gamma(\gamma-2)^2z^2
\end{align*}
for all $\gamma\in(1,2]$. On the other hand, for any fixed $\gamma\in(1,2]$, 
\begin{align*}
	&\frac{d}{d z}(A(\gamma,z,1)w+B(\gamma,z,1)) \\
	&= 2\gamma^2+3\gamma+6-2\gamma(\gamma-2)^2z+A(\gamma,z,1)\frac{d w}{dz}+\gamma(\gamma-2) w\\
	&=\gamma(2\gamma+(\gamma-2)w)+2\gamma(1-(\gamma-2)^2z)+6+\gamma-A(\gamma,z,1)\frac{(\gamma+2)-(\gamma-2)^2z}{w}\\
	&>0,
\end{align*}
where the last inequality follows from $-1<\gamma-2\leq 0$, $0\le w<1$ and $z<1$ for any $\gamma\in(1,2]$ and $z\in[z_m,z_M]$. Here we extend the domain of $z$ into $[z_m, z_M]$ to facilitate computations. Since $A(\gamma,z,1)w+B(\gamma,z,1)$ is increasing in $z$, 
it is then enough to check that  $A(\gamma,z_m,1)w(z_m)+B(\gamma,z_m,1) >0$. 
From \eqref{zm}, we obtain the following relation between $w(z_m)$ and $z_m$: 
\begin{align}\label{w(zm)}
	V_6(\gamma,z_m) = \frac{-2}{\gamma+1} \ \Longleftrightarrow \ \frac{-1+(\gamma-2)z_m-w(z_m)}{2} = \frac{-2}{\gamma+1} \ \Longleftrightarrow \ w(z_m) = \frac{3-\gamma}{\gamma+1}+(\gamma-2)z_m.
\end{align}
Hence, using \eqref{zm} again we get for any $\gamma\in(1,2]$, 
\begin{align*}
	&A(\gamma,z_m,1)w(z_m)+B(\gamma,z_m,1) \\
	&= (1-\gamma+\gamma(\gamma-2)z_m)(\frac{3-\gamma}{\gamma+1}+(\gamma-2)z_m)+1-\gamma+(2\gamma^2+3\gamma+6)z_m-\gamma(\gamma-2)^2z_m^2\\
	&=\frac{4(\gamma-1)(\gamma^2+2)}{(\gamma+1)^2(2\gamma-1)}\\
	&>0.
\end{align*}
Therefore, we deduce that 
$A(\gamma,z,1)w+B(\gamma,z,1)>0$ for any $\gamma\in(1,2]$ and $z\in[z_m,z_M]$, which in turn implies  \eqref{inequality of C'(V6)<(sqrt(kV))}.

\

\underline{Case 2: $m=2$.} When $m=2$, the sign of $A(\gamma,z,2) = 3-2\gamma+(2-5\gamma+2\gamma^2)z$ changes for $\gamma\in(1,2]$ and $z\in[z_g,z_M]$ and the argument for $m=1$ is not applicable. We will employ another approach. 
We first decompose $\mathcal Q$ in \eqref{inequality of C'(V6)<(sqrt(kV))} into two parts
\[
\mathcal Q =: I + 2 \, II 
\]
where 
\begin{align*}
	I &:=2V_6^2F_VG_C-(1+V_6)^2G_C^2, \\
	II& := V_6^2F_VG_C+(1+V_6)V_6(F_C-G_V)G_C.
\end{align*}
We claim that $I$ and $II$ are both positive. For $I$, we first rewrite it by again using the identities \eqref{F_C}, \eqref{G_C}, \eqref{F_C+F_V=constant(G_C+G_V)}, and \eqref{G_C+G_V=C_6}, as well as $C_6= 1+V_6$ to get 
\begin{align*}
	I &= G_C[2V_6^2F_V-(1+V_6)^2G_C]\\
	&=G_C\Big[V_6^2(2(\gamma-1)wC_6-4C_6(C_6+2z))-2V_6C_6^2(C_6+2\gamma z)\Big]\\
	&=G_C\Big[2(\gamma-1)wV_6^2C_6-4V_6^2C_6^2-8zV_6^2C_6-2V_6C_6^3-4\gamma zV_6C_6^2\Big]\\
	&=G_CV_6C_6\Big[2(-V_6-C_6)C_6 +4z(-2V_6-\gamma C_6) + (2(\gamma-1)w-2C_6)V_6 \Big].
\end{align*}
The goal is to show the bracket is positive. For any $\gamma\in(1,2]$ and $z\in[z_g,z_M]$, by Lemma \ref{Lemma V_6/8 monoton}, we have $V_6 \leq V_6(z_M(2)) = -\frac{1}{2}$. Thus, $-V_6-C_6\geq 0$ since $1+V_6 = C_6$. Next, we  show that the remainder of the bracket is also positive. By using $C_6 = 1+V_6$, we have 
\begin{align*}
	4z(-2V_6-\gamma C_6) + (2(\gamma-1)w-2C_6)V_6 = -2V_6^2-(2+(8+4\gamma)z-2(\gamma-1)w)V_6-4\gamma z =:p_1(V_6).
\end{align*}
We observe that $p_1(V_6)$ is a quadratic polynomial of $V_6$. Since the coefficient of $V_6^2$ is negative and $-1<V_6\leq -\frac{1}{2}$ for $\gamma\in(1,2]$ and $z\in[z_g,z_M]$, to show $p_1(V_6)>0$, it is sufficient to check that $p_1(-1)>0$ and $p_1(-\frac{1}{2})>0$. Note that 
\begin{align*}
	p_1(-1)& = -2+2+(8+4\gamma)z-2(\gamma-1)w-4\gamma z = 8z-2(\gamma-1)w\\
	& > 8z_m-2(\gamma-1)w(z_m)
\end{align*}
for any $\gamma\in(1,2]$ and $z\in[z_g,z_M]$ because $\frac{dw(z)}{dz}<0$. By \eqref{zm} and \eqref{w(zm)}, we deduce that 
\begin{align*}
	p_1(-1)>8z_m-2(\gamma-1)w(z_m) &= \frac{8(\gamma-1)}{(2\gamma-1)(\gamma+1)} - 2(\gamma-1)[\frac{3-\gamma}{\gamma+1}+\frac{(\gamma-2)(\gamma-1)}{(2\gamma-1)(\gamma+1)}]\\
	&= \frac{2(\gamma-1)(\gamma^2-4\gamma+5)}{(2\gamma-1)(\gamma+1)}>0.
\end{align*}
For $p_1(-\frac{1}{2})$, observe that 
\begin{align*}
	p_1(-\frac{1}{2})&= -\frac{1}{2}+1+(4+2\gamma)z-(\gamma-1)w-4\gamma z = \frac{1}{2}+(4-2\gamma)z-(\gamma-1)w \\
	&> \frac{1}{2}+(4-2\gamma)z_m-(\gamma-1)w(z_m)
\end{align*}
for any $\gamma\in(1,2]$ and $z\in[z_g,z_M]$ because $\frac{dw(z)}{dz}<0$. By \eqref{zm} and \eqref{w(zm)}, we have 
\begin{align*}
	p_1(-\frac12)>\frac{1}{2}+(4-2\gamma)z_m-(\gamma-1)w(z_m) &= \frac{1}{2} + \frac{(4-2\gamma)(\gamma-1)}{(2\gamma-1)(\gamma+1)} -(\gamma-1)[\frac{3-\gamma}{\gamma+1}+\frac{(\gamma-2)(\gamma-1)}{(2\gamma-1)(\gamma+1)}]\\
	&= \frac{2\gamma^3-12\gamma^2+23\gamma-11}{2(2\gamma-1)(\gamma+1)}>0
\end{align*}
where the positive sign is shown in Proposition \ref{poly: App C-1}. Therefore, we conclude that $I>0$ since $G_C V_6 C_6>0$. 

For $II$, by using \eqref{G_C}, \eqref{F_C}, \eqref{F_C+F_V=constant(G_C+G_V)} and \eqref{G_C+G_V=C_6}, we have the following: 
\begin{align*}
	II &= G_C V_6 \Big [V_6(F_V+F_C)+F_C+C_6(2wC_6+G_C)\Big ]\\
	&= G_C V_6 \Big [\frac{2(\gamma-1)w}{2}V_6C_6+2C_6(C_6+2z)+C_6(2wC_6+2V_6(C_6+2\gamma z))\Big ]\\
	&= G_C V_6 C_6 \Big [2w(\frac{\gamma-1}{2}V_6+C_6)+4z(1+\gamma V_6)+2(1+V_6)^2\Big]\\
	&=  G_C V_6 C_6 \Big [\tilde{A}(\gamma,z)w+\tilde{B}(\gamma,z)\Big],
\end{align*}
where
\begin{align*}
	\tilde{A}(\gamma,z) &= \frac{1-\gamma}{2}+\frac{\gamma^2-7\gamma+2}{2}z<0,\\
	\tilde{B}(\gamma,z) &=-\frac{\gamma-1}{2}+(\gamma^2+\gamma+2)z+\frac{(2-\gamma)(\gamma^2-7\gamma+2)}{2}z^2
\end{align*}
for $\gamma\in(1,2]$.
We will first show $\tilde{B}(\gamma,z)>0$ for any $\gamma\in(1,2]$ and $z\in[z_m,z_M]$. When $\gamma=2$, $\tilde{B}(2,z)= - \frac12 + 8 z \ge  - \frac12 + 8 z_m =  - \frac12 + \frac{8}{15} >0$. 
For any fixed $\gamma\in(1,2)$, $\tilde{B}(\gamma,z)$ is a quadratic polynomial of $z$. Since $\frac{(2-\gamma)(\gamma^2-7\gamma+2)}{2}<0$, $\tilde{B}(\gamma,z)$ has global maximum at $z = \frac{(\gamma^2+\gamma+2)}{(2-\gamma)(\gamma^2-7\gamma+2)}>\frac{1}{2}$ since $2(\gamma^2+\gamma+2)>\gamma^2-7\gamma+2$ and $z\leq z_M =\frac{1}{\gamma+2+2\sqrt{2\gamma}}<\frac{1}{2}$ for any $\gamma\in(1,2)$. Hence, to show $\tilde{B}(\gamma,z)>0$, it is sufficient to check the sign of $\tilde{B}(\gamma,z_m)$ for any $\gamma\in(1,2)$. Direct computations show that 
\begin{align*}
	\tilde{B}(\gamma,z_m) &= -\frac{\gamma-1}{2}+\frac{(\gamma^2+\gamma+2)(\gamma-1)}{(2\gamma-1)(\gamma+1)}+\frac{(2-\gamma)(\gamma^2-7\gamma+2)(\gamma-1)^2}{2(2\gamma-1)^2(\gamma+1)^2}\\
	&=\frac{(\gamma-1)(-\gamma^4+12\gamma^3-14\gamma^2+24\gamma-9)}{2(2\gamma-1)(\gamma+1)}>0,
\end{align*}
where the positive sign is verified in Proposition \ref{poly: App C-2}. 
Now since $\tilde{B}(\gamma,z)>0$, to show $\tilde{A}(\gamma,z)w+\tilde{B}(\gamma,z)>0$, it is sufficient to check the sign of $\tilde{B}^2(\gamma,z)-\tilde{A}^2(\gamma,z)w^2>0$. By direct computations, 
\begin{align*}
	&\tilde{B}^2(\gamma,z)-\tilde{A}^2(\gamma,z)w^2\\ 
	&=2z\Big[-2\gamma^2+\gamma+1+(4\gamma^3+4\gamma^2-6\gamma+4)z+(-2\gamma^4+13\gamma^3+5\gamma^2-16\gamma+4)z^2\Big]\\
	&=2z\Big[-2\gamma^2+\gamma+1+2(\gamma+2)(2\gamma^2-2\gamma+1)z-(\gamma^2 - 7 \gamma + 2) (2 \gamma^2 + \gamma - 2)z^2\Big]=:2zp_2(z).
\end{align*}
Clearly $p_2(z)$ is a quadratic polynomial in $z$.  We notice that $-(\gamma^2 - 7 \gamma + 2) (2 \gamma^2 + \gamma - 2)>0$, and $\tilde{B}(\gamma,z)$ has global minimum  
at $z = \frac{(\gamma+2)(2\gamma^2-2\gamma+1)}{(\gamma^2 - 7 \gamma + 2) (2 \gamma^2 + \gamma - 2)}<0$ for any $\gamma\in(1,2]$. Hence, in order to prove $p_2(z)>0$ on $[z_g,z_M]$, it is enough to check the sign of $p_2(z_g)$ for each $\gamma\in(1,2]$. From \eqref{zgm2}, we write $z_g$ as 
\begin{align*}
	z_g=\frac{2(\gamma-1)}{\sqrt{(2\gamma^2-\gamma+1)^2 +2\gamma(\gamma-1)[4\gamma(\gamma-1)+\frac{8}{3}]}+(2\gamma^2-\gamma+1)}=:\frac{2(\gamma-1)}{q(\gamma)+(2\gamma^2-\gamma+1)}.
\end{align*}
We then compute $p_2(z_g)[q(\gamma)+(2\gamma^2-\gamma+1)]^2$ to obtain 
\begin{align*}
	&p_2(z_g)[q(\gamma)+(2\gamma^2-\gamma+1)]^2 \\
	=& (-2\gamma^2+\gamma+1)[q(\gamma)+(2\gamma^2-\gamma+1)]^2+4(\gamma+2)(2\gamma^2-2\gamma+1)(\gamma-1)[q(\gamma)+(2\gamma^2-\gamma+1)]\\
	&-4(\gamma-1)^2(\gamma^2 - 7 \gamma + 2) (2 \gamma^2 + \gamma - 2)\\
	=&(-2\gamma^2+\gamma+1)[2(2\gamma^2-\gamma+1)^2 +2\gamma(\gamma-1)[4\gamma(\gamma-1)+\frac{8}{3}]]+4(\gamma+2)(2\gamma^2-2\gamma+1)(\gamma-1)(2\gamma^2-\gamma+1)\\
	&-4(\gamma-1)^2(\gamma^2 - 7 \gamma + 2) (2 \gamma^2 + \gamma - 2)+2q(\gamma)[(-2\gamma^2+\gamma+1)(2\gamma^2-\gamma+1)+2(\gamma+2)(2\gamma^2-2\gamma+1)(\gamma-1)]\\
	=&\frac{2(\gamma-1)^2(-36\gamma^4+114\gamma^3-4\gamma^2-83\gamma+15)}{3}+2(\gamma-1)^2(4\gamma-3)q(\gamma)>0,
\end{align*}
where the positive sign of the quartic polynomial  is shown in Proposition \ref{poly: App C-3}. Therefore, we deduce that $\tilde{A}(\gamma,z)w+\tilde{B}(\gamma,z)>0$ and hence $II>0$. This completes the proof of \eqref{inequality of C'(V6)<(sqrt(kV))} for $m=2$.

\section{Proof of \eqref{ineq of c1(z1)} and \eqref{ineq of c1(z2)} }
In this section, we consider specific values of the parameters: $z=z_1(\gamma)$ or $z_2(\gamma)$, and $\kappa=1$ or $\frac{3}{2}$. For notational convenience, we use $G_C$, $G_V$, $F_C$, $F_V$, $R$ to represent their evaluations at $P_8=(V_8,C_8)$ where $G_C$, $G_V$, $F_C$, $F_V$, $R$ are given in \eqref{G_C}--\eqref{F_V} and \eqref{R} respectively. Recalling \eqref{c1} 
, the two inequalities \eqref{ineq of c1(z1)} and \eqref{ineq of c1(z2)} can be written as
\begin{align*}
	\frac{F_C-G_V+R}{2G_C} < -\frac{1}{2}\sqrt{\frac{\kappa}{-V_8}} = -\frac{\kappa}{2C_8},
\end{align*}
where $\kappa=1$ corresponds to \eqref{ineq of c1(z1)} and $\kappa=\frac32$ is equivalent to \eqref{ineq of c1(z2)}.
From \eqref{G_C}, we see that $G_C<0$. Moreover, $R>0$, and hence it is equivalent to prove that
$$ R^2 > (-G_C\sqrt{\frac{\kappa}{-V_8}}-F_C+G_V)^2.$$
Expanding $R$ using the definition in \eqref{R}, we find that this is equivalent to
\begin{align}\label{equivalent ineq}
	-\frac{\kappa G_C}{V_8}+\frac{2\kappa }{C_8}(F_C-G_V)-4F_V>0.
\end{align} 
By using \eqref{F_C}, \eqref{F_C+F_V=constant(G_C+G_V)}, \eqref{G_C+G_V=C_8} and $w = 2C_8-1-(\gamma-2)z$ which is given by \eqref{P8}, we compute
\begin{align}
	4F_V &= 4(-\frac{\gamma-1}{2}mwC_8 - 2C_8(C_8+mz)) = -2m(\gamma-1)wC_8-8C_8^2-8mzC_8\notag\\
	&=-[4m(\gamma-1)+8] C_8^2+2m(\gamma-1)C_8+2m(\gamma^2-2\gamma-3)zC_8.\label{LHS}
\end{align}
Thanks to  \eqref{G_C}, \eqref{F_C}, \eqref{G_C+G_V=C_8} , $C_8^2 = -\kappa V_8$,  $w = 2C_8-1-(\gamma-2)z$ and $C_8=1+V_8$, we obtain
\begin{align}
-\frac{\kappa G_C}{V_8}+\frac{2\kappa}{C_8}(F_C-G_V) &= 2\kappa\left( -m\gamma z -1+V_8+2C_8+2mz-mw \right) - 4m\gamma z C_8\notag\\
&=2\kappa\left[ m-2+(3-2m)C_8\right] - 4m\gamma z C_8. \label{RHS}
\end{align}
Now, we are ready to show \eqref{ineq of c1(z1)} and \eqref{ineq of c1(z2)} hold.
\begin{lemma}\label{Lemma: ineq of c1(z1)}
For any $\gamma\in(\gaSix, \gaOne]$, $z=z_1$ and $\kappa=1$, \eqref{equivalent ineq} holds,  and therefore so does \eqref{ineq of c1(z1)}.
\end{lemma}
\begin{proof}
To show that \eqref{ineq of c1(z1)} holds, by using \eqref{equivalent ineq},  it is enough to check the positivity of $-\frac{\kappa G_C}{V_8}+\frac{2\kappa}{C_8}(F_C-G_V)-4F_V$. 
When $m=1$, by using \eqref{LHS}, \eqref{RHS}, $V_8(z_1)=-C_8^2(z_1)$, we obtain that, for any $\gamma\in(\gaSix,\gaOne]$,
	\begin{align*}
	-\frac{ G_C}{V_8}+\frac{2}{C_8}(F_C-G_V)-4F_V &=  2C_8-2+4(\gamma+1) C_8^2 -2(\gamma-1)C_8+2(3-\gamma^2) z_1C_8\\
	&=-2\gamma(2 V_8 +C_8)+2[(3-\gamma^2)z_1C_8+1] \\
	&> (7-3\sqrt{5})\gaSix + 2[(3-\gaOne^2)z_1(\gaSix)C_8+1] >0,
	\end{align*}
	where we have used \eqref{z1} and \eqref{V8C8z1} in the last two inequalities.

	When $m=2$, by using \eqref{LHS}, \eqref{RHS} and $V_8(z_1)=-C_8^2(z_1)$, we have,  for any  $\gamma\in(\gaSix, \gaOne]$,
	\begin{align*}
		-\frac{ G_C}{V_8}+\frac{2}{C_8}(F_C-G_V)-4F_V &= 8\gamma C_8^2-(4\gamma-2)C_8+4(3-\gamma^2)z_1C_8\\
		&=-4\gamma(2V_8+C_8)+2C_8+4(3-\gamma^2) z_1 C_8\\
		&>2(7-3\sqrt{5})\gamma +2[2(3-\gaOne^2)z_1(\gaSix)+1]C_8>0,
	\end{align*}
	where we have used \eqref{z1} and \eqref{V8C8z1} in the last two inequalities. 
\end{proof}

\begin{lemma}\label{Lemma: ineq of c1(z2)}
For any $\gamma\in(\gaOne,3]$, $z=z_2$ and $\kappa=\frac{3}{2}$, \eqref{equivalent ineq} holds and therefore so does \eqref{ineq of c1(z2)}.
\end{lemma}
\begin{proof}
To show that \eqref{ineq of c1(z2)} holds, by using \eqref{equivalent ineq},  it is enough to check the positivity of $-\frac{\kappa G_C}{V_8}+\frac{2\kappa}{C_8}(F_C-G_V)-4F_V$. When $m=1$, by using \eqref{LHS}, \eqref{RHS} and $C_8^2 = -\frac{3}{2}V_8$, we have, for any $\gamma\in(\gaOne,3]$,
	\begin{align*}
	-\frac{3}{2}\frac{ G_C}{V_8}+\frac{3}{C_8}(F_C-G_V)-4F_V &= 3C_8-3+4(\gamma+1) C_8^2 -2(\gamma-1)C_8+2(3-\gamma^2)z_2C_8\\
	&=-2\gamma(3V_8+C_8)-V_8+2[(3-\gamma^2)z_2C_8+1]\\
	&>2(6-\sqrt33)\gamma-V_8+2(-6z_2(\gaOne)C_8+1)>0,
	\end{align*}
	where we have used \eqref{z2} and \eqref{C8z2} in the last two equalities.

	When $m=2$, by using \eqref{LHS}, \eqref{RHS} and $C_8^2 = -\frac{3}{2}V_8$, we obtain, for any $\gamma\in(\gaOne,3]$,
	\begin{align*}
		-\frac{3}{2}\frac{ G_C}{V_8}+\frac{3}{C_8}(F_C-G_V)-4F_V &= 8\gamma C_8^2-(4\gamma-1)C_8+4(3-\gamma^2)z_2C_8\\
		&=-4\gamma(3V_8+C_8)+[4(3-\gamma^2)z_2+1]C_8\\
		&>-4\gamma(3V_8+C_8)+[-24z_2(2)+1]C_8\\
		&=4(6-\sqrt33)\gamma+16\sqrt{33}-93>0,
	\end{align*}
	where we have used \eqref{z2} and \eqref{C8z2} in the last two inequalities.
\end{proof}

\section{Calculation of polynomials}\label{poly}
\begin{lemma}[\cite{Irving04}, Exercise 10.14]\label{cubic poly}
	Consider the cubic polynomial 
	\begin{align*}
		p(x)=ax^3+bx^2+cx+d
	\end{align*}
	where $a\neq 0$, $b$, $c$ and $d$ are all real numbers. Define the discriminant $\Delta$ of $p(x)$ as
	\begin{align}\label{discriminant of cubic poly}
		\Delta = b^2c^2-4ac^3-4b^3d-27a^2d^2+18abcd.
	\end{align}
	Then,
		\begin{enumerate}
		\item If $\Delta = 0$, then $p(x)$ has a multiple root and all its roots are real;
		\item If $\Delta >0$, then $p(x)$ has three distinct real roots;
		\item If $\Delta<0$, then $p(x)$ has one real root and two complex conjugate roots.
	\end{enumerate}

\end{lemma}

\begin{lemma}[\cite{Irving04}, Exercise 10.30 and Exercise 10.36]\label{quartic poly}
	Consider the quartic polynomial 
	\begin{align*}
		p(x)=ax^4+bx^3+cx^2+dx+e
	\end{align*}
	where $a\neq 0$, $b$, $c$, $d$ and $e$ are all real numbers. Define the discriminant $\Delta$ of $p(x)$ as 
	\begin{align}\label{discriminant of quartic poly}
		\Delta = & 18abcd^3+18b^3cde-80abc^2de-6ab^2d^2e+144a^2cd^2e+144ab^2ce^2-128a^2c^2e^2-192a^2bde^2+b^2c^2d^2\nonumber\\
		&-4b^3d^3-4c^3d^3-4b^2c^3e+16ac^4e-27a^2d^4-27b^4e^2+256a^3e^3.
	\end{align}
	Then,
	\begin{enumerate}
		\item If $\Delta = 0$, then $p(x)$ has a multiple root;
		\item If $\Delta >0$, then the roots of $p(x)$ are either all real or all complex;
		\item If $\Delta<0$, then $p(x)$ has two real roots and two complex conjugate roots.
	\end{enumerate}
\end{lemma}

\begin{lemma}[\cite{Sturm09}, Sturm's Theorem]\label{Sturm's theorem}
	Take any polynomial $p(x)$, and let $p_0(x), \ldots p_m(x)$ denote the Sturm chain corresponding to $p(x)$. Take any interval $(a, b)$ such that $p_i(a), p_i(b) \neq$ 0 , for any $i\in\{0,1,\dots,m\}$. For any constant $c$, let $\sigma(c)$ denote the number of changes in sign in the sequence $p_0(c), \ldots p_m(c)$. Then $p(x)$ has $\sigma(a)-\sigma(b)$ distinct roots in the interval $(a, b)$.
\end{lemma}

\begin{proposition}\label{poly: lemma 5.2-1}
	For any $\gamma\in(1,2]$, the cubic polynomial
	\begin{align*}
		p(\gamma)=-\gamma^3-7\gamma^2+\frac{106}{3}\gamma-36<0.
	\end{align*}
\end{proposition}
\begin{proof}
	Since the discriminant of $p(x)$ is
	\begin{align*}
		\Delta = \frac{-189956}{27} <0,
	\end{align*}
	$p(x)$ has only one real root by Lemma \ref{cubic poly}. Moreover, $p(-\infty)>0$ and $p(0)=-36<0$, which implies that the real root must be negative. Thus, $p(\gamma)<0$ for $\gamma\in(1,2]$.
\end{proof}

\begin{proposition}\label{poly: lemma 5.2-2}
	For any $\gamma\in(1,2]$, the quintic polynomial
	\begin{align*}
		p(\gamma)=81(\gamma-1)(\gamma+1)^4-8\gamma(3\gamma-1)^4<0.
	\end{align*}
\end{proposition}
\begin{proof}
By direct computations, the Sturm chain of $p(\gamma)$ is given by
\begin{align*}
	p_0(\gamma)&=-81 - 251 \gamma - 66 \gamma^2 - 270 \gamma^3 + 1107 \gamma^4 - 567 \gamma^5,\\
	p_1(\gamma)&=-251 - 132 \gamma - 810 \gamma^2 + 4428 \gamma^3 - 2835 \gamma^4,\\
	p_2(\gamma)&=\frac{52816}{525} + \frac{36944 \gamma}{175} + \frac{720 \gamma^2}{7} - \frac{41616 \gamma^3}{175},\\
	p_3(\gamma)&=-\frac{92164800}{83521}-\frac{126201600 \gamma}{83521}+\frac{162187200 \gamma^2}{83521},\\
	p_4(\gamma)&=-\frac{14651587904}{271832505}-\frac{5446905536 \gamma}{453054175},\\
	p_5(\gamma)&=-\frac{3876577982771200}{86724949081}.
\end{align*} 
Therefore, $\sigma(-\infty)=3$ and $\sigma(\infty)=2$. Thus, $p(\gamma)$ only has one real root. Since $p(-\infty)>0$ and $p(1)<0$, $p(\gamma)<0$ for any $\gamma\in(1,2]$.
\end{proof}

\begin{proposition}\label{poly: Lemma 6.2-1}
	For any $\gamma\in(\gaSix,\gaOne]$, the quartic polynomial
	\begin{align*}
		p(\gamma)=\frac{8(\gamma-2)^2\gamma^2}{625}+ \frac{2(1-2\gamma)\gamma(\gamma+1)}{25} +(1-2\gamma)(2-\gamma)<0.
	\end{align*}
\end{proposition}
\begin{proof}
	By \eqref{discriminant of quartic poly}, the discriminant of $p(\gamma)$ is
	\begin{align*}
		\Delta = -\frac{1586137650624}{3814697265625}<0.
	\end{align*}
	By Lemma \ref{quartic poly}, $p(x)$ has two real roots. Since $p(-\infty) >0$, $p(1)=-\frac{717}{625}<0$, $p(\frac{5}{2})=-\frac{39}{50}<0$, and $p(+\infty) >0$, $p(\gamma)<0$ for any $\gamma\in(\gaSix,\gaOne]$. 
\end{proof}

\begin{proposition}\label{poly: Lemma 6.2-2}
	For any $\gamma\in(\gaSix,\gaOne]$, the quadratic polynomial
	\begin{align}
		p(\gamma)=(3\sqrt{2}-4)\gamma+1-\sqrt{2}+\frac{2}{25}[-(2-\sqrt{2})\gamma^2+(3-2\sqrt{2})\gamma+2\sqrt{2}+2]>0.
	\end{align}
\end{proposition}
\begin{proof}
	Since $p(-\infty)<0$, $p(1)=\frac{58\sqrt{2}-77}{25}>0$, $p(3)=\frac{264\sqrt{2}-361}{25}>0$, and $p(\infty)<0$. So, $p(\gamma)>0$ for any $\gamma\in(\gaSix,\gaOne]$. 
\end{proof}

\begin{proposition}\label{poly: Lemma 6.2-3}
	For any $\gamma\in(\gaSix,\gaOne]$, the quadratic polynomial
	\begin{align}
		p(\gamma)=2(3\sqrt{2}-4)\gamma+3-3\sqrt{2}+\frac{4}{25}[-(2-\sqrt{2})\gamma^2+(3-2\sqrt{2})\gamma+2\sqrt{2}+2]>0.
	\end{align}
\end{proposition}
\begin{proof}
	Since $p(-\infty)<0$, $p(\frac{3}{2})=\frac{182 \sqrt{2}-253}{25}>0$, $p(3)=\frac{503\sqrt{2}-697}{25}>0$, and $p(\infty)<0$. So, $p(\gamma)>0$ for any $\gamma\in(\gaSix,\gaOne]$. 
\end{proof}

\begin{proposition}\label{poly: Lemma 6.3-1}
	For any $\gamma\in(\gaOne,3]$, the cubic polynomial
	\begin{align*}
		p(\gamma) = \gamma^3-17\gamma^2+40\gamma-14<0.
	\end{align*}
\end{proposition}
\begin{proof}
	Since $p(-\infty)<0$, $p(\infty)>0$, $p(1)=10>0$, $p(\gaOne)=11 \sqrt{2}-18<0$, $p(3) = -20<0$. Thus, $p(x)<0$ for any $\gamma\in(\gaOne,3]$.
\end{proof}

\begin{proposition}\label{poly: Lemma 6.3-2}
	For any $\gamma\in(\gaOne,3]$, the quartic polynomial
	\begin{align*}
		p(\gamma)=-\gamma^4+12\gamma^3-36\gamma^2+32\gamma+1>0.
	\end{align*}
\end{proposition}
\begin{proof}
	Since  the discriminant of $p(\gamma)$ is
	\begin{align*}
		\Delta = -495616<0,
	\end{align*}
	by Lemma \ref{quartic poly}, $p(x)$ has two real roots. Thus, $p(-\infty)<0$, $p(1) =8 >0$, $p(3)=16$, and $p(\infty)<0$  implies that $p(x)>0$ for any $\gamma\in(\gaOne,3]$.
\end{proof}

\begin{proposition}\label{poly: App C-1}
	For any $\gamma\in(1,2]$, the cubic polynomial
	\begin{align*}
		p(\gamma) = 2\gamma^3-12\gamma^2+23\gamma-11>0.
	\end{align*}
\end{proposition}
\begin{proof}
	Since  the discriminant of $p(\gamma)$ is
	\begin{align*}
		\Delta = -964<0,
	\end{align*}
	by Lemma \ref{cubic poly}, $p(x)$ has only one real root. Thus, as $p(-\infty)<0$ and $p(1)=2$, we must have that $p(x)>0$ for any $\gamma\in(1,2]$.
\end{proof}

\begin{proposition}\label{poly: App C-2}
	For any $\gamma\in(1,2]$, the quartic polynomial
	\begin{align*}
		p(\gamma) = -\gamma^4+12\gamma^3-14\gamma^2+24\gamma-9>0.
	\end{align*}
\end{proposition}
\begin{proof}
		Since  the discriminant of $p(\gamma)$ is
	\begin{align*}
		\Delta = -30235392<0,
	\end{align*}
	by Lemma \ref{quartic poly}, $p(x)$ has two real roots. Thus, as $p(-\infty)<0$, $p(\infty)<0$, $p(1) = 12>0$ and $p(2)=63>0$, we have that  $p(x)>0$ for any $\gamma\in(1,2]$.
\end{proof}

\begin{proposition}\label{poly: App C-3}
	For any $\gamma\in(1,2]$, the quartic polynomial
	\begin{align*}
		p(\gamma) = -36\gamma^4+114\gamma^3-4\gamma^2-83\gamma+15>0.
	\end{align*}
\end{proposition}
\begin{proof}
	Note that  the first derivative of $p(\gamma)$ is
	\begin{align*}
		\frac{d p(\gamma)}{d \gamma} = -144\gamma^3+342\gamma^2-8\gamma-83.
	\end{align*}
	Since $\frac{d p}{d \gamma}(-\infty)>0$, $\frac{d p}{d \gamma}(\infty)<0$, $\frac{d p}{d \gamma}(0)=-83<0$, $\frac{d p}{d \gamma}(1) = 107>0$, and $\frac{d p}{d \gamma}(2) = 117>0$, $\frac{d p(\gamma)}{d \gamma}>0$ for any $\gamma\in(1,2]$. Thus, $p(1)=6$ implies that $p(x)>0$ for any $\gamma\in(1,2]$.
\end{proof}

\end{appendices}

\end{document}